\documentclass[12pt]{article}
\usepackage{graphicx}
\usepackage{amssymb,amsmath,amsthm,secdot,bbm,mathrsfs}
\usepackage[T1]{fontenc}
\usepackage{authblk}
\usepackage{subfig}

\usepackage[
bookmarks=true,
bookmarksnumbered=true,
hypertexnames=false,
colorlinks=true, pdfstartview=FitV, linkcolor=blue, citecolor=blue,
urlcolor=blue]{hyperref}

\usepackage[nameinlink,capitalize]{cleveref}

\numberwithin{equation}{section}
\numberwithin{figure}{section}

\crefname{figure}{Figure}{figures}

\crefrangeformat{section}{Sections~#3#1#4--#5#2#6}
\crefrangeformat{exercise}{Exercises~#3#1#4--#5#2#6}

\newcommand{\dnew}[4]{[#1]_{\C^{#2}L_{#3}(#4)}^{\mathrm{cond}}}

\newcommand{\ms}[1]{\Delta^{shift}_{#1}}
\usepackage[shortlabels]{enumitem}

\usepackage{thmtools}

\newtheorem{theorem}{Theorem}[section]
\newtheorem{lemma}[theorem]{Lemma}
\newtheorem{proposition}[theorem]{Proposition}

\newtheorem{corollary}[theorem]{Corollary}

\theoremstyle{definition}
\newtheorem{remark}[theorem]{Remark}
\newtheorem{exercise}[theorem]{Exercise}
\crefname{exercise}{Exercise}{Exercises}

\newtheorem{definition}[theorem]{Definition}

\newtheorem{convention}[theorem]{Convention}

\newcommand{\eps}{\varepsilon}

\newcommand{\wt}{\widetilde}
\newcommand{\wh}{\widehat}

\renewcommand{\phi}{\varphi}

\newcommand{\hn}{\textbf{Hint}: }
\newcommand{\hns}{\textbf{Hints}: }

\def\E{\hskip.15ex\mathsf{E}\hskip.10ex}
\def\P{\mathsf{P}}

\renewcommand{\d}{\partial}

\newcommand{\A}{\mathcal{A}}
\newcommand{\B}{\mathcal{B}}
\newcommand{\C}{\mathcal{C}}

\newcommand{\EE}{\mathcal{E}}
\newcommand{\F}{\mathcal{F}}
\newcommand{\G}{\mathcal{G}}
\newcommand{\FF}{\mathbb{F}}
\newcommand{\I}{\mathbbm{1}}
\renewcommand{\Im}{\mathbb{I}}

\newcommand{\N}{\mathbb{N}}
\newcommand{\PP}{\mathcal{P}}
\newcommand{\R}{\mathbb{R}}
\renewcommand{\S}{\mathcal{S}}
\renewcommand{\t}{\mathsf{T}}
\newcommand{\V}{\mathcal{V}}
\newcommand{\W}{\mathbb{W}}
\newcommand{\Z}{\mathbb{Z}}

\newcommand{\nll}{\mathrel{\not{\mkern -7mu\ll}}}

\newcommand{\rrr}{\textrm{rate}}
\newcommand{\err}{\mathrm{Err}}

\newcommand{\pa}{p^{\alpha}}
\newcommand{\Pa}{P^{\alpha}}

\newcommand*{\0}{\mathbf{0}}

\newcommand{\la}{\langle}
\newcommand{\ra}{\rangle}

\newcommand{\nn}{\nonumber}

\newcommand{\evalat}[1]{\bigg\rvert_{#1}}

\DeclareMathOperator{\law}{Law}%

\DeclareMathOperator{\sign}{sign}

\newcommand{\lm}{L_m(\Omega)}

\begin{document}
	
\title{Lectures on  stochastic sewing with applications}

\author[1,2,3]{Oleg Butkovsky\thanks{Email: \texttt{oleg.butkovskiy@gmail.com}}}

\renewcommand\Affilfont{\small}
	
\affil[1]{Weierstrass Institute (WIAS), Berlin}
\affil[2]{Institut für Mathematik, Humboldt-Universität zu Berlin}
\affil[3]{Simons Laufer Mathematical Sciences Institute (MSRI), Berkeley}

\maketitle
\begin{abstract}
These are lecture notes for a mini-course on stochastic sewing, taught at the University of Edinburgh and Beijing Institute of Technology in Spring/Summer 2025. The aim is to introduce the reader to stochastic sewing techniques and to show how they can be successfully applied to study various problems in stochastic analysis, including: regularization by noise for stochastic differential equations driven by Brownian motion, a L\'evy process or fractional Brownian motion, 
the study of averaging operators, 
 and the analysis of numerical algorithms.
\end{abstract}

\tableofcontents

%
%
%
%
%
%
%
%
%
%
%
%
%

\section{Introduction}

The stochastic sewing lemma (SSL) is a tool invented by Khoa L\^e \cite{LeSSL} in 2018  for analyzing random perturbations of differential equations. 
It is a stochastic generalization of the classical sewing lemma (Gubinelli \cite{Gubi}, Feyel and De La Pradelle \cite{PressF}). Since then, this research area has  rapidly developed, and stochastic sewing (and its further extensions) has found numerous applications in stochastic PDEs, rough stochastic differential equations, rough path theory, stochastic numerics, the study of fast--slow systems, and related topics. The goal of these lecture notes is to provide the first systematic overview of stochastic sewing techniques and their applications, combining recent developments and highlighting the main ideas.

For the reader's convenience, we present the key techniques and methods in the simplest possible yet nontrivial setting. This way the technical calculations are greatly reduced, while the reader can still fully grasp the main ideas and arguments. Further extensions are left as exercises.

It turns out that the problems described above have the following common feature: they rely on accurate bounds for integrals of the form
\begin{equation}\label{type} 
\int f(\text{noise}(t) + \text{perturbation}(t))\,dt, \qquad
\int f(\text{noise}(t) + \text{perturbation}(t))\,d\,\text{noise}(t),
\end{equation}
where $f$ is a measurable, non-smooth function or even a Schwartz distribution, the ``noise'' is a given highly irregular process with a known law, and the “perturbation’’ is a process whose law is unknown but is more regular (in some sense) than the noise. Such bounds are sometimes called Krylov-type bounds.
Integrals of the first type arise in the analysis of stochastic equations with irregular drift coefficients, while those of the second type arise in the analysis of stochastic equations with irregular diffusion coefficients. The key point is that although $f$ itself may be quite rough, its integral along the noise is much more regular, thanks to the averaging effect of the noise. 

When the noise is a finite-dimensional Brownian motion, such bounds can be obtained using the so-called Veretennikov–Zvonkin transformation method, as beautifully explained in \cite{F11}. However, the method strongly relies on  It\^o's formula. Therefore, its application is very difficult or even impossible if the noise is a fractional Brownian motion (which for $H\neq1/2$ is neither a Markov process nor a semimartingale) or in the context of stochastic PDEs.

Stochastic sewing has emerged as a very efficient alternative to the Veretennikov–Zvonkin transformation method, as it does not rely on the Markov or semimartingale property of the noise and is applicable to a broad range of processes. 

Let us stress that the stochastic sewing lemma is a very flexible technique that can be adapted to the specific problem at hand. This will be the main guiding principle for these lecture notes: to show the reader that even when the original SSL is not enough to obtain a certain result, a suitable modification may work. We begin with SSL and apply it to an initial problem where it suffices. We then proceed to more challenging problems, and whenever the current set of tools is no longer sufficient, we introduce a new tool and use it together with the previous ones to overcome the next difficulty. In this way the collection of methods gradually grows, and we will also point out open problems for which the present set of tools still falls short.

Thus, we start with the strong well-posedness of SDEs driven by Brownian or fractional Brownian motion with H\"older drift (\cref{s:s3}), where the original SSL is sufficient. We then move to weak well-posedness (uniqueness in law) for SDEs with distributional drift, where SSL alone is no longer enough, and we introduce generalized coupling techniques (\cref{s:WU}) to handle this challenge. Next, we consider the rate of convergence of the Euler scheme for SDEs with irregular drift. When the noise is Brownian and the drift is Hölder continuous, SSL still applies directly (\cref{s:numbd}). We then study SDEs driven by a Lévy process: to establish well-posedness, we introduce the shifted stochastic sewing lemma (\cref{s:levywp}), and to analyze the rate of convergence of the Euler scheme, we additionally rely on the John-Nirenberg inequality (\cref{s:levynum}). Finally, we look at SDEs with Sobolev drifts, for which we introduce  the taming singularities lemma (\cref{s:levyex}).

Each chapter is accompanied by exercises. Some of them provide a path to the proof of classical results which we use in the lecture notes, while others focus on further applications of stochastic sewing.

\textbf{Comments on literature.} \cref{s:2} is based on \cite{Gubi,PressF,LeSSL}, \cref{s:s3} follows the approach of \cite{LeSSL}, \cref{s:WU} builds on \cite{BM24}, and \cref{s:5} draws on \cite{BDG,BDGLevy,Gerreg22,DGL23}.

\textbf{Further directions}. These lecture notes are necessarily limited in scope, and it is not possible to present all the interesting problems where stochastic sewing ideas have found effective applications. To give the reader at least a glimpse of this broader landscape, we briefly mention a few further directions here. 

First, one can use stochastic sewing to study well-posedness and qualitative properties of a variety of models. Examples include rough and Young differential equations \cite{MP22,matsuda2023pathwise}, stochastic McKean–Vlasov equations and interacting particle systems \cite{GHMmv23,GLM24}, and stochastic PDEs \cite{ABLMmw,ABLM,D24,djurdjevac2024higher}. Further applications to SDEs, such as the existence of a flow of solutions and Malliavin differentiability, can be found in \cite{GG22,HP20,BGallay}. Equations with multiplicative noise are treated in \cite{dareiotis2024regularisation}. Applications to ergodicity and long-term behavior are given in 
\cite{mayorcas2025ergodic,DHL2026}.

Second, \cite{Hairerli} uses the stochastic sewing lemma to analyze slow–fast systems in which the slow component is driven by a fractional Brownian motion.

Third, \cite{FHL,BCN,BFS24} introduce and study an entirely new class of equations, called \textit{Rough SDEs}, where stochastic sewing plays a crucial role.

Additional applications include the study of local times of stochastic processes \cite{BLM23,DLMP25} and the analysis of densities of solutions \cite{anzeletti2025density}, where stochastic sewing is combined with Romito’s lemma \cite{Romito}.

\textbf{Convention on constants}. Throughout the paper, $C$ denotes a positive constant whose value may change from line to line; its dependence is always specified in the corresponding statement.

\textbf{Acknowledgements}. I would like to thank Gon\c{c}alo Dos Reis and Zhenyao Sun for inviting me to give mini-courses in Edinburgh and Beijing, respectively. I am very grateful to Xicheng Zhang for his outstanding hospitality and many stimulating discussions. I also would like to very warmly thank Benjamin Robinson, Verena Schwarz and three anonymous referees for the time and care they put into reading these notes, spotting inaccuracies, and suggesting improvements, both big and small. The notes are much better thanks to their efforts.

 I also thank the participants of my lectures at the University of Edinburgh and the Beijing Institute of Technology for their insightful questions and interesting conversations. I  gratefully acknowledge financial support towards the Edinburgh mini-course on \emph{Stochastic Sewing lemma} (Spring 2025) by the Maxwell Institute for Mathematical Sciences (Edinburgh) and support by the UK Research and Innovation (UKRI) under the UK government’s Horizon Europe funding Guarantee [Project UKRI343]. The author acknowledges funding by the Deutsche Forschungsgemeinschaft (DFG, German Research Foundation) --- CRC/TRR 388 "Rough Analysis, Stochastic Dynamics and Related Fields" – Project ID 516748464, subproject B08.  These lecture notes are based on work supported by the National Science Foundation under Grant No. DMS-2424139, while I was in residence at the Simons Laufer Mathematical Sciences Institute in Berkeley, California, during Fall 2025.

\subsection{Example 1: Regularization by noise for SDEs}\label{sect:11}
It has been known for a long time that an ill-posed deterministic system can become well-posed when randomly perturbed. Indeed, 
consider a differential equation
\begin{equation}\label{ode}
dX_t=\frac{1}{1-\alpha} \sign(X_t)|X_t|^\alpha dt, \quad X_0=0,
\end{equation}
where $\alpha\in(0,1)$. It is easy to see that it has infinitely many solutions: for example one can fix any $t_0>0$ and take $X_t=\I_{t\ge t_0}(t-t_0)^{\frac{1}{1-\alpha}}$, $t\ge0$. It is also clear that the equation
\begin{equation*}
dX_t= -\sign(X_t)dt,\quad X_0=0
\end{equation*}
with the convention $\sign(0)=1$ has no solutions.  On the other hand, if we perturb these equations by a random Brownian forcing, that is if we consider  an SDE 
\begin{equation}\label{mainSDE} 
dX_t = b(X_t)\,dt + dW_t,\quad X_0=x,
\end{equation}
 where $d\in\N$, $x\in\R^d$, $W$ is a $d$-dimensional standard Brownian motion, then this equation has a unique solution when the drift $b$ is merely bounded measurable.
This phenomenon is called \textit{regularization by noise}; we refer the
reader to the monograph \cite{F11} and  a  short note \cite{djurdjevac2024randomness} for many interesting examples and further discussion.

The intuition behind regularization by noise is that the stochastic perturbation prevents the solution from getting stuck at points where the deterministic dynamics are singular. For example, the dynamics of equation \eqref{ode} allow the solution to remain at zero for an arbitrary amount of time, move up, or move down. Adding a noise term creates fluctuations that immediately push the process away from these unstable points, effectively selecting a unique trajectory. 

Let us be a bit more precise and recall the standard definitions. Fix $d\in\N$. Let $X\colon\Omega\times[0,1]\to\R^d$ be a continuous process. We say  that a couple $(X,W)$ on a complete filtered probability space $(\Omega,\F,(\F_t)_{t\in[0,1]},\P)$ is  a \textit{weak solution} to SDE \eqref{mainSDE}, if $W$ is an $(\F_t)$-Brownian motion, $X$ is adapted to $(\F_t)$, and 
\begin{equation*}
X_t=x+\int_0^t b(X_r)\,dr + W_t,\quad \text{a.s. for $t\in[0,1]$}.
\end{equation*}
 A weak solution $(X,W)$ is called a \textit{strong solution} if $X$ is adapted to $(\F^W_t)$, that is  the completion of the filtration generated by $W$. We say that \textit{strong (pathwise) uniqueness} holds for \eqref{mainSDE} if for any two weak solutions of  \eqref{mainSDE} $(X,W)$ and
$(Y,W)$ with common noise $W$ on a common filtered probability space  one has $\P (X_t = Y_t \text{ for all $t\in[0,1]$}) = 1$.

A seminal result of Veretennikov and Zvonkin states that regularization by noise happens for any bounded measurable drift. We refer to \cite{kr_rock05} for further extensions. 
\begin{theorem}[\cite{ver80,zvonkin74}]\label{t:VK}
Let $d\in \N$, $x\in\R^d$, $b\colon\R^d\to\R^d$ be a bounded measurable function. Then SDE (1.3) has a strong solution, and pathwise uniqueness holds.
\end{theorem}

Establishing this result using stochastic sewing will be our primary goal in the first part of the course. Note that \cite{ver80,zvonkin74} employed a very different proof strategy, namely a suitable transformation of the equation. Although their method is very powerful, it relies heavily on It\^o’s calculus and PDE techniques, which are largely inapplicable beyond the semimartingale setting, for instance when the noise is a fractional Brownian motion or space–time white noise. In contrast, stochastic sewing is much more general and requires minimal assumptions on the driving noise.

Let us understand the main difficulty in proving \cref{t:VK} and see what we need to overcome it. 

Fix $d\in\N$. For a function $f\colon Q\to E$, where $Q\subset \R^d$, $(E,\|\cdot\|_E)$ is a normed space and $\gamma\in(0,1]$, we define H\"older norms and seminorms:
\begin{align*}
&\|f\|_{\C^0(Q;E)}:=\sup_{x\in Q}\|f(x)\|_E;\qquad [f]_{\C^\gamma(Q;E)}:=\sup_{x,y\in Q}\frac{\|f(x)-f(y)\|_E}{|x-y|^\gamma};\\
&\|f\|_{\C^\gamma(Q;E)}:=\|f\|_{\C^0(Q;E)}+[f]_{\C^\gamma(Q;E)}.
\end{align*}
Denote by $\C^\gamma(Q;E)$ the space of all bounded continuous functions $f\colon Q\to E$ such that
${\|f\|_{\C^\gamma(Q;E)}<\infty}$. If $Q=\R^d$ and $E=\R$ we will often drop the indices and write $\C^\gamma$ in place of $\C^\gamma(Q;E)$.

Let $X=:W+\phi$, $Y=:W+\psi$ be two solutions to the SDE \eqref{mainSDE} with the initial condition $x\in\R^d$. We would like to show that $X=Y$. A naive attempt to prove this is to write
\begin{equation}\label{firstattempt}
|X_t-Y_t|=|\phi_t-\psi_t|=\Bigl|\int_0^t (b(W_r+\phi_r)-b(W_r+\psi_r))\,dr \Bigr|\le [b]_{\C^\gamma}\int_0^t |\phi_r-\psi_r|^\gamma\,dr, \quad  t\in[0,1].
\end{equation}
However, this does not give much: if $\gamma = 1$, Gronwall’s lemma implies that $X-Y \equiv 0$, but for $\gamma < 1$ (our main case of interest), we obtain nothing. Indeed,  estimate \eqref{firstattempt} does not imply uniqueness, since the corresponding integral inequality has nonzero solutions with zero initial value (recall \eqref{ode}). Instead, we aim to replace the bound in \eqref{firstattempt} with the following:
\begin{equation}\label{whatwewantmore}
\Bigl\|\int_0^t (b(W_r+\phi_r)-b(W_r+\psi_r))\,dr \Bigr\|_{L_2(\Omega)}\le C [b]_{\C^\gamma} t^{\rho}\sup_{r\le t} \|\phi_r-\psi_r\|_{L_2(\Omega)}, \quad  t\in[0,1]
\end{equation}
for some $\rho > 0$. This is exactly a bound of the type \eqref{type}! At the very least, as a starting point, we would like to show that for $x, y \in \R^d$
\begin{equation}\label{whatwewant}
\Bigl\|\int_0^t (b(W_r+x)-b(W_r+y))\,dr \Bigr\|_{L_2(\Omega)}\stackrel{???}{\le} C \|b\|_{\C^\gamma} t^{\rho}|x-y|, \quad  t\in[0,1].
\end{equation}

We establish bounds \eqref{whatwewantmore} and \eqref{whatwewant} in \cref{s:s3} and their extension to the fractional Brownian case will be done in \cref{e:313,e:ext,e:strongfbm}.

\cref{s:WU} is devoted to \textit{weak} regularization by noise, that is, uniqueness in law of solutions to \eqref{mainSDE} with even less regular drifts: we assume that $b$ is a Schwartz distribution. In this case, the drift  $\int b(X_t)dt$ is not defined directly but will be understood as a limit of the corresponding approximations $\int b^n(X_t)dt$, where $(b^n)$ is a suitable sequence of smooth approximations of $b$. Here a different stochastic sewing-type argument will be needed, as for such irregular $b$ bounds 
\eqref{whatwewantmore} and \eqref{whatwewant} do not hold. We will prove the following result.

\begin{theorem}[\cite{bib:zz17,FIR17}]\label{t:weak}
Let $d\in \N$, $x\in\R^d$, $b\in\C^\gamma(\R^d,\R^d)$, $\gamma>-\frac12$. Then SDE \eqref{mainSDE}
has a weak solution $(X,W)$. If $(\wt X,\wt W)$ is another weak solution to \eqref{mainSDE} belonging to a certain natural class, then $\law(X)=\law(\wt X)$.
\end{theorem}
We refer to \cref{s:41} for the precise definitions, including the definition of the space $\C^\gamma$ for negative $\gamma$. Let us note here that the technique of \cite{bib:zz17,FIR17} was again the Zvonkin transformation method and thus relies on the Markov property of the noise. We will show how this result can be established using stochastic sewing, which has to be combined with some ideas from ergodic theory. Since stochastic sewing does not require an It\^o formula to be applicable, the same ideas also lead to  weak uniqueness for SDEs driven by a fractional Brownian motion and for stochastic PDEs.

Finally, in \cref{s:5} we see how stochastic sewing works for SDEs driven by a discontinuous driver. Let   $\alpha\in(0,2)$. Recall that $L\colon\Omega\times[0,1]\to\R^d$ is a symmetric  rotationally invariant $\alpha$-stable process if it is c\`adl\`ag, has independent stationary increments and 
\begin{equation}\label{charlevy}
	\E \exp(i\lambda^\t L_t)=\exp(- c_{\alpha,d} |\lambda|^\alpha t),\qquad\lambda\in\R^d,\, t\in[0,1]
\end{equation}
for some $c_{\alpha,d}>0$. We will give a stochastic sewing proof of the following result in the regime $\alpha\in(1,2)$.

\begin{theorem}[\cite{Pr12,chen2017well,bib:csz15}]\label{t:VKlevy}
Let $d\in \N$, $x\in\R^d$, $\alpha\in(0,2)$. Suppose that 
\begin{equation}\label{levycond}
\gamma>1-\frac\alpha2.
\end{equation}
Let $b\in\C^\gamma(\R^d,\R^d)$. Then SDE
\begin{equation}\label{mainSDElevy}
X_t=x+\int_0^t b(X_r)\,dr +L_t
\end{equation}
 has a strong solution, and pathwise uniqueness holds.
\end{theorem}
We note again that the above-mentioned papers \cite{Pr12,chen2017well,bib:csz15} established this result using a Zvonkin transformation method, whereas in \cref{s:5} we prove  it directly using stochastic sewing combined with a number of new tools.

\subsection{Example 2: Numerical methods}\label{s:12}

Next, once the well-posedness of the SDE \eqref{mainSDE} is established, let us discuss how this equation can be solved numerically. It turns out that stochastic sewing techniques are also very useful for analyzing numerical algorithms. 

We consider the SDE \eqref{mainSDE} and study the convergence of the standard Euler scheme. That is,  we fix $n\in\N$, the initial condition $x\in\R^d$, consider a partition $\{0=t_0<t_1<\dots<t_n=1\}$ of the interval $[0,1]$,  and define the process $X^n$ recursively by $X_0^n=x$ and
\begin{equation}\label{Eulerscheme}
	X_{t_{i+1}}^n=X_{t_i}^n+(t_{i+1}-t_i)b(X^n_{t_i})+(W_{t_{i+1}}-W_{t_{i}}),\quad i=0,\hdots,n-1.
\end{equation}
To simplify the notation we work only on  uniform grids, so $t_i=i/n$. For an integer $n\ge1$ denote 
\begin{equation}\label{kappanr}
	\kappa_n(t):=\frac{\lfloor nt\rfloor}n,\quad t\in[0,1],
\end{equation}	
that is the  largest point of the grid $\{0,\frac1n,\frac2n,\hdots\}$ which is less than or equal to $t$. Then we can rewrite the Euler scheme \eqref{Eulerscheme} as 
\begin{equation}\label{mainSDEn} 
X^n_t=x+\int_0^t b(X^n_{\kappa_n(s)})\,ds+W_t,\quad  t\in[0,1].
\end{equation}
Indeed, it is easy to see that this equation agrees with \eqref{Eulerscheme} when $t$ is a grid point. 

Let $b\in\C^\gamma$, $\gamma\in(0,1]$. It is clear that for any $t\in[0,1]$ we have the following naive bound: 
\begin{align*}
\Bigl |X^n_t-X_t\Bigr|&=\Bigl|\int_0^t b(X^n_{\kappa_n(r)})-b(X_r)\,dr\Bigr|\\
&\le \Bigl|\int_0^t b(X^n_{\kappa_n(r)})-b(X^n_r)\,dr \Bigr|+\Bigl|\int_0^t b(X^n_r)-b(X_r)\,dr \Bigr|\\
&\le \|b\|_{\C^\gamma}\int_0^t |X^n_{\kappa_n(r)}-X^n_r|^\gamma \,dr +\Bigl|\int_0^t b(X^n_r)-b(X_r)\,dr \Bigr|\\
&\le \|b\|_{\C^\gamma}\|b\|_{\C^0}^\gamma n^{-\gamma}+\|b\|_{\C^\gamma}\int_0^t |W_{\kappa_n(r)}-W_r|^\gamma \,dr +\Bigl|\int_0^t b(X^n_r)-b(X_r)\,dr \Bigr|.
\end{align*}
Since $\E |W_{\kappa_n(r)}-W_r|^\gamma\le C n^{-\gamma/2}$, we see that the best rate one can obtain  with the naive approach (even ignoring all other terms) is $n^{-\gamma/2}$. In \cref{s:numbd}, we will use stochastic sewing to improve this rate by $1/2$. More precisely, we will show the  following result.

\begin{theorem}[\cite{BDG}]\label{t:mainnumt}
Let $d\in \N$, $m\ge1$, $x\in\R^d$, $\gamma\in(0,1]$, $\eps>0$, $b\in\C^\gamma(\R^d,\R^d)$. Let $X$ solve \eqref{mainSDE} and $X^n$ be its Euler scheme \eqref{mainSDEn}. Then there exists a constant  $C=C(\|b\|_{\C^\gamma},\gamma,\eps,d,m)>0$ such that for any $n\in\N$ we have
	\begin{equation}\label{mbt35}
		\Bigl\|\sup_{t\in[0,1]}|X_t-X^n_t|\Bigr\|_{L_m(\Omega)} \le Cn^{-\frac12-\frac\gamma2+\eps}.
	\end{equation}
\end{theorem}

Comparing this result with \cref{t:VK}, we see that \cref{t:mainnumt} establishes convergence in the entire regime (up to a loss of an arbitrary $\eps>0$) of $\gamma$ for which uniqueness of a strong solution is known. Furthermore, the new convergence rate does not deteriorate to $0$ as the regularity of the drift approaches the threshold $\gamma=0$, and it is at least $n^{-\frac12}$. Moreover, it is known that this rate is optimal \cite{ellinger2025optimal}; however, showing the optimality is out of the scope of these lecture notes.

In \cref{s:JN}--\cref{s:levynum}, we extend this result to SDEs driven by a jump process. This requires a significant extension of the stochastic sewing arguments and the introduction of several new tools. In the end, we will prove the following bound on the convergence rate.

\begin{theorem}[\cite{BDGLevy}]\label{t:levy}
Let $d\in \N$, $m\ge1$, $x\in\R^d$, $\alpha\in(1,2)$, $\eps>0$. Let $b\in\C^\gamma(\R^d,\R^d)$, where $\gamma\in(0,1]$ satisfies condition \eqref{levycond}. Let $X$ be a solution to  \eqref{mainSDElevy}, and let $X^n$ be the corresponding Euler scheme
\begin{equation*}
	dX^n_t=b(X^n_{\kappa_n(t)})\,dt+\,dL_t,\qquad X^n_0=x,\quad t\in[0,1].
\end{equation*}

 Then there exists a constant  $C=C(\|b\|_{\C^\gamma},\alpha,\gamma,\eps,d,m)>0$ such that for any $n\in\N$ we have
	\begin{equation*}
		\|\sup_{t\in[0,1]}|X_t-X^n_t|\|_{L_m(\Omega)} \le Cn^{-\frac12-(\frac\gamma\alpha\wedge\frac12)+\eps}.
	\end{equation*}
\end{theorem}
We see again that for $\alpha\in(1,2)$, \cref{t:levy} provides a convergence rate in the entire range of $\gamma$ where strong well-posedness is known.

Finally, in \cref{e:mbb2}--\cref{e:531} we will study numerics for SDEs with Sobolev rather than H\"older drifts. This will require further generalizations of stochastic sewing.


\section{Sewing and  stochastic sewing}\label{s:2}

\subsection{Introduction to sewing lemmas}
We begin with the presentation of the sewing lemma and its stochastic analogue.  For $0 \le S < T$, consider the simplices
\begin{equation*}
	\Delta_{[S,T]}:=\{(s,t)\in[S,T]^2\colon s< t\};\quad \Delta^3_{[S,T]}:=\{(s,u,t)\in[S,T]^3\colon s< u< t\}.
\end{equation*}
If $A_{\cdot,\cdot}$ is a function $\Delta_{[S,T]}\to E$, where $E$ is a Banach space,  we put
\begin{equation}\label{deltadiff}
	\delta A_{s,u,t}:= A_{s,t}-A_{s,u}-A_{u,t},\quad  (s,u,t)\in\Delta^3_{[S,T]}. 
\end{equation}
For a partition $\Pi=\{S=t_0<\dots<t_k=T\}$ of the interval $[S,T]$, we write $|\Pi|$ for its mesh, that is, the length of its longest subinterval: $|\Pi|:=\max_i |t_i-t_{i-1}|$.

The next statement is the celebrated sewing lemma, obtained independently
by Feyel and De La Pradelle \cite{PressF}  and Gubinelli \cite{Gubi}.
\begin{theorem}[Sewing lemma, \cite{PressF,Gubi}]\label{t:SL}
Let $E$ be a Banach space, $0\le S< T$. Suppose that there exist functions $\A\colon [S,T]\to E$, $A\colon\Delta_{[S,T]}\to E$, such  that the following holds:
\begin{enumerate}[(i)]
\item  for any sequence of partitions $\Pi_N:=\{S=t^N_0,...,t^N_{k(N)}=T\}$ of $[S,T]$ with $|\Pi_N|\to0$  one has
\begin{equation}\label{conssl:ds3}
	\A_T-\A_S=\lim_{N\to\infty} \sum_{i=0}^{k(N)-1} A_{t^N_i,t_{i+1}^N};
\end{equation}
\item there exist $\gamma>1$, $\Gamma>0$ such that for every $(s,u,t)\in\Delta^3_{[S,T]}$ we have
\begin{equation}\label{keycondSL}	
\| \delta A_{s,u,t}\|_E\le \Gamma|t-s|^{\gamma}.
\end{equation}
\end{enumerate}	
Then there exists a constant $C=C(\gamma)>0$  independent of $S$, $T$, $\Gamma$ such that 
\begin{equation}
		\|\A_{T}-\A_{S}\|_E\le \|A_{S,T}\|_E+	C\Gamma |T-S|^{\gamma}.
		\label{est:dssl1}
	\end{equation}
\end{theorem}
\begin{remark}\label{nocondi}
Actually, condition (i) of the sewing lemma can be omitted. One can show that condition (ii) alone implies that the Riemann sums $\sum A_{t^N_i,t_{i+1}^N}$ always converge and that their limit satisfies \eqref{est:dssl1}; see, e.g., \cite[Lemma 4.2]{FH}. However, we keep condition~(i), since it significantly simplifies the proof and we will only apply the sewing lemma in this form.
\end{remark}

\begin{proof}[Proof of \cref{t:SL}]
Fix $0\le S< T$. For $n\in\Z_+$, consider a dyadic partition $t^n_i=S+\frac{i}{2^{n}}(T-S)$, $i=0, 1,\hdots, 2^{n}$, of the interval $[S,T]$. Denote the $n$th Riemann sum by 
\begin{equation}\label{abrn}
A^{(n)}:=\sum_{i=0}^{2^n-1} A_{t_i^n,t_{i+1}^n}.
\end{equation}
Then, by \eqref{conssl:ds3} we have
\begin{align}\label{whatwehaveSL}
\|\A_T-\A_S\|_E&=\lim_{N\to\infty}\| A^{(N)}\|_E\le \|A_{S,T}\|_{E}+\lim_{N\to\infty}\Bigl\|\sum_{n=0}^{N-1} (A^{(n+1)}-A^{(n)})\Bigr\|_E\nn\\
&\le \|A_{S,T}\|_{E}+ \sum_{n=0}^\infty\| A^{(n+1)}-A^{(n)}\|_E.
\end{align}
Using condition \eqref{keycondSL}, it is easy to see that for any $n\in\Z_+$
\begin{equation*}
\|A^{(n+1)}-A^{(n)}\|_E=\Bigl\|\sum_{i=0}^{2^n-1}\delta A_{t_i^n,t_{2i+1}^{n+1},t_{i+1}^n}\Bigr\|_E\le
\sum_{i=0}^{2^n-1}\bigl\|\delta A_{t_i^n,t_{2i+1}^{n+1},t_{i+1}^n}\bigr\|_E\le \Gamma|T-S|^\gamma 2^{-n(\gamma-1)}.
\end{equation*}
Summing this over $n\in\Z_+$ (recall that $\gamma>1$ by assumption) and substituting into \eqref{whatwehaveSL}, we finally get
\begin{equation*}
	\|\A_T-\A_S\|_E\le \|A_{S,T}\|_{E}+ C\Gamma |T-S|^\gamma,
\end{equation*}
for $C=C(\gamma)$ which is the desired bound \eqref{est:dssl1}.
\end{proof}

The sewing  lemma is crucial in the context of rough path theory, see, e.g., \cite{FH}. We can also try to apply it for our purposes. By taking $E=L_m(\Omega)$, $m \ge 1$, in \cref{t:SL}, we see that we need to verify the following condition on increments
\begin{equation}\label{notgood}
\| \delta A_{s,u,t}\|_{L_m(\Omega)}\le \Gamma|t-s|^{\gamma}
\end{equation}
for $\gamma>1$. However, this condition can be significantly improved if we switch to the stochastic setting. As an analogy, recall that to define the deterministic Young integral $\int f_sdg_s$, where $f\in\C^\alpha$, $g\in\C^\beta$, one generally requires $\alpha+\beta>1$; see, e.g., \cite[Section~4.1]{FH}. In comparison, thanks to stochastic cancellations, the It\^o integral $\int f_sdW_s$ can be defined for progressively measurable, square-integrable $f$ without any H\"older regularity assumption.

Let $(\Omega, \F,\P)$ be a probability space, $T>0$. If $(\F_t)_{t\in[0,T]}$ is a complete filtration, then we write
\begin{equation*}
\E^t[\cdot]:=\E[\cdot|\F_t],\quad t\in[0,T].
\end{equation*}	
We will also frequently use the following inequality, which follows directly from Jensen’s inequality: if $m \ge 1$, $X \colon \Omega \to \mathbb{R}^d$ is a random variable, and $\mathcal{G} \subset \mathcal{F}$ is a $\sigma$-algebra, then
\begin{equation}\label{useful}
	\| \E [X|\mathcal{G}]\|_{L_m(\Omega)}\le 
	\| X\|_{L_m(\Omega)}.
\end{equation}

We are now ready to discover the central tool of these lecture notes: the stochastic sewing lemma of L\^e. 
\begin{theorem}[Stochastic sewing lemma, \cite{LeSSL}]\label{T:SSLst}
Let $m\in[2,\infty)$, $0\le S< T$. 
Suppose that there exist measurable functions ${\A\colon \Omega\times[S,T]\to \R^d}$, $A\colon\Omega\times \Delta_{[S,T]}\to \R^d$ and a complete filtration $(\F_t)_{t\in[S,T]}$ such that the following holds:
\begin{enumerate}[(i)]
\item 
for  any sequence of deterministic partitions $\Pi_N:=\{S=t^N_0,...,t^N_{k(N)}=T\}$ of $[S,T]$ with $|\Pi_N|\to0$  one has
	\begin{equation}\label{con:s3}
		\sum_{i=0}^{k(N)-1} A_{t^N_i,t_{i+1}^N}\to  \A_{T}-\A_{S}\quad\text{in probability as $N\to\infty$};
	\end{equation}	
	\item for any $(s,t)\in \Delta_{[S,T]}$, the random variable $A_{s,t}$ is $\F_t$--measurable;
	\item there exist $\Gamma_1, \Gamma_2\ge0$, $\gamma_1>\frac12$, $\gamma_2>1$ such that for every $(s,u,t)\in\Delta^3_{[S,T]}$ we have 
	\begin{align}	
		&\|\delta A_{s,u, t}\|_{L_m(\Omega)}\le \Gamma_1|t-s|^{\gamma_1},\label{con:s1}	\\
		&\|\E^s \delta A_{s,u,t}\|_{L_m(\Omega)}\le \Gamma_2|t-s|^{\gamma_2}.\label{con:s2}
	\end{align}
\end{enumerate}	
Then there exists a constant $C=C(\gamma_1,\gamma_2,d,m)$ independent of $S$, $T$, $\Gamma_1$, $\Gamma_2$ such that 
	\begin{equation}
		\|\A_{T}-\A_{S}\|_{L_m(\Omega)} \le \|A_{S,T}\|_{L_m(\Omega)}+ C \Gamma_1 |T-S|^{\gamma_1}+
		C \Gamma_2 |T-S|^{\gamma_2}.
		\label{est:ssl1}
	\end{equation}
\end{theorem}

\begin{remark}\label{r:SSL}
\begin{enumerate}[(i)]
\item It is interesting to compare the conditions \eqref{con:s1} and \eqref{con:s2} of the stochastic sewing lemma with condition \eqref{notgood}. We observe that assumption \eqref{con:s2}, compared to \eqref{notgood}, additionally involves a conditional expectation. As we will see below, $\E^s \delta A_{s,u,t}$ is often much smaller than $\delta A_{s,u,t}$. We also note that the exponent $\gamma_1$ in condition \eqref{con:s1} is assumed to be greater than $1/2$, whereas the corresponding parameter $\gamma$ in \eqref{notgood} was required to be greater than $1$. All this will be crucial in our applications of the stochastic sewing lemma later.
\item As in the deterministic case (cf. \cref{nocondi}), conditions (ii) and (iii) of \cref{T:SSLst} imply that there exists a process $\A$ satisfying condition (i); see the original paper of L\^e \cite{LeSSL}. We retain condition (i) because it considerably simplifies the proof and this formulation suffices for all our applications.
\end{enumerate}
\end{remark}

To prove the stochastic sewing lemma, we recall the Burkholder-Davis-Gundy inequality (BDG inequality).
\begin{proposition}[BDG inequality, \cite{Bu66,BDG72}]\label{p:BDG}
Let $n\in\N$, let $(\F_i)_{i=0,\hdots,n}$ be a filtration. Let  $(X_i,\F_i)_{i=1,\hdots,n}$ be a sequence of martingale differences, that is each $X_i$ is $\F_i$-measurable and $\E (X_i|\F_{i-1})=0$. Then for any $m\ge1$ there exists a constant $C=C(m)>0$ so that
\begin{equation*}
\E \Bigl|\sum_{i=1}^n X_i \Bigr|^m\le C\E\Bigl[ \bigl(\sum_{i=1}^n | X_i|^2\bigr)^{\frac{m}2}\Bigr]
\end{equation*}
\end{proposition}
In the setting of \cref{p:BDG}, for $m\ge2$ we have
\begin{equation}\label{BDG}
\Bigl\|\sum_{i=1}^n X_i \Bigr\|_{L_m(\Omega)}\le C\Bigl\| \sum_{i=1}^n | X_i|^2\Bigr\|_{L_{\frac{m}2}(\Omega)}^{\frac12}\le C\Bigl( \sum_{i=1}^n \| X_i\|^2_{L_m(\Omega)}\Bigr)^{\frac12},
\end{equation}
where $C=C(m)>0$. The first inequality is the BDG inequality, while the second follows from Minkowski’s inequality in $L_{m/2}(\Omega)$, which applies since $m/2\ge1$.

\begin{proof}[Proof of \cref{T:SSLst}]
	Fix $m\ge2$, $0\le S<T$. 
As before, we consider the dyadic partition $t^n_i=S+\frac{i}{2^{n}}(T-S)$, $i=0, 1,\hdots, 2^{n}$, $n\in\Z_+$ of the interval $[S,T]$. We define $A^{(n)}$ as in \eqref{abrn} and note that by Fatou's lemma
\begin{equation}\label{whatwehaveSSL}
\|\A_T-\A_S\|_{L_m(\Omega)}\le \liminf_{n\to\infty}\| A^{(n)}\|_{L_m(\Omega)}\le \|A_{S,T}\|_{L_m(\Omega)}+ \sum_{n=0}^\infty\| A^{(n+1)}-A^{(n)}\|_{L_m(\Omega)}.
\end{equation}

We estimate the difference  $A^{(n+1)}-A^{(n)}$ using the BDG inequality  and conditions \eqref{con:s1}, \eqref{con:s2}. We get for $n\in\Z_+$
\begin{align}
&\|A^{(n+1)}-A^{(n)}\|_{L_m(\Omega)}\nn\\
&\,\,=\Bigl\|\sum_{i=0}^{2^n-1}\delta A_{t_i^n,t_{2i+1}^{n+1},t_{i+1}^n}\Bigr\|_{L_m(\Omega)}\nn\\
&\,\,\le
\Bigl\|\sum_{i=0}^{2^n-1}(\delta A_{t_i^n,t_{2i+1}^{n+1},t_{i+1}^n}-\E^{t_i^n}\delta A_{t_i^n,t_{2i+1}^{n+1},t_{i+1}^n})\Bigr\|_{L_m(\Omega)}+\Bigl\|\sum_{i=0}^{2^n-1}\E^{t_i^n}\delta A_{t_i^n,t_{2i+1}^{n+1},t_{i+1}^n}\Bigr\|_{L_m(\Omega)}\nn\\
&\,\,\le 
C \Bigl(\sum_{i=0}^{2^n-1}\bigl\|\delta A_{t_i^n,t_{2i+1}^{n+1},t_{i+1}^n}\!\!-\E^{t_i^n}\delta A_{t_i^n,t_{2i+1}^{n+1},t_{i+1}^n}\bigr\|_{L_m(\Omega)}^2\Bigr)^{\frac12}\!+\!\sum_{i=0}^{2^n-1}\bigl\|\E^{t_i^n}\delta A_{t_i^n,t_{2i+1}^{n+1},t_{i+1}^n}\bigr\|_{L_m(\Omega)}\label{bdguse}\\
&\,\,\le 
C \bigl(2^n \Gamma_1^2 2^{-2n\gamma_1} (T-S)^{2\gamma_1}\bigr)^\frac12+C \Gamma_2|T-S|^{\gamma_2}2^{-n\gamma_2}2^n\label{bdguse2}\\
&\,\,\le 
C \Gamma_1 2^{-n(\gamma_1-\frac12)} (T-S)^{\gamma_1}+C \Gamma_2|T-S|^{\gamma_2}2^{-n(\gamma_2-1)},\label{bdguse3}
\end{align}
for $C=C(\gamma_1,\gamma_2,d,m)>0$. Here in \eqref{bdguse} we used that the sequence 
$$
\{\delta A_{t_i^n,t_{2i+1}^{n+1},t_{i+1}^n}-\E^{t_i^n}\delta A_{t_i^n,t_{2i+1}^{n+1},t_{i+1}^n},\F_{t_{i+1}^n}\}_{i=0,\hdots,2^n-1}
$$
is a martingale difference sequence, and we applied the BDG inequality \eqref{BDG}. Inequality \eqref{bdguse2} follows from the elementary inequality $(a+b)^2\le 2(a^2+b^2)$ valid for any $a,b\in\R$ and \eqref{useful}. 
Summing \eqref{bdguse3} over $n\in\Z_+$ and substituting into \eqref{whatwehaveSSL}, we finally get
\begin{equation*}
\|\A_T-\A_S\|_{L_m(\Omega)}\le \|A_{S,T}\|_{L_m(\Omega)}+C \Gamma_1  (T-S)^{\gamma_1}+C \Gamma_2|T-S|^{\gamma_2},
\end{equation*}
for some $C=C(\gamma_1,\gamma_2,d,m)>0$ which is \eqref{est:ssl1}.
\end{proof}

\begin{remark}\label{r:26}
To verify condition \eqref{con:s1} it is sufficient to check that there exist constants $\Gamma_1\ge0$, $\gamma_1>\frac12$ so that for any $(s,t)\in\Delta_{[S,T]}$ one has 
\begin{equation*}
\|A_{s,t}\|_{L_m(\Omega)}\le \Gamma_1|t-s|^{\gamma_1}.
\end{equation*}
\end{remark}
 
\subsection{Exercises}
 
The exercises in this subsection are inspired by \cite[Section~1]{Pavel}.
 
We will use the Burkholder–Davis–Gundy inequality (\cref{p:BDG}) only in the regime $m\ge2$. Let us prove it in this setting.  We begin with Doob's inequality, which reduces the problem of bounding the maximum of a martingale to bounding its terminal value.

\begin{exercise}[Doob's inequality]\label{E:doob}
Let $n\in\N$, $m>1$. Let $\FF=(\F_i)_{1\le i\le n}$ be a filtration. Let $((M_i)_{i=1,\hdots n},\FF)$ be a martingale. Show that 
	$$
	\E \max_{k\le n} |M_k|^m\le \Bigl(\frac{m}{m-1}\Bigr)^m \E  |M_n|^m.
	$$

	\hns 1) Use the identity $x^m=m\int_0^\infty \lambda^{m-1}\I(\lambda<x)\,d\lambda$ valid for any $x\ge0$. Here and below, $\I(\cdot)$ denotes the indicator function.
	
	2) To bound $\P(\max_{k\le n} |M_k|>\lambda)$, use the identity
	$$
	\{\max_{k\le n} |M_k|>\lambda\}=\bigcup_{i=1}^n \{\max_{k\le i-1} |M_k|\le\lambda, |M_i|>\lambda\}
	$$
	and the inequality  $\I(|M_i|>\lambda)\le \lambda^{-1}|M_i|\I(|M_i|>\lambda)$.
	
	3) First assume additionally that $M$ is bounded by an arbitrarily large constant, and then apply Fatou's lemma to remove this assumption.
\end{exercise} 	

A key step toward the proof of the BDG inequality is the following nice idea.

\begin{exercise}[Garsia-Neveu's lemma]
	Let $\eta, \xi$ be two nonnegative random variables.  Assume that for any $\lambda\ge0$ 
	\begin{equation*}
		\E(\eta-\lambda)^2\I(\eta>\lambda)\le \E\xi^2\I(\eta>\lambda).
	\end{equation*}
	Show that this implies  $\E \eta^m\le \Bigl(\frac{m(m-1)}{2}\Bigr)^{m/2} \E\xi^m$ for any $m\ge2$. 
	
	\hn For $m>2$ use the identity $x^m=C_m\int_0^x (x-\lambda)^2\lambda^{m-3}\,d\lambda$ valid for any $x\ge0$, where $C_m=\frac{m(m-1)(m-2)}2$.
\end{exercise}

\begin{exercise}[Burkholder--Davis--Gundy inequality]\label{e:BDG}
Let $n\in\N$.  Let $\FF=(\F_i)_{0\le i\le n}$ be a filtration. Let  $(X_i,\F_i)_{i=1,\hdots,n}$ be a sequence of martingale differences, that is each $X_i$ is $\F_i$ measurable and $\E (X_i|\F_{i-1})=0$.
\begin{enumerate}[(i)]
\item Show that 	 $\E\bigl| \sum_{i=1}^n X_i\bigr|^2= \sum_{i=1}^n \E| X_i|^2$.
\item Use Garsia-Neveu's lemma, part (i) of this exercise, and Doob's inequality to show that for any $m\ge2$ 
$$
\E \Bigl|\sum_{i=1}^n X_i \Bigr|^m\le (2m)^m \E\Bigl[ \bigl(\sum_{i=1}^n | X_i|^2\bigr)^{m/2}\Bigr].
$$

\hn At some point, you may want to use a decomposition similar to the one in Hint~2) of \cref{E:doob}.
\end{enumerate}
\end{exercise}	
 
\section{Regularization by noise for SDEs driven by Brownian motion}\label{s:s3}

Now let us see how the stochastic sewing lemma can be used to obtain uniqueness for SDEs with irregular drift. We will follow the plan sketched in \cref{sect:11}: first, in \cref{sect:31}, we establish the bound \eqref{whatwewant}, then proceed to the proof of \eqref{whatwewantmore} in \cref{sect:32}, and finally, in \cref{s:SU}, prove \cref{t:VK} (up to the loss of an arbitrarily small $\eps>0$ in the regularity of the drift).

We would like to stress that the argument is quite general. For simplicity, and to highlight the main ideas, we illustrate it here for SDEs driven by Brownian motion and leave the extension to SDEs driven by fractional Brownian motion as an exercise for the reader; see \cref{s:upr3}. For other types of noise, the general proof strategy still applies, although further refinement of the stochastic sewing argument is required. We will identify which steps fail when the driving noise is a L\'evy process and fix them in \cref{s:levyint,s:sssl,s:levywp,sec:44}. 

We fix a probability space $(\Omega, \F,\P)$ and a complete filtration $(\F_t)_{t\in[0,1]}$. Let $d\in\N$ and let $W$ be a $d$-dimensional standard $(\F_t)$-Brownian motion.  The results in this section are mostly due to \cite{LeSSL}.

\subsection{Step 1: Bounding integral functionals of a Brownian motion}\label{sect:31}

We begin with the following technical but crucial lemma. Let $\Im_d$ be the $d\times d$ identity matrix.   Denote by $p_t$, $t>0$,  the density of a  $d$-dimensional Gaussian vector with zero mean and covariance matrix $t\Im_d$, and let $P_t$ be the corresponding semigroup, that is
\begin{equation}\label{heatkernel}
P_tf(x):=\int_{\R^d} p_t(y) f(x-y)\,dy,\quad p_t(x):=\frac1{(2\pi t)^{\frac{d}2}}e^{-\frac{|x|^2}{2t}},\qquad x\in\R^d,
\end{equation}
where $f$ is a bounded measurable function $\R^d\to\R$. Setting $P_0f:=f$, it is easy to check that $(P_t)$ indeed has the semigroup property
\begin{equation}\label{semigr}
P_{t+s}f=P_t(P_sf),\qquad s,t\ge0.
\end{equation}

\begin{lemma}\label{l:gb}
	Let $f\in\C^\gamma(\R^d,\R)$, $\gamma\in[0,1]$. Then there exists $C=C(d)>0$ such that for any  $t\in(0,1]$ we have 
	\begin{equation}\label{ineq}
		\|P_t  f\|_{\C^1}\le C t^{\frac{\gamma-1}2}\|f\|_{\C^\gamma};\quad
		\|\nabla P_t  f\|_{\C^1}\le C t^{\frac{\gamma-2}2}\|f\|_{\C^\gamma}.
	\end{equation} 	
\end{lemma}
\begin{proof} Assume additionally that $f\in\C^\infty(\R^d,\R)$. Then for any $x\in\R^d$ we have
	\begin{align*}
		\Bigl|\int_{\R^d}p_t(y)\nabla f(x-y)\,dy\Bigr|&= \Bigl|\int_{\R^d}\nabla p_t(y) f(x-y)\,dy\Bigr|=\Bigl|\int_{\R^d}\nabla p_t(y) (f(x-y)-f(x))\,dy\Bigr|\\
		&\le C\|f\|_{\C^\gamma}\int_{\R^d} |\nabla p_t(y)| |y|^\gamma \,dy\le 
		C \|f\|_{\C^\gamma}t^{-1}\int_{\R^d} p_t(y) |y|^{1+\gamma} \,dy\\
		&\le C\|f\|_{\C^\gamma}t^{-1}\E |W_t|^{1+\gamma}\le 
		C\|f\|_{\C^\gamma}t^{-\frac12+\frac\gamma2},
	\end{align*}
	where $C=C(d)>0$, and the last inequality follows from Brownian scaling. 	Therefore for any $x,y\in\R^d$
	\begin{equation*}
		|P_t f (x)-	P_t f (y)|=\Bigl|(x-y)^\t \int_0^1  \nabla P_t f(\theta x+(1-\theta) y)\,d\theta\Bigr|\le C|x-y|\|f\|_{\C^\gamma}t^{-\frac12+\frac\gamma2},
	\end{equation*}	
	which yields the first inequality in \eqref{ineq} for $f\in\C^\infty(\R^d,\R)$. 

In the general case, $f\in\C^\gamma$, we note that $P_\eps f\in\C^\infty(\R^d,\R)$ for any $\eps>0$ and therefore
	\begin{equation}\label{prelim31}
	\|P_t  f\|_{\C^1}=\|P_{t-\eps}(P_\eps f)\|_{\C^1}\le C (t-\eps)^{\frac{\gamma-1}2}\|P_\eps f\|_{\C^\gamma}\le C (t-\eps)^{\frac{\gamma-1}2}\|f\|_{\C^\gamma},
\end{equation}
for $C=C(d)$. Here in the first identity we used the semigroup property \eqref{semigr}, and in the last inequality we used that 
\begin{equation}\label{contra}
\|P_\eps f\|_{\C^\gamma}\le \int_{\R^d}p_\eps(y)\|f(\cdot-y)\|_{\C^\gamma}\,dy=\|f\|_{\mathcal C^\gamma}.
\end{equation}
By passing to the limit in \eqref{prelim31} as $\eps\to0$, we get the first part of \eqref{ineq}.
	
To get the second part of \eqref{ineq}, we apply the first part with $\gamma=0$. Using again the semigroup property of $P_t$, we get 
	\begin{equation*}
		\|\nabla P_t  f\|_{\C^1}=\| P_{\frac{t}2}\nabla P_{\frac{t}2}  f\|_{\C^1}\le C t^{-\frac12}\|\nabla P_{\frac{t}2}  f\|_{\C^0}\le C t^{-\frac12}\| P_{\frac{t}2}  f\|_{\C^1}\le C\|f\|_{\C^\gamma} t^{-1+\frac\gamma2},
	\end{equation*}	
	for $C=C(d)>0$, which implies the statement of the lemma.
\end{proof}

Consider now a function $b\in\C^\gamma(\R^d,\R^d)$, $\gamma>0$.  Recall that we would like to bound the moments of 
\begin{equation}\label{ourint}
\A_t:=\int_0^t (b(W_r+x)-b(W_r+y))\,dr,\quad x,y\in\R^d,\,\, t\in[0,T],
\end{equation}
in terms of $|x-y|$ (weaker bounds of the form $|x-y|^\rho$ for $\rho<1$ are not allowed).
We would like to use the SSL, that is \cref{T:SSLst}. But how to do it? What should we choose as a germ $A_{s,t}$? A naive attempt would be to approximate for $(s,t)\in \Delta_{[0,T]}$
\begin{equation*}
\A_t-\A_s\approx (t-s)(b(W_s+x)-b(W_s+y)).
\end{equation*}
Recall, however, that the optimal approximation of $\A_t-\A_s$ in the $L_2(\Omega)$ sense by an $\F_s$-measurable random variable is given by the conditional expectation.
\begin{equation*}
\A_t-\A_s\approx \E^s \int_s^t (b(W_r+x)-b(W_r+y))\,dr.
\end{equation*}
We will use this as our germ $A_{s,t}$. Thus, we are now ready to present the very first application of the SSL. Before we begin, let us observe that the following very useful identity holds: for any bounded measurable $f\colon\R^d\to\R$, $0\le s< t$, $x\in\R^d$ we have
\begin{equation}\label{1bound}
\E^s f(W_t+x)=\E^s f(W_s+(W_t-W_s)+x)=\E f(W_t-W_s+y)\evalat{y=W_s+x}=P_{t-s}f (W_s+x),
\end{equation}
where we used that $W_t-W_s$ is independent of $\F_s$ and has density $p_{t-s}$.
\begin{theorem}\label{t:firstk}
Let $m\ge2$, $\gamma\in(0,1]$, $b\in \C^\gamma(\R^d,\R)$. Then there exists a constant $C=C(\gamma,d,m)>0$ such that for any $x,y\in\R^d$, $(S,T)\in\Delta_{[0,1]}$ one has 
\begin{equation}\label{firstkey}
\Bigl\|\int_S^T (b(W_r+x)-b(W_r+y))\,dr\Bigr\|_{L_m(\Omega)}\le C\|b\|_{\C^\gamma} |x-y||T-S|^{\frac12+\frac\gamma2}.
\end{equation}	
\end{theorem}	

\cref{t:firstk} is illustrated in \cref{fig:brownian-averaging}. Even if the original function $b$ is discontinuous, its average  along a Brownian path $\wt b(x):=\int_0^1 b(x+W_r)\,dr$, $x\in\R$, is a.s. locally almost Lipschitz, as follows from \eqref{firstkey}.

\begin{figure}[htbp]
	\centering
	\subfloat[$b$]{%
		\includegraphics[width=0.48\linewidth]{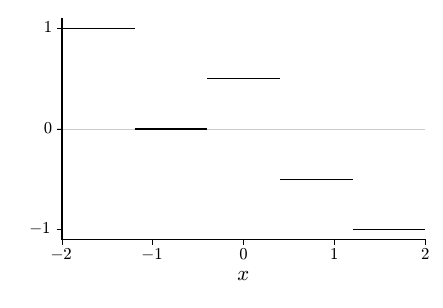}}%
	\hfill
	\subfloat[$\wt b(x):=\int_0^1 b(x+W_r) dr$]{%
		\includegraphics[width=0.48\linewidth]{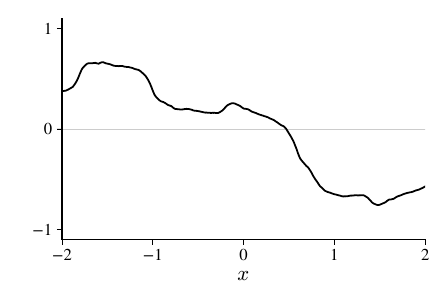}}
	\caption{A discontinuous function $b$ and its average $\wt b$ along one Brownian path.}
	\label{fig:brownian-averaging}
\end{figure}

\begin{proof}[Proof of  \cref{t:firstk}]
Fix $x,y\in\R^d$, $0\le S< T$. Motivated by the above discussion, let us apply \cref{T:SSLst} to the process $(\A_t)_{t\in[0,T]}$ defined in  \eqref{ourint} and the germ
\begin{equation}\label{germ}
A_{s,t}:=\E^s \int_s^t (b(W_r+x)-b(W_r+y))\,dr,\quad (s,t)\in\Delta_{[S,T]}.
\end{equation}

Let us verify that all the conditions of the SSL are satisfied. 
First, we check \eqref{con:s2} to gain an intuition why the SSL is useful. Recalling the notation \eqref{deltadiff}, we have for $(s,u,t)\in\Delta^3_{[S,T]}$
\begin{equation*}
\delta A_{s,u, t}=(\E^s-\E^u) \int_u^t (b(W_r+x)-b(W_r+y))\,dr.
\end{equation*}
This implies $\E^s \delta A_{s,u, t}=0$ and \eqref{con:s2} holds. Next, using \eqref{1bound}, we get 
\begin{align*}
|A_{s,t}|&\le \int_s^t \bigl|P_{r-s}b(W_s+x)-P_{r-s}b(W_s+y)\bigr|\,dr\le 
C|x-y|\int_s^t \|P_{r-s}b\|_{\C^1}\,dr\\
&\le C\|b\|_{\C^\gamma} |x-y|(t-s)^{\frac12+\frac\gamma2},
\end{align*}
where $C=C(\gamma,d)>0$ and the last inequality follows from \cref{l:gb}. This obviously implies 
\begin{equation}\label{astbound}
\|A_{s,t}\|_{L_m(\Omega)}\le C\|b\|_{\C^\gamma} |x-y|(t-s)^{\frac12+\frac\gamma2}
\end{equation}
and since $\gamma>0$, then condition \eqref{con:s1} is satisfied (recall \cref{r:26}). 

Condition (ii) of \cref{T:SSLst} is satisfied by the definition of $A_{s,t}$ in \eqref{germ}. 
Finally, let us check condition \eqref{con:s3}. For any partition $\Pi=\{S=t_0<t_1<\hdots< t_k=T\}$ of $[S,T]$ we
use the orthogonality of martingale differences in $L_2(\Omega)$ to obtain
\begin{align*}
\Bigl\|\A_T-\A_S-\sum_{i=0}^{k-1} A_{t_i,t_{i+1}}\Bigr\|_{L_2(\Omega)}^2&= 
\Bigl\|\sum_{i=0}^{k-1} (\A_{t_{i+1}}-\A_{t_i}-\E^{t_i} (\A_{t_{i+1}}-\A_{t_i}))\Bigr\|_{L_2(\Omega)}^2\\
&=\sum_{i=0}^{k-1} \|\A_{t_{i+1}}-\A_{t_i}-\E^{t_i} (\A_{t_{i+1}}-\A_{t_i})\|_{L_2(\Omega)}^2\\
&\le C\sum_{i=0}^{k-1} \|b\|_{\C^0}^2(t_{i+1}-t_i)^2\le C\|b\|_{\C^0}^2|\Pi||T-S|\to0,
\end{align*}	
as $|\Pi|\to0$. Thus, condition \eqref{con:s3} is also satisfied.

Thus, all the assumptions of the SSL hold, and taking into account \eqref{astbound}, we obtain from \eqref{est:ssl1}
\begin{equation*}
	\Bigl\|\int_S^T (b(W_r+x)-b(W_r+y))\,dr\Bigr\|_{L_m(\Omega)}\le C\|b\|_{\C^\gamma} |x-y||T-S|^{\frac12+\frac\gamma2},
\end{equation*}	
where $C=C(\gamma,d,m)>0$, which is \eqref{firstkey}.
\end{proof}

\begin{remark}
Now we see why the stochastic sewing lemma is more effective than the deterministic sewing lemma for bounding integral functionals of Brownian motion. Indeed, \eqref{astbound} implies   $\|A_{s,t}\|_{L_m(\Omega)}\le C(t-s)^{\frac12+\frac\gamma2}$. Therefore, to apply the SSL and verify condition \eqref{con:s1}, we only need to require $\gamma > 0$. By contrast, the corresponding condition for the deterministic sewing lemma (condition \eqref{notgood}) fails unless $\gamma > 1$, so the usual sewing lemma is absolutely useless  here. Note also that the second SSL condition, \eqref{con:s2},   is totally harmless here  since
$\E^s \delta A_{s,u,t}\equiv0$.
\end{remark}

\subsection{Step 2: Bounding integrals along a perturbed Brownian path}\label{sect:32}
Our next step is to allow $x$ and $y$ in \eqref{firstkey} to be random and time-dependent. However, we cannot simply take arbitrary stochastic processes $(x_r)_{r\in[0,1]}$, $(y_r)_{r\in[0,1]}$ and expect \eqref{firstkey} to remain valid. For example, setting in \eqref{firstkey} $x_r := -W_r+z$ and $y_r := -W_r$, where $z\in\R^d$, would lead to the estimate $|b(z) - b(0)| \le C |b|_{\C^\gamma} |z|$, which is clearly false for a generic $b\in\C^\gamma$. We must therefore assume that the perturbation is smoother (in a certain sense) than the noise. 

The next question is how we want to capture the smoothness of the perturbation. Of course, a natural approach would be to require that its $\C^\tau([0,1])$ norm is a.s. finite for some $\tau>\frac12$ (recall that $W \in \C^{1/2 - \varepsilon}$ for any $\varepsilon > 0$). However, since we aim to prove strong uniqueness via a Gronwall-type argument, it is much more convenient to first take the $L_m(\Omega)$-norm of the increment of the perturbation ($m\ge2$), and then measure its $\C^\tau([0,1])$ norm. This motivates the following definitions.

For $m\ge1$, $(S,T)\in\Delta_{[0,1]}$, a measurable function $f\colon\Omega\times[S,T]\to\R^d$, $\tau\in(0,1]$ put
\begin{equation}\label{norms}
	\|f\|_{\C^0L_m([S,T])}:=\sup_{t\in[S,T]}\|f(t)\|_{L_m(\Omega)};\,\,\,\,\, [f]_{\C^\tau L_m([S,T])}:=\sup_{(s,t)\in\Delta_{[S,T]}} \frac{\|f(t)-f(s)\|_{L_m(\Omega)}}{|t-s|^\tau}.\\
\end{equation}
It follows from the definition that for any $\tau\in(0,1]$ we have
\begin{equation}\label{norm2normgen}
\|f\|_{\C^0L_m([S,T])}\le \|f_S\|_{L_m(\Omega)}+\sup_{t\in[S,T]}\|f(t)-f(S)\|_{L_m(\Omega)}\le \|f_S\|_{L_m(\Omega)}+(T-S)^\tau[f]_{\C^\tau L_m([S,T])}.
\end{equation}

The following is a key bound for uniqueness.
\begin{theorem}\label{e:step1uni}
Let $m\ge2$, $\gamma,\tau\in(0,1]$. Let $b\in\C^\gamma(\R^d,\R)$. Let $\phi,\psi\colon\Omega\times[0,1]\to\R^d$ be continuous measurable functions adapted to $(\F_t)_{t\in[0,1]}$. Assume that 
\begin{equation}\label{gammatau}
\frac\gamma2+\tau>\frac12.
\end{equation}
Suppose further that
there exists a constant $\Gamma_0\ge0$ such that a.s. for all $(s,t)\in\Delta_{[0,1]}$
\begin{equation}\label{gamma0cond}
|\phi_t-\phi_s|\le \Gamma_0|t-s|.
\end{equation}
 Then  there exists a constant $C=C(\gamma,\tau,d,m)>0$ such that 
for any  $(S,T)\in\Delta_{[0,1]}$ one has
\begin{align}\label{seckey}
&\Bigl \| \int_S^T \bigl(b(W_r+\phi_r)-b(W_r+\psi_r)\bigr)\,dr \Bigr\|_{L_m(\Omega)}\nn\\
&\,\,\,\le C\|b\|_{\C^\gamma}(T-S)^{\frac12+\frac\gamma2}\bigl((1+\Gamma_0)\|\phi-\psi\|_{\C^0L_m([S,T])}
+ [\phi-\psi]_{\C^\tau L_m([S,T])} (T-S)^{\tau}\bigr).
\end{align}
\end{theorem}	

\cref{e:step1uni} brings us back to the discussion in the introduction, where we explained that to bound an integral functional of a stochastic process, we decompose the process into the sum of a very irregular noise about which we know everything (in this case, $W$) and a generic perturbation about which we do not know much, apart from the fact that it is much smoother than the noise, see condition \eqref{gamma0cond} and note the appearance of the term $[\phi-\psi]_{\C^\tau L_m([S,T])}$ on the right-hand side of \eqref{seckey}.

Before we proceed to the proof of \cref{e:step1uni}, we need to prepare ourselves with the following simple  technical lemma.
\begin{lemma}\label{l:l29}
Let $f\in\C^\infty(\R^d)$, $x_1,x_2,x_3,x_4\in\R^d$. Then 
\begin{equation*}
|f(x_1)-f(x_2)-f(x_3)+f(x_4)|\le  \|\nabla f \|_{\C^1}|x_2-x_1||x_3-x_1|+\|f\|_{\C^1}|x_1-x_2-x_3+x_4|.
\end{equation*}	
\end{lemma}
\begin{proof}
We have for $x,y,z\in\R^d$
\begin{align*}
|f(x)-f(x+z)-f(y)+f(y+z)|&=\Bigl|\int_0^1 z^\t (\nabla f(x+\theta z)-\nabla f(y+\theta z))\,d\theta \Bigr|\\
&\le |z||x-y| \|\nabla f \|_{\C^1}.
\end{align*}	
Therefore
\begin{align*}
|f(x_1)-f(x_2)-f(x_3)+f(x_4)|&\le |f(x_1)-f(x_2)-f(x_3)+f(x_3+x_2-x_1)|\\
&\phantom{\le}+|f(x_4)-f(x_3+x_2-x_1)|\\
&\le  \|\nabla f \|_{\C^1}|x_2-x_1||x_3-x_1|+\|f\|_{\C^1}|x_1-x_2-x_3+x_4|.\qedhere
\end{align*}

\end{proof}

Now we can prove \cref{e:step1uni}.
\begin{proof}[Proof of \cref{e:step1uni}]
Fix $m\ge2$, $0\le S< T$. Let us verify that all the conditions of the SSL are satisfied for the germ 
\begin{equation}\label{germper}
A_{s,t}:=\int_s^t \E^s(b(W_r+\phi_s)-b(W_r+\psi_s))\,dr,\quad (s,t)\in\Delta_{[S,T]}
\end{equation}	
and the process 
\begin{equation*}
\A_t:=\int_S^t \bigl(b(W_r+\phi_r)-b(W_r+\psi_r)\bigr)\,dr,\quad t\in[S,T].
\end{equation*}

Note that we used $W_r + \phi_s$ as the argument of $b$ in the germ \eqref{germper}, rather than $W_r + \phi_r$: this choice exactly matches our ideology. We do not know anything about the conditional law $\law(W_r + \phi_r \mid \F_s)$, whereas
$\law(W_r + \phi_s \mid \F_s)$ is Gaussian  with mean $W_s+\phi_s$ and covariance matrix $(r-s)\Im_d$. This is due to the fact that $W$ is an $(\F_t)$-Brownian motion and $\phi_s$ is $\F_s$-measurable.

Another important detail of the proof is that we first verify all the conditions of the SSL for a smooth function $b$, and then use Fatou's lemma to pass to an arbitrary $b$. This will significantly simplify the verification of condition (i) of the SSL and will also be useful for future purposes, when the drift $b$ is assumed to be only bounded.

\textbf{Step 1}. Assume additionally that $b\in\C^\infty$ (but of course we would like to bound the germ only in terms of $\|b\|_{\C^\gamma}$).  Arguing exactly as in the proof of \cref{t:firstk}, we get
\begin{align*}
	|A_{s,t}|&\le \int_s^t \bigl|P_{r-s}b(W_s+\phi_s)-P_{r-s}b(W_s+\psi_s)\bigr|\,dr\le 
	|\phi_s-\psi_s|\int_s^t \|P_{r-s}b\|_{\C^1}\,dr\\
	&\le C\|b\|_{\C^\gamma}|\phi_s-\psi_s|(t-s)^{\frac12+\frac\gamma2},
\end{align*}
for $C=C(\gamma,d)>0$,
where we consecutively used identity \eqref{1bound} and \cref{l:gb}.
Recalling the definitions of norms in \eqref{norms}, we get 
\begin{equation}\label{cond1pert}
\|A_{s,t}\|_{L_m(\Omega)}\le  C\|b\|_{\C^\gamma}\|\phi-\psi\|_{\C^0L_m([S,T])}(t-s)^{\frac12+\frac\gamma2}
\end{equation}	
for $C=C(\gamma,d)>0$,
and thus condition \eqref{con:s1} holds.

Next, let us verify condition \eqref{con:s2}. This time it is a bit trickier, as $\E^s\delta A_{s,u,t}$ is not zero. We have for $(s,u,t)\in\Delta^3_{[S,T]}$
\begin{equation*}
\delta A_{s,u,t}=\E^s \int_u^t (b(W_r+\phi_s)-b(W_r+\psi_s))\,dr-
\E^u \int_u^t (b(W_r+\phi_u)-b(W_r+\psi_u))\,dr.
\end{equation*}	
Therefore, 
\begin{align}\label{step1}
&\E^s \delta A_{s,u,t}\nn\\
&\,\,=\int_u^t \E^s\E^u  (b(W_r+\phi_s)-b(W_r+\psi_s)-b(W_r+\phi_u)+b(W_r+\psi_u))\,dr\nn\\
&\,\,=\E^s\!\! \int_u^t (P_{r-u}b(W_u+\phi_s)-P_{r-u}b(W_u+\psi_s)-P_{r-u}b(W_u+\phi_u)+P_{r-u}b(W_u+\psi_u))\,dr.
\end{align}
Using \cref{l:l29}, we get
\begin{align}\label{crucial}
&|P_{r-u}b(W_u+\phi_s)-P_{r-u}b(W_u+\psi_s)-P_{r-u}b(W_u+\phi_u)+P_{r-u}b(W_u+\psi_u)|\nn\\
&\quad \le\|\nabla P_{r-u}b\|_{\C^1}|\phi_s-\psi_s|\,|\phi_u-\phi_s|+
\| P_{r-u}b\|_{\C^1}|(\phi_u-\psi_u)-(\phi_s-\psi_s)|.
\end{align}
Combining this with \eqref{step1} and recalling heat kernel bounds in \cref{l:gb}, we  get 
\begin{align}\label{verydangerous}
\|\E^s \delta A_{s,u,t}\|_{L_m(\Omega)}&\le C\Gamma_0\|b\|_{\C^\gamma}\int_u^t (r-u)^{\frac{\gamma-2}2}\|\phi-\psi\|_{\C^0L_m([S,T])}|u-s|\,dr\\
&\phantom{\le}+C\|b\|_{\C^\gamma}\int_u^t(r-u)^{\frac\gamma2-\frac12}[\phi-\psi]_{\C^\tau L_m([S,T])}|u-s|^\tau\,dr\nn\\
&\le C \Gamma_0\|b\|_{\C^\gamma}  (t-s)^{\frac\gamma2+1}\|\phi-\psi\|_{\C^0L_m([S,T])}\nn\\
&\phantom{\le}+C\|b\|_{\C^\gamma}  (t-s)^{\frac\gamma2+\frac12+\tau}[\phi-\psi]_{\C^\tau L_m([S,T])}.\label{cond2pert}
\end{align}
Here in inequality \eqref{verydangerous} we first applied \eqref{crucial} and assumption  \eqref{gamma0cond}, and then used  Minkowski's inequality to bring the $\|\cdot\|_{L_m(\Omega)}$-norm inside the time integral.
Since $\gamma>0$ and $\frac\gamma2+\tau>\frac12$, we see that condition \eqref{con:s2} is satisfied. 

Condition (ii) of \cref{T:SSLst} is satisfied by the construction of the germ in \eqref{germper}. Finally, to verify condition (i) of the SSL, let  $\Pi:=\{S=t_0,t_1,...,t_k=T\}$ be an arbitrary partition of $[S,T]$.
Note that for any $i=0,\hdots,k-1$
\begin{align*}
\|\E^{t_i}(\A_{t_{i+1}}-\A_{t_i})-A_{t_i,t_{i+1}}\|_{L_2(\Omega)}&\le 
\int_{t_i}^{t_{i+1}} \|b(W_r+\phi_r)-b(W_r+\phi_{t_i})\|_{L_2(\Omega)}\,dr\\
&\phantom{\le}+\int_{t_i}^{t_{i+1}} \|b(W_r+\psi_r)-b(W_r+\psi_{t_i})\|_{L_2(\Omega)}\,dr\\
&\le \|b\|_{\C^1}(t_{i+1}-t_i)^{1+\tau}([\phi]_{\C^\tau L_2([S,T])}+[\psi]_{\C^\tau L_2([S,T])}).
\end{align*}
Therefore, using the orthogonality of martingale differences in $L_2(\Omega)$, we get 
\begin{align*}
\Bigl\|\A_{T}-\A_{S}-\sum_{i=0}^{k-1} A_{t_i,t_{i+1}}\Bigr\|_{L_2(\Omega)}^2
&\le C\Bigl\|\sum_{i=0}^{k-1} (\A_{t_{i+1}}-\A_{t_i}-\E^{t_i}(\A_{t_{i+1}}-\A_{t_i}))\Bigr\|_{L_2(\Omega)}^2\\
&\phantom{\le}+C\Bigl\|\sum_{i=0}^{k-1} (\E^{t_i}(\A_{t_{i+1}}-\A_{t_i})-A_{t_i,t_{i+1}})\Bigr\|_{L_2(\Omega)}^2\\
&\le C\sum_{i=0}^{k-1} \|\A_{t_{i+1}}-\A_{t_i}-\E^{t_i}(\A_{t_{i+1}}-\A_{t_i})\|_{L_2(\Omega)}^2\\
&\phantom{\le} +C\|b\|_{\C^1}^2 ([\phi]_{\C^\tau L_2([S,T])}^2+[\psi]_{\C^\tau L_2([S,T])}^2)(\sum_{i=0}^{k-1}(t_{i+1}-t_i)^{1+\tau})^2\\
	&\le  C\|b\|_{\C^1}^2|\Pi|^{2\tau\wedge1} (1+[\phi]^2_{\C^\tau L_2([S,T])}+[\psi]_{\C^\tau L_2([S,T])}^2)\to0.
\end{align*}
for $C>0$ as $|\Pi|\to0$. Thus, condition \eqref{con:s3} also holds. Therefore, all the conditions of SSL are satisfied and \eqref{est:ssl1} together with \eqref{cond1pert} and \eqref{cond2pert} implies \eqref{seckey}.

\textbf{Step 2}. For a general $b\in\C^\gamma$ we put $b_n:=P_{\frac1n}b$. Then obviously $b_n(x)\to b(x)$ for any $x\in\R^d$, $\|b_n\|_{\C^\gamma}\le \|b\|_{\C^\gamma}$ and $b_n\in\C^\infty$. Thus, Fatou's lemma and Step 1 of the proof yield
\begin{align*}
&\Bigl\|\int_S^T (b(W_r+\phi_r)-b(W_r+\psi_r))\,dr\Bigr\|_{L_m(\Omega)}\\
&\quad\le\limsup_{n\to\infty}	\Bigl\|\int_S^T (b_n(W_r+\phi_r)-b_n(W_r+\psi_r))\,dr\Bigr\|_{L_m(\Omega)}\\
	&\quad\le C\|b\|_{\C^\gamma}(T-S)^{\frac12+\frac\gamma2}\bigl((1+\Gamma_0)\|\phi-\psi\|_{\C^0L_m([S,T])}
	+ [\phi-\psi]_{\C^\tau L_m([S,T])} (T-S)^{\tau}\bigr),
\end{align*}
which is \eqref{seckey}.
\end{proof}

We see that, in general, we have not required much from the driving process $W$; in particular, no It\^o's formula was needed. However, let us point out some parts of the proof that are ``fragile'' and would need to be replaced in other settings.

\begin{remark}\label{r:d}
Let us look again at the crucial inequality \eqref{crucial}. When we substitute \eqref{crucial} into \eqref{step1}, we need to bound
$$
\Bigl\|\, \E^s|\phi_s-\psi_s|\,|\phi_u-\phi_s|\,\Bigr\|_{L_m(\Omega)}
$$
in terms of $\|\phi_s-\psi_s\|_{L_m(\Omega)}$. Note that H\"older's inequality won't help here, because the resulting bound would involve the $(m+\eps)$-moment of $|\phi_s-\psi_s|$, and we would not be able to close the buckling argument; see inequality \eqref{prefin} below. Therefore, we had to assume an a.s. bound on $|\phi_u-\phi_s|$; see \eqref{gamma0cond}. Of course, if $X=W+\phi$ is a solution to the SDE \eqref{mainSDE} with bounded drift, then this condition is automatically satisfied. However, for less regular drifts (not locally bounded, measures, or generalized functions), this becomes a problem; this will be further discussed in \cref{sec:44}. The same issue also appears for rough SDEs, see \cite{FHL}.
\end{remark}

\begin{remark}\label{r:vd}
Another very dangerous inequality is \eqref{verydangerous}. Here we were very lucky: we had to ensure that $(r-u)^{\frac\gamma2-1}|u-s|$ is integrable in $r$, and it barely is, since $\gamma>0$.  In other situations, however, this is not the case. Note that the factor $(u-s)$ does not help to improve integrability.  This problem was resolved in \cite{Gerreg22}, where a different version of stochastic sewing was developed. We will look at  this in more detail in \cref{s:levywp}, where we consider SDEs driven by $\alpha$-stable processes.
\end{remark}

\subsection{Proof of pathwise uniqueness}\label{s:SU}
We are now  ready to prove our first main result: pathwise uniqueness of solutions to the SDE \eqref{mainSDE} via the SSL, following the plan outlined in \cref{sect:11}. 

Let $(X_t)_{t\in[0,1]}$, $(Y_t)_{t\in[0,1]}$ be two solutions to SDE \eqref{mainSDE} adapted to the filtration $(\F_t)_{t\in[0,1]}$ and with the same initial condition $X_0=Y_0=x\in\R^d$. Define for $t\in[0,1]$
\begin{equation}\label{phipsi}
	\phi_t:=X_t-W_t;\quad	\psi_t:=Y_t-W_t.
\end{equation}	

The main idea is to bound  $\|X-Y\|_{\C^0L_2([0,1])}=\|\phi-\psi\|_{\C^0L_2([0,1])}$ by half of the same quantity, which would immediately imply that $X_t=Y_t$ a.s. for all $t\in[0,1]$, and hence strong uniqueness. However, doing this directly seems to be not possible. Indeed, when we bound $\|\phi-\psi\|_{\C^0L_2([0,1])}$ via \cref{e:step1uni}, a higher norm, namely $[\phi-\psi]_{\C^{1/2}L_2([0,1])}$, appears. Therefore, we modify our strategy and aim to show that  $[\phi-\psi]_{\C^{1/2}L_2([0,T])}$ can be bounded by half of itself, at least for small enough $T>0$. Since $X_0=Y_0$, this still yields $X_t=Y_t$ a.s. for all $t\in[0,T]$. Iterating the argument then gives pathwise uniqueness on the whole interval $[0,1]$.

\begin{proof}[Proof of \cref{t:VK}]
To simplify the argument, we prove \cref{t:VK} under slightly more restrictive assumptions: namely, we assume that the drift $b \in \C^\gamma$ for some $\gamma > 0$. For the stochastic sewing proof of \cref{t:VK} in full generality, we refer to \cite{le2021taming}.

\textit{Pathwise uniqueness}. 
Let us apply \cref{e:step1uni} to the processes $\phi$, $\psi$ defined above with $\tau=\frac12$, $m=2$. We note that $|\phi_t-\phi_s|\le \|b\|_{\C^0}|t-s|$ for any $s,t\in[0,1]$ and thus condition \eqref{gamma0cond} is satisfied. Condition \eqref{gammatau} also holds since $\tau=\frac12$ and $\gamma>0$. Then all conditions of \cref{e:step1uni} are satisfied and it follows from \eqref{seckey} that there exists $C=C(\gamma,d)$ such that for any $(s,t)\in\Delta_{[0,1]}$
\begin{align*}
\|X_t-X_s-(Y_t-Y_s)\|_{L_2(\Omega)}&=\Bigl \| \int_s^t \bigl(b(W_r+\phi_r)-b(W_r+\psi_r)\bigr)\,dr \Bigr\|_{L_2(\Omega)}\nn\\
&\le
C\|b\|_{\C^\gamma}(1+\|b\|_{\C^0})(t-s)^{\frac12+\frac\gamma2}\|\phi-\psi\|_{\C^0L_2([s,t])}\\
&\phantom{\le}+  C\|b\|_{\C^\gamma}(t-s)^{1+\frac\gamma2}[\phi-\psi]_{\C^{\frac12} L_2([s,t])}.
\end{align*}	
Let $T\in(0,1)$. Dividing both sides of the above inequality by $(t-s)^{\frac12}$ and taking the supremum over $(s,t)\in\Delta_{[0,T]}$, we obtain (note that $X-Y=\phi-\psi$)
\begin{equation}\label{prefin}
[\phi-\psi]_{\C^{\frac12} L_2([0,T])}\le 
C\|b\|_{\C^\gamma}(1+\|b\|_{\C^0})\|\phi-\psi\|_{\C^0L_2([0,T])}+  C\|b\|_{\C^\gamma}T^{\frac12}[\phi-\psi]_{\C^{\frac12} L_2([0,T])}.
\end{equation}
Note that by definition 
\begin{equation}\label{norm2norm}
	\|\phi-\psi\|_{\C^0L_2([0,T])}\le \|\phi_0-\psi_0\|_{L_2(\Omega)}+T^{\frac12}[\phi-\psi]_{\C^{\frac12} L_2([0,T])}=T^{\frac12}[\phi-\psi]_{\C^{\frac12} L_2([0,T])}.
\end{equation}
Hence we get from \eqref{prefin}
\begin{equation*}
	[\phi-\psi]_{\C^{\frac12} L_2([0,T])}\le 
	C\|b\|_{\C^\gamma}(1+\|b\|_{\C^0})T^{\frac12}[\phi-\psi]_{\C^{\frac12} L_2([0,T])},
\end{equation*}
for $C=C(\gamma,d)$.
Pick  now $T_0$ small enough so that 
\begin{equation*}
	C\|b\|_{\C^\gamma}(1+\|b\|_{\C^0})T_0^{\frac12}\le \frac12.
\end{equation*}
Then we deduce from \eqref{prefin}
\begin{equation*}
[\phi-\psi]_{\C^{\frac12} L_2([0,T_0])}=0
\end{equation*}
and thus by \eqref{norm2norm}, $X_t=Y_t$ a.s. for $t\in[0,T_0]$. Repeating this argument $\lceil1/T_0\rceil$ times we get that $X_t=Y_t$ a.s. for any  $t\in[0,1]$ and thus pathwise uniqueness holds.

\textit{Strong existence} of solutions to \eqref{mainSDE} now follows easily. Indeed, weak existence is immediate by the Girsanov theorem. By the Yamada–Watanabe theorem (see, e.g., \cite[Corollary 5.3.23]{KS91}), weak existence and strong uniqueness for \eqref{mainSDE} imply strong existence.
\end{proof}

\begin{remark}
Note that the proof given above is quite robust. The only properties of $W$ we used are the Gaussian bounds \eqref{ineq} and the identity \eqref{1bound}. Therefore, the same general idea can be applied in many other settings. In particular, \cref{e:strongfbm} extends the uniqueness result to SDEs driven by fractional Brownian motion. 
\end{remark}

\begin{remark}
Note that even though the drift $b$ is bounded and thus the functions $\phi,\psi$ defined in \eqref{phipsi} belong to $\C^1L_\infty$, we establish the contraction estimate \eqref{prefin} in a much weaker norm, namely $\C^{\frac12}L_m$. The stronger regularity $\phi,\psi\in\C^1L_\infty$ is used only to verify condition \eqref{gamma0cond}. This is a typical situation in proofs of results of this type: to show contraction one needs to consider much weaker norms.
\end{remark}

\begin{remark}
We used Girsanov’s theorem in the proof of strong existence just to simplify the argument and focus on pathwise uniqueness. In general, this trick might not work, but there are at least two alternative approaches. One can upgrade estimates \eqref{prefin} and \eqref{norm2norm} to obtain a full fixed point argument; see \cite[Proof of Theorem~4.6]{FHL} for rough SDEs and \cref{s:levywp} below for a related argument. Alternatively, one can approximate the drift by smooth functions and use a priori estimates and tightness to construct a weak solution. Weak existence and pathwise uniqueness then give strong existence; see \cref{s:we} for an argument along these lines.
\end{remark}
\subsection{Exercises}\label{s:upr3}
In order to get used to the $\C^\tau L_m$ seminorms introduced in \eqref{norms}, let us prove their lower semicontinuity, which will also be needed in the future.
\begin{exercise}\label{e:lscnorm}
Let $m\in[1,\infty)$, $\tau\in(0,1]$, and let $\psi,\psi^n\colon\Omega\times [0,1]\to\R^d$,  $n\in\Z_+$, be measurable processes. 	Suppose that for every $(s,t)\in\Delta_{[0,1]}$, the increments $\psi_t^n-\psi_s^n$ converge in law to $\psi_t-\psi_s$.
Show that
\begin{equation*}
[\psi]_{\C^\tau L_m([0,1])}\le \liminf_{n\to\infty}[\psi^n]_{\C^\tau L_m([0,1])}.
\end{equation*}
	
\hn Use the Portmanteau theorem for each fixed pair of times $(s,t)\in\Delta_{[0,1]}$.
\end{exercise}

The next block of exercises  extends \cref{t:VK} to SDEs driven by a fractional Brownian motion (fBM). First, let us recall its definition and basic properties.

\begin{definition}
A process $X\colon\Omega\times[0,1]\to\R$ is called \textit{Gaussian} if for any $n\in\N$ and any 
$t_1, t_2,\hdots, t_n\in[0,1]$, the vector $(X_{t_1}, X_{t_2}, \dots, X_{t_n})$ has a Gaussian distribution.
\end{definition}

\begin{definition}\label{e:fbm}
Let $H\in(0,1)$. A Gaussian process $B^H\colon\Omega\times[0,1]\to\R$ is called a \textit{fractional Brownian motion with Hurst parameter $H$} if it is 	continuous and for any $(s,t)\in\Delta_{[0,1]}$ one has $\E B^H_t=0$, $\E B^H_sB^H_t=\frac12(t^{2H}+s^{2H}-(t-s)^{2H})$.
\end{definition}	
If $H=\frac12$, then an fBM is a Brownian motion. For $H\neq\frac12$, an fBM is not a Markov process and not a semimartingale (see, e.g., \cite[Section 2]{Nourdin}). For any $\eps>0$, almost all  trajectories of fBM $B^H$ are H\"older continuous functions with exponent $H-\eps$, that is $\P(B^H\in\C^{H-\eps}([0,1];\R))=1$. We can see from the following pictures what typical trajectories of $B^H$ look like for  various $H\in(0,1)$.
\begin{figure}[h]
\centering
\subfloat[\centering $H=0.15$]{{\includegraphics[width=4.63cm]{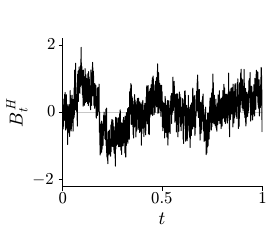} }}%
\qquad
\subfloat[\centering $H=0.55$]{{\includegraphics[width=4.63cm]{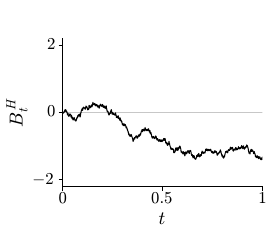} }}%
\qquad
\subfloat[\centering $H=0.95$]{{\includegraphics[width=4.63cm]{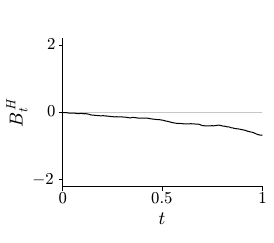} }}
\caption{Typical trajectories of fractional Brownian motion for various values of  $H$.}%
\label{fig:example}%
\end{figure}

A $d$-dimensional fBM with Hurst parameter $H\in(0,1)$ ($d\in\N$) is simply a vector consisting of $d$ independent fBMs with Hurst parameter $H$. We recall from \cite[Section~5.1.3, formula~(5.8) and Proposition~5.1.3]{Nu} that given a $d$-dimensional fBM $B^H$ one can construct on the same probability space a standard $d$-dimensional BM $W$ such that
\begin{equation}\label{WB}
	B^H_t=\int_0^t K_H(t,s)\,dW_s,
\end{equation}
for a certain function $K_H\colon[0,1]^2\to \R_+$ satisfying for any $(s,t)\in\Delta_{[0,1]}$ (see, e.g., \cite[Proposition B.2(ii)]{BLM23})
\begin{equation}\label{KBound}
K_H(t,s)\ge C(t-s)^{H-\frac12},
\end{equation}
for $C=C(H)$.

We fix now $H\in(0,1)$, a $d$-dimensional fBM $B^H$, and let $(\F_t)_{t\in[0,1]}$ be the filtration of the underlying BM $W$ from representation \eqref{WB}.
\begin{exercise}\label{e311}
\begin{enumerate}[(i)] 
	\item Generalize formula \eqref{1bound} and show that  for any bounded measurable $f\colon\R^d\to\R$, $0\le s< t$ we have
	\begin{equation}\label{1boundg}
		\E^s f(B^H_t)=P_{\sigma^2(s,t)}f (\E^s B^H_t),
	\end{equation}
	for $\sigma^2(s,t):=\int_s^t K_H(t,r)^2\,dr$.
	\item Adapt the proof of \cref{t:firstk} to the fractional Brownian case. Namely, given $m\ge2$, $\gamma\in(0,1]$, $b\in \C^\gamma(\R^d,\R)$, $\rho\in[0,1]$, $\rho<\gamma+\frac1{2H}$, show that 
	 there exists a constant $C=C(\gamma,\rho,d,H,m)>0$ such that for any $x,y\in\R^d$, $(S,T)\in\Delta_{[0,1]}$ one has 
	\begin{equation*}
		\Bigl\|\int_S^T (b(B^H_r+x)-b(B^H_r+y))\,dr\Bigr\|_{L_m(\Omega)}\le C\|b\|_{\C^\gamma} |x-y|^{\rho}|T-S|^{1+(\gamma-\rho)H}.
	\end{equation*}
	\hn Use identity \eqref{1boundg} and inequality \eqref{KBound}.
\end{enumerate}	
\end{exercise}

To solve the next exercise, we will need Kolmogorov's continuity theorem. It allows us to exchange the order of taking the supremum and the expectation in norms of the form \eqref{norms}. That is, 
it gives a bound on the moments of the H\"older norm of a stochastic process if the moments of its increments are H\"older-bounded.

\begin{theorem}[Kolmogorov's continuity theorem, see, e.g., {\cite[Theorem~A.11]{FV2010}},
{\cite[Theorem~1.4.1]{Kunita}}]
Let $(E,d)$ be a Polish space.  \label{t:kolmi} Let $\gamma\in(0,1]$, $\Gamma\ge0$, $k\in\N$, $m>\frac{k}\gamma$, $T>0$. Let $X\colon\Omega\times[0,T]^k\to E$ be a continuous process such that for any $s,t\in[0,T]^k$
\begin{equation*}
\|d(X_s,X_t)\|_{L_m(\Omega)}\le \Gamma |t-s|^\gamma.
\end{equation*}
Then  for any $\beta\in(0,\gamma-\frac{k}m)$ there exists a constant $C=C(\beta,\gamma,k,m,T)>0$
such that
\begin{equation*}
\|[X]_{\C^{\beta}([0,T]^k;E)}\|_{L_m(\Omega)}\le C \Gamma.
\end{equation*}
\end{theorem}

\begin{exercise}\label{e:313}
Let 	$\gamma\in(0,1]$, $b\in \C^\gamma(\R^d,\R)$, $H\in(0,1)$.  
Use  Kolmogorov's continuity theorem and \cref{e311}  to deduce that for every $N>0$ the process 
\begin{equation*}
x\mapsto\int_0^1 b(B_r^H+x)\,dr,	\quad x\in\R^d
\end{equation*}
is a.s. in $\C^{(\gamma+\frac1{2H}-\eps)\wedge(1-\eps)}([-N,N]^d,\R)$ for any $\eps>0$.
\end{exercise}

Next, we extend \cref{e:step1uni} to the fractional Brownian case. We consider only the setting where no extensions of stochastic sewing are needed and therefore restrict ourselves to the case $H<\frac12$.

\begin{exercise}\label{e:ext}
Let $H\in(0,\frac12)$, $m\ge2$, $\tau\in(H,1]$. Let $b$ be a bounded measurable function $\R^d\to\R$. Let $\psi,\phi\colon\Omega\times[0,1]\to\R^d$ be continuous measurable functions adapted to $(\F_t)_{t\in[0,1]}$. 	Suppose that
there exists a constant $\Gamma_0\ge0$ such that a.s. for all $(s,t)\in\Delta_{[0,1]}$
\begin{equation*}
		|\phi_t-\phi_s|\le \Gamma_0|t-s|.
\end{equation*}
Show that  there exists a constant $C=C(\tau,d,H,m)>0$ such that 
for any  $(s,t)\in\Delta_{[0,1]}$ one has
	\begin{align*}
		&\Bigl \| \int_s^t \bigl(b(B^H_r+\phi_r)-b(B^H_r+\psi_r)\bigr)\,dr \Bigr\|_{L_m(\Omega)}\nn\\
		&\quad\le C\|b\|_{\C^0}(t-s)^{1-H}\bigl((1+\Gamma_0)\|\phi-\psi\|_{\C^0L_m([s,t])}
		+ [\phi-\psi]_{\C^\tau L_m([s,t])} (t-s)^{\tau}\bigr).
	\end{align*}
\end{exercise}	
 
\begin{exercise}[Catellier-Gubinelli theorem, {\cite[Theorem~1.9]{CG16}}]\label{e:strongfbm}
Let $H\in(0,\frac12)$, let $b$ be a bounded continuous function $\R^d\to\R^d$. Use \cref{e:ext} to show that the solution  to SDE
\begin{equation*}
dX_t=b(X_t)dt+dB_t^H,\quad X_0=x,
\end{equation*}	
where $x\in\R^d$, is pathwise unique.
\end{exercise}	

Note that the statement of \cref{e:strongfbm} actually holds in much greater generality, namely for $b\in\C^\gamma$, $\gamma>1-\frac1{2H}$ and $H\in(0,\infty)\setminus\N$; see \cite{CG16,Gerreg22} (there is a way to define an fBM for $H>1$). The proof for the case $H<\frac12$, $\gamma < 0$ can be carried out using the same stochastic sewing technique, but it requires familiarity with the basic theory of Besov spaces. We will return to this later in \cref{s:WU}.
In contrast, if $H>\frac12$, the stochastic sewing lemma in the form of \cref{T:SSLst} leads to a suboptimal result (recall \cref{r:vd}). To obtain the optimal condition, \cite{Gerreg22} introduced a shifted stochastic sewing lemma. We will look at this extension in \cref{s:sssl}.

\section{Weak well-posedness of SDEs driven by Brownian motion}\label{s:WU}

The main goal of this section is to show how stochastic sewing techniques can be combined with ideas from ergodic theory to obtain \textbf{weak} uniqueness (uniqueness in law). In this section we follow \cite{BM24}. As in the previous section, to ease the understanding of the main ideas, we deliberately concentrate on the simplest model: SDEs driven by a Brownian motion. However, the same strategy can be successfully applied to SDEs driven by a fractional Brownian motion or to SPDEs; see \cite{BM24}. 
Thus, in this section we prove \cref{t:weak} without relying on the Zvonkin transformation argument.

First, in \cref{s:41} we recall the notion of Besov spaces of negative regularity and define solutions to SDEs with distributional drift and we review various notions of distance between probability measures in \cref{s:metrics}. Then, in \cref{s:sketch}, we explain the main challenges, clarify why the strategy of \cref{s:s3} does not work in this setting, and outline the new proof strategy, which combines stochastic sewing with generalized coupling. In the remainder of the section, we implement this strategy: we establish the corresponding integral bounds in \cref{s:keyintegral}, prove weak existence in \cref{s:we}, and finally obtain weak uniqueness in \cref{s:stochs}.

\subsection{SDEs with distributional drifts}\label{s:41}

We begin this subsection by recalling some basic notions from functional analysis.

The Schwartz space $\S(\R^d)$ is the space of all functions $\phi\in\C^\infty(\R^d,\mathbb{C})$ such that $\phi$ and all its derivatives are rapidly decreasing at infinity. That is, for any $k\in\Z_+$ we have
\begin{equation*}
\|\phi\|_{\S(\R^d),k}:=\sup_{\mu\in \Z^d_+, |\mu|\le k}\|(1+|\cdot|)^k \d^\mu \phi\|_{L_\infty(\R^d)}<\infty,
\end{equation*}
where for $\mu=(\mu_1,\hdots,\mu_d)$ we used the usual convention: $|\mu|:=\sum_{i=1}^d \mu_i$; $\d^\mu:=\d^{\mu_1}_{x_1}\hdots\d^{\mu_d}_{x_d}$.

The space of tempered distributions $\S'(\R^d)$ is the space of all linear maps $f\colon\S(\R^d)\to\mathbb{C}$ such that for some $C>0$, $k\in\Z_+$ one has
\begin{equation*}
|\la f,\phi \ra |\le C \|\phi\|_{\S(\R^d),k},\quad \phi\in\S(\R^d).
\end{equation*}

\begin{convention}\label{c:conv}
As usual, if a measurable function $f\colon\R^d\to\mathbb{C}$ satisfies for some  $N<\infty$ ${\int_{\R^d} |f(x)|(1+|x|)^{-N}\,dx<\infty}$, we identify it with a  tempered distribution $T_f\in\S'(\R^d)$ given by 
\begin{equation}\label{convention}
\langle T_f,\phi\rangle:=\int_{\R^d}f(x)\phi(x)\,dx,\qquad \phi\in\S(\R^d).
\end{equation}
\end{convention}

Now we are ready to define the notion of regularity of a tempered distribution.
\begin{definition}\label{d:besov}
Let $\alpha<0$. We say that a tempered distribution $f\in\S'(\R^d)$ belongs to the Besov space $\C^\alpha(\R^d)$ if  
\begin{equation}\label{negbesov}
\|f\|_{\C^\alpha(\R^d)}:=\sup_{t\in(0,1]}t^{-\frac\alpha2}\|P_t f\|_{L_\infty(\R^d)}<\infty.
\end{equation}
Here $P_tf$ denotes  the convolution of the distribution $f$ with the heat kernel $p_t$ defined in \eqref{heatkernel}, that is 
\begin{equation*}
P_tf(x):=\la f, p_t(x-\cdot)\ra,\quad x\in\R^d.
\end{equation*}
\end{definition}	
It is easy to see that \eqref{negbesov} implies that 
\begin{equation}\label{dbescor}
\|P_t f\|_{L_\infty(\R^d)}\le (t^{\frac\alpha2}\vee1)\|f\|_{\C^\alpha(\R^d)}
\end{equation}
for all $t>0$. Indeed, if $t\in(0,1]$ this follows directly from the definition, and if $t>1$ then we have from the semigroup property of $(P_t)$
\begin{equation*}
\|P_tf\|_{L_\infty(\R^d)}
=\|P_{t-1}(P_1f)\|_{L_\infty(\R^d)}
\le\|P_1f\|_{L_\infty(\R^d)}\le \|f\|_{\mathcal C^\alpha}.
\end{equation*}

\cref{d:besov} is equivalent to the standard definition of the Besov space via Littlewood–Paley blocks but is more convenient for our analysis, see \cref{e:besovdef}. The same exercise will also show that the heat kernel plays no special role here: replacing $p_t$ with a different kernel $h_t(x):=t^{-\frac{d}2}h(t^{-\frac12}x)$, where $h$ is an arbitrary function in $\S(\R^d)$  satisfying $\int_{\R^d}h(x)\,dx\neq0$, yields an equivalent norm.

The definition of the Besov space keeps the following natural property (see \cref{e:easyprop}): if $f\in\C^\alpha(\R^d)$ for $\alpha\in(0,1)$, then $\nabla f$ belongs to the Besov space of negative regularity $\C^{\alpha-1}(\R^d)$. Furthermore, integrable functions belong to a Besov space of certain negative regularity, that is, $L_p(\R^d)\subset\C^{-\frac{d}p}(\R^d)$ for any $p\in[1,\infty)$, $d\ge1$ (a more general version of this statement is called the Besov embedding theorem, see, e.g.,  \cite[Section~1.11.1]{Tr3}).

We also note that \cref{d:besov} matches \cref{l:gb} in the following sense.

\begin{lemma}\label{l:gbext}
	Let $f\in\C^\gamma(\R^d,\R)$, $\gamma<0$, $\rho\in(0,1]$. Then there exists $C=C(\gamma,\rho,d)>0$ such that for any  $t\in(0,1]$ we have 
	\begin{equation*}
		\|P_t  f\|_{\C^\rho}\le C t^{\frac{\gamma-\rho}2}\|f\|_{\C^\gamma}.
	\end{equation*} 	
\end{lemma}

\begin{proof}
We write $P_t f=P_{\frac{t}2}P_{\frac{t}2} f=:P_{\frac{t}2} g$, where we denoted $g:=P_{t/2} f$. It follows from \cref{d:besov} that 
\begin{equation}\label{gsoup}
\|g\|_{L_\infty(\R^d)}\le C t^{\frac\gamma2}\|f\|_{\C^\gamma}.
\end{equation}
for $C=C(\gamma)>0$. Next, using \cref{l:gb} we deduce for any $x,y\in\R^d$
\begin{align*}
|P_{\frac{t}2} g(x)-P_{\frac{t}2} g(y)|&\le 2 \|g\|_{L_\infty(\R^d)}^{1-\rho} |P_{\frac{t}2} g(x)-P_{\frac{t}2} g(y)|^\rho\le 2
 \|g\|_{L_\infty(\R^d)}^{1-\rho} |x-y|^\rho \|P_{\frac{t}2} g\|_{\C^1}^\rho\\
 &\le C \|g\|_{L_\infty(\R^d)} |x-y|^\rho t^{-\frac\rho2},
\end{align*}
for $C=C(\gamma,\rho,d)$. Thus, recalling also \eqref{gsoup}, we finally get
\begin{equation*}
\|P_{t} f\|_{\C^\rho}=\|P_{\frac{t}2} g\|_{\C^\rho}\le  C \|g\|_{L_\infty(\R^d)}  t^{-\frac\rho2}\le C t^{-\frac\rho2}t^{\frac\gamma2}\|f\|_{\C^\gamma}.\qedhere
\end{equation*}
\end{proof}

Next, let us explain how we can define the notion of a solution to the SDE 
\begin{equation}\label{distrsde}
X_t=x+\int_0^t b(X_r)\,dr +W_t,
\end{equation}
where the drift $b$ is not a function but a distribution from the Besov space $\C^\alpha(\R^d)$, with ${\alpha<0}$. Note that in this case the term $b(X_r)$ is not defined: indeed, one cannot evaluate Schwartz distributions at fixed points. Recall, however, that in \cref{s:s3} we observed that for any bounded measurable function ${f\colon\R^d\to\R}$, the function ${\wt f(x):=\int_0^1 f(W_r+x)\,dr}$, $x\in\R^d$, is a.s. much smoother than $f$; namely, $\wt f$ is locally almost Lipschitz. Therefore, instead of defining $b(X_r)$ directly, we attempt to make sense  of the whole integral  $\int b(X_r)\,dr$, hoping that it will be well defined thanks to the averaging effect of $X$. This motivates the following definition.

\begin{definition}\label{d:apprsec}
We say that a sequence of functions $f^n\colon\R^d\to\R^d$, $n\in\Z_+$, converges to a distribution $f$ in $\C^{\beta-}$, $\beta\in\R$,  if $\sup_n \|f^n\|_{\C^{\beta}}<\infty$ and  for any $\beta'<\beta$ we have $\|f^n-f\|_{\C^{\beta'}}\to0$ as $n\to\infty$.
\end{definition}

\begin{definition}[{\cite[Definition~2.1]{BC}}]\label{D:sol}
Let $b\in\C^{\beta}(\R^d,\R^d)$, $\beta\in\R$. We say that a continuous   process $(X_t)_{t\in[0,1]}$ taking values in $\R^d$ is a solution to \eqref{distrsde} with the initial condition $x\in\R^d$, if there exists a continuous process $(\psi_t)_{t\in[0,1]}$ taking values in $\R^d$ such that:
\begin{enumerate}[(i)]
\item $X_t=x+\psi_t+W_t$, $t\in[0,1]$ a.s.;
\item\label{cond2md} for \textit{any} sequence $(b^n)_{n\in\Z_+}$ of $\C^\infty(\R^d,\R^d)$ functions converging to $b$ in $\C^{\beta-}$ we have
\begin{equation*}
\lim_{n\to\infty}\sup_{t\in[0,1]}\Big|\int_0^t b^n(X_r)\,dr-\psi_t\Big|= 0\,\,\text{in probability}.
\end{equation*}
\end{enumerate}
\end{definition}

Note that in \cref{d:apprsec} we did not require $f^n$ to converge to $f$ in $\C^\beta$, but imposed a slightly weaker condition. This is due to the fact that given $f\in\C^\beta$, $\beta\in\R$, one can always construct a sequence $(f^n)_{n\in\Z+}$ of smooth functions converging to $f$ in $\C^{\beta-}$; for example, one can take $f^n:=P_{1/n}f$, see \cref{e:betamin}. On the other hand, for a general $f\in\C^\beta$ a sequence of smooth functions $(f^n)_{n\in\Z+}$ converging to $f$ in $\C^\beta$ might not exist, see   \cref{e:betamin} for a counterexample.

We also note that if $b\in\C^\beta$ with $\beta>0$, then $b(X_t)$ is of course well defined, and in this case it is immediate to see that this notion of  solution coincides with the standard one.

As in the classical case, we define a \textit{weak solution} to \eqref{distrsde} to be a couple $(X,W)$ on a complete filtered probability space $(\Omega, \F,  (\F_t)_{t\in[0,1]}, \P)$ such that $W$ is an $ (\F_t)$-Brownian motion, $X$ is adapted to $(\F_t)$ and solves equation \eqref{distrsde} in the sense  of \cref{D:sol}. We say that \textit{weak uniqueness} holds for \eqref{distrsde} if whenever    $(X,W)$ and $(Y,\overline W)$ are two  weak solutions of this equation (not necessarily defined on the same probability space), then   $\law(X)=\law(Y)$ on the path space  $\C([0,1];\R^{d})$. 

\subsection{Relative entropy and distances between probability measures}\label{s:metrics}

In this subsection we introduce some notation related to distances between probability measures and recall some fundamentals. Let $(E,\EE)$ be a measurable space and denote by $\PP(E)$ the space of probability measures on it. For  $\mu,\nu\in\PP(E)$ we write $\mu\ll\nu$ if $\mu$ is absolutely continuous with respect to $\nu$, that is, $\nu(A)=0$ implies $\mu(A)=0$ for any $A\in\EE$. If $\mu\ll\nu$, then the \textit{Radon–Nikodym derivative} $\frac{d\mu}{d\nu}$ is well-defined and one has $\mu(A)=\int_A  \frac{d\mu}{d\nu}(x)\nu(dx)$ for any $A\in\EE$.

For two probability measures $\mu,\nu\in\PP(E)$ define the \textit{total variation distance} between them as
\begin{equation}\label{dtv}
d_{TV}(\mu,\nu):=\sup_{A\in\EE}|\mu(A)-\nu(A)|.
\end{equation}
One can show that if the measure $\rho\in\PP(E)$ satisfies $\mu\ll\rho$, $\nu\ll\rho$ then the total variation distance is proportional to the integral of the absolute difference between the two densities:
\begin{equation}\label{dtvden}
d_{TV}(\mu,\nu)=\frac12\int_E \Bigl|\frac{d\mu}{d\rho}(x)-\frac{d\nu}{d\rho}(x)\Bigr|\rho(dx).
\end{equation}

Another quantity characterizing the closeness of two probability measures is the  \textit{Kullback–Leibler divergence} (KL divergence), also called relative entropy. If $\mu,\nu\in\PP(E)$ and $\mu\ll\nu$, then we set 
\begin{equation*}
D_{KL}(\mu\|\nu):=\int_E \log \bigl({\frac{d\mu}{d\nu}}\bigr)d \mu=
\int_E \frac{d\mu}{d\nu}\log \bigl({\frac{d\mu}{d\nu}}\bigr)d \nu.
\end{equation*}
If $\mu\nll\nu$, then, for convenience, we put $D_{KL}(\mu\|\nu)=\infty$. Even though the KL divergence is not a metric (it is not even symmetric), it provides a useful measure of similarity between probability measures. 

First, let us make sure that the KL divergence is always well defined. Indeed, when defining a Lebesgue integral $\int f\,d\mu$ for a function $f$ that may take both positive and negative values, we need at least one of the integrals $\int f_+\,d\mu$ and $\int f_-\,d\mu$ to be finite (otherwise we cannot define $\int f d \mu$). Here $a_+$ and $a_-$ denote, respectively, the positive and negative parts of a real number $a$: $a_+=a\vee0$ and $a_-=-(a\wedge0)$.

\begin{proposition}\label{p:basicdkl} Let $\mu,\nu\in\PP(E)$. Then the following holds
\begin{enumerate}[(i)]
\item Suppose that additionally $\mu\ll \nu$. Then
\begin{equation*}
\int_E \frac{d\mu}{d\nu} \Bigl(\log \bigl({\frac{d\mu}{d\nu}}\bigr)\Bigr)_-\,d\nu\le d_{TV}(\mu,\nu),
\end{equation*}
and thus $D_{KL}(\mu\|\nu)$ is always well-defined.
\item $D_{KL}(\mu\|\nu)\ge0$ and $D_{KL}(\mu\|\nu)=0$ if and only if $\mu=\nu$.
\end{enumerate}
\end{proposition}
\begin{proof}
Both parts follow from the inequality
\begin{equation}\label{ineqent}
x\log x-x+1\ge0,\qquad x>0,
\end{equation}
with equality only at $x=1$. Indeed, this inequality implies $-x\log x\le1 -x$, therefore $x(\log x)_-=(-x\log x)_+\le(1 -x)_+$ for $x\ge0$. Setting now $A:=\{\frac{d\mu}{d\nu}\le1\}\subset E$, we get
\begin{equation*}
\int_E \frac{d\mu}{d\nu} \Bigl(\log \bigl({\frac{d\mu}{d\nu}}\bigr)\Bigr)_-\,d\nu\le
\int_E \bigl(1-\frac{d\mu}{d\nu}\bigr)_+d\nu= \int_{A} (1-\frac{d\mu}{d\nu})d\nu=\nu(A)-\mu(A)\le  d_{TV}(\mu,\nu),
\end{equation*}
where the final inequality follows from \eqref{dtv}. This shows part (i) of the proposition.

In a similar manner we deduce from \eqref{ineqent} that 
\begin{equation*}
D_{KL}(\mu\|\nu)=\int_E \frac{d\mu}{d\nu}\log \bigl({\frac{d\mu}{d\nu}}\bigr)d \nu\ge
\int_E (\frac{d\mu}{d\nu}-1)d \nu=0,
\end{equation*}
which implies part (ii).
\end{proof}	

We would also need that the KL divergence is lower semicontinuous.

\begin{proposition}\label{p:lowersc}
Let $\mu,\nu\in\PP(E)$. Then for any measures $\mu_n\in\PP(E)$, $n\in\Z_+$, that converge  to $\mu$ in total variation, we have
\begin{equation}\label{liminfeq}
D_{KL}(\mu\|\nu)\le \liminf_{n\to\infty}	D_{KL}(\mu_n\|\nu).
\end{equation}	
\end{proposition}
\begin{proof}
Without loss of generality, we can assume that $\liminf_{n\to\infty}D_{KL}(\mu_n\|\nu)<\infty$ (otherwise \eqref{liminfeq} is immediate). Therefore, we can drop the terms with infinite KL divergence without changing the right-hand side of \eqref{liminfeq}. Thus, we may assume that $\mu_n\ll\nu$ for every $n$. Then  for any set $A\in\EE$ with $\nu(A)=0$ we have $\mu_n(A)=0$, and therefore
\begin{equation*}
0\le\mu(A)=|\mu(A)-\mu_n(A)|\le d_{TV}(\mu,\mu_n)\to0\quad\text{as $n\to\infty$}.
\end{equation*}
Hence $\mu(A)=0$ and thus $\mu\ll\nu$. Therefore, using formula \eqref{dtvden} we get
\begin{equation*}
\lim_{n\to\infty}\int_E\Bigl|\frac{d \mu_n}{d\nu}(x)-\frac{d \mu}{d\nu}(x)\Bigr|\nu(dx)=2\lim_{n\to\infty}d_{TV}(\mu,\mu_n)=0.
\end{equation*}
Thus, $\frac{d \mu_n}{d\nu}$ converges to $\frac{d \mu}{d\nu}$ in $\nu$-measure. Since the function $x\mapsto x\log x$, where $x>0$, is continuous and bounded from below by a constant, Fatou’s lemma implies
\begin{equation*}
D_{KL}(\mu\|\nu)
=\int_E \frac{d \mu}{d\nu}\log\bigl(\frac{d \mu}{d\nu}\bigr)\,d\nu
\le\liminf_{n\to\infty}\int_E \frac{d \mu_n}{d\nu}\log\bigl(\frac{d \mu_n}{d\nu}\bigr)\,d\nu
=\liminf_{n\to\infty}D_{KL}(\mu_n\|\nu).\qedhere
\end{equation*}
\end{proof}
A different version of the lower semicontinuity property is studied in \cref{e:lsc}.

While the KL divergence is an important object in its own right, we introduce it mainly as a tool for bounding the total variation distance between two measures.
As we will see later in \cref{l:dtvlemma}, the KL divergence between the laws of two diffusion processes can be effectively estimated using Girsanov’s theorem. Pinsker’s inequality (provided below) then converts these estimates into bounds on total variation.

\begin{proposition}\label{pinsker}
 Let $\mu,\nu\in\PP(E)$. Then the following holds:
\begin{enumerate}[(i)] 
\item (Pinsker's inequality)  $d_{TV}(\mu,\nu)\le \sqrt{\frac12D_{KL}(\mu\|\nu)}$.
\item (Data processing inequalities) Let $(E_1,\mathcal{E}_1)$ be another measurable space, let $\Phi\colon E\to E_1$ be a measurable function. Then 
\begin{equation*}
D_{KL}(\Phi^{\#}\mu\|\Phi^{\#}\nu)\le 
D_{KL}(\mu\|\nu)\qquad \text{and} \qquad d_{TV}(\Phi^{\#}\mu,\Phi^{\#}\nu)\le d_{TV}(\mu,\nu),
\end{equation*}	
where $\Phi^\#\mu$ denotes the push-forward of a measure $\mu$ under a map $\Phi$: 
$$
\Phi^\#\mu(A):=\mu\bigl(\{x\in E:\Phi(x)\in A\}\bigr),\qquad A\in\mathcal E_1.
$$
\end{enumerate}
\end{proposition} 
The proof of the proposition is covered in \cref{e:data,e:pinsker}.

If the space  $E$ is a Polish space equipped with a distance $\rho$, and $\EE$ is the Borel $\sigma$-algebra $\mathscr{B}(E)$,  we can define the Wasserstein (Kantorovich-Rubinstein) distance between two measures $\mu,\nu\in\PP(E)$
\begin{equation}\label{wrho}
\W_\rho(\mu,\nu):=\inf_{\substack{\law(X)=\mu\\\law(Y)=\nu}}\E \rho(X,Y),
\end{equation}
where the infimum is taken over all random variables $X,Y$ with $\law(X)=\mu$, $\law(Y)=\nu$. 
In this case we also have
\begin{equation}\label{dtvv}
d_{TV}(\mu,\nu)=\inf_{\substack{\law(X)=\mu\\\law(Y)=\nu}} \P(X\neq Y).
\end{equation}
It is immediate to see that if the metric $\rho$ is bounded by $1$, then one has $\rho(x,y)\le\I(x\neq y)$ for $x,y\in E$ and therefore \eqref{wrho} and \eqref{dtvv} imply in this case
\begin{equation}\label{wrhocon}
	\W_\rho(\mu,\nu)\le d_{TV}(\mu,\nu),\quad \mu,\nu\in\PP(E).
\end{equation}

It is well known that if   $\rho$ is bounded, then weak convergence of measures is equivalent to convergence in $\W_\rho$, see, e.g., \cite[Corollary~6.13]{Villani}.

\subsection{Challenges and proof strategy} \label{s:sketch}

Before going into the details, let us explain what the challenges are in proving weak uniqueness of SDEs/SPDEs in general and why stochastic sewing alone is not efficient. Then we present a sketch of the proof strategy, skipping over some technicalities which will be treated later. In particular, we omit the arbitrarily small exponents.

Thus, let $(X,W)$ and $(Y,\overline{W})$ be two solutions to SDE \eqref{distrsde}. Our goal is to show that
$\law (X)=\law(Y)$. 

For the case of SDEs driven by a Brownian motion, a possible proof strategy is to use the Zvonkin transformation method. The idea is to consider an SDE for $\Psi(X)$, where $\Psi\colon \R^d \to \R^d$ is a certain ``nice'' function. Using It\^o's formula, one can show that for a carefully chosen $\Psi$, the new equation has a less singular drift and therefore it  has a unique weak solution. Then, since $\Psi$ is nice, one deduces weak uniqueness for the original SDE, see \cite{bib:zz17,FIR17} for more details.

However, this approach cannot be extended to SDEs driven by a more general process, such as fractional Brownian motion (in which case It\^o's formula is not available) or to stochastic PDEs (in this case only a rather restricted It\^o-type formula is available,  which does not allow to implement the above strategy).

Alternatively,  weak uniqueness can also be established using Girsanov's theorem if one can show 
that the process $W+\int_0^\cdot b(X_r) \,dr$ has the law of a Brownian motion under some  equivalent probability measure $\wt \P$. However, we  are in the range of regularity of the drifts $b$, where this is not the case. Indeed, we have to require that at least $\int_0^1 |b|^2(X_r)\,dr<\infty$, which does not hold  if $b$ is a true distribution (one cannot square the distributions).

Thus, we aim to  develop a general strategy which can show directly that  $\E \Phi(X)=\E \Phi(Y)$ for a large class of test functions $\Phi\colon \C([0,1];\R^d)\to\R$ without relying on PDE methods or Markov property of the noise. 
Assume for a moment that $X$ and $Y$ are defined on the same probability space, that $W = \overline{W}$, and that we consider the simpler case of test functions $\Phi\colon\mathbb R^d\to\mathbb R$ applied to the solution at a single fixed time $t\in[0,1]$. Then a naive direct approach, which fails here, would be simply to write
\begin{equation}\label{firstineqw}
|\E \Phi(X_t)-\E \Phi(Y_t)|\le \|\Phi\|_{\C^1} \E |X_t-Y_t|= \|\Phi\|_{\C^1}\E \Bigl|\int_0^t (b(X_r)-b(Y_r)) \,dr \Bigr|,
\end{equation}
and then use stochastic sewing, arguing as  in \cref{s:s3}. Similarly to \cref{e:313} applied with $H=\frac12$, one would get for $T\in[0,1]$
\begin{equation}\label{directbound:s}
\sup_{t\in[0,T]}\|X_t-Y_t\|_{L_2(\Omega)}\le C T^{\frac12}\|b\|_{\C^\gamma}\sup_{t\in[0,T]}\|X_t-Y_t\|_{L_2(\Omega)}^{(\gamma+1)\wedge1}+\text{good small terms}.
\end{equation}
While this bound is definitely better than a direct bound \eqref{firstattempt}, it is still not sufficient for our purposes. Indeed, if $\gamma>0$, one can remove the bad term $\sup_{t\in[0,T]}\|X_t-Y_t\|_{L_2(\Omega)}$ on the right-hand side for small $T>0$, and even get strong uniqueness of solutions to SDE \eqref{distrsde}, as shown  in \cref{s:s3}.  However, we are interested in distributional drifts, so $\gamma$ can be negative as well. In this regime, the exponent $\gamma+1$ is less than $1$ and \eqref{directbound:s} does not imply that $\sup_{t\in[0,T]}\|X_t-Y_t\|_{L_2(\Omega)}$ is small. Therefore this direct approach does not work. 

We also see that already the first inequality in \eqref{firstineqw} is too harsh, as we do not expect $\E |X_t-Y_t|$ to be $0$ (this would imply strong uniqueness and we are interested in the regime where weak uniqueness holds and strong may fail). On the other hand, it is not clear how to bound $|\E \Phi(X_t)-\E \Phi(Y_t)|$ with stochastic sewing.

Therefore, we avoid comparing the solutions $X$ and $Y$ directly. Instead, we combine stochastic sewing  with some ideas from ergodic theory: the generalized coupling method of \cite{H02,M02,Ksch}. Fix a sequence $(b^n)_{n\in\Z_+}$  of smooth functions converging to $b$ in $\C^{\gamma-}$. Let  $X^n$ be the strong solution to SDE
\begin{equation}\label{distrsden}
X^n_t=x+\int_0^t b^n(X^n_r)\,dr +W_t,\quad t\in[0,1]
\end{equation}
We will prove that $X^n$ converges weakly to $X$ and similarly that  $X^n$ converges weakly to $Y$. By uniqueness of the limit, this would imply $\law (X)=\law(Y)$. 

To do this we consider the bounded separable metric space $\C([0,1];\R^d)=\C^0([0,1];\R^d)$ equipped with the distance
\begin{equation*}
d(x,y):=\|x-y\|_{\C([0,1];\R^d)}\wedge1,
\end{equation*}
 and the corresponding Wasserstein distance $\W_{\|\cdot\|\wedge1}$, recall~\eqref{wrho}. As explained in \cref{s:metrics}, to show that the sequence $(X^n)$ converges weakly to $X$ it is sufficient to show that  
\begin{equation*}
\lim_{n\to\infty}\W_{\|\cdot\|\wedge1}(\law(X^n),\law(X))=0.
\end{equation*}
The main idea of the  generalized coupling method  is to introduce an auxiliary process $\wt X^{n,\lambda}$, which is defined on the same space as $X$ and solves SDE
\begin{equation}\label{wtxn}
	d \wt X^{n,\lambda}_t =b^n(\wt X^{n,\lambda}_t) dt+\lambda(X_t-\wt X^{n,\lambda}_t)dt+d W_t,\quad \wt X^{n,\lambda}_0=x,
\end{equation}
where the parameter $\lambda>1$ will be fixed later. We see that the additional term ${\lambda(X_t-\wt X^{n,\lambda}_t)dt}$ pushes $\wt X^{n,\lambda}$ towards $X$, see \cref{fig:coupling}. 
The  goal now is to show that the process $\wt X^{n,\lambda}$ is close  to $X^n$ in law for $\lambda$ not too large, and close to $X$ in distance for  $\lambda$ not too small. Then by picking the best $\lambda$ from these two opposite requirements, we derive that $X^n$ and $X$ are close to each other in law. More precisely, by the triangle inequality, bound \eqref{wrhocon} and the definition of the distances $\W_{\|\cdot\|\wedge1}$ in \eqref{wrho}, we obtain
\begin{align}\label{key}
\W_{\|\cdot\|\wedge1}(\law(X^n),\law(X))&\le \W_{\|\cdot\|\wedge1}(\law(X^n),\law(\wt X^{n,\lambda}))+\W_{\|\cdot\|\wedge1}(\law(\wt X^{n,\lambda}),\law(X))\nn\\
&\le d_{TV}(\law(X^n),\law(\wt X^{n,\lambda}))+\E \|\wt X^{n,\lambda}-X\|_{\C([0,1];\R^d)}.
\end{align}
\begin{figure}[htbp]
	\centering
	\includegraphics[width=\linewidth]{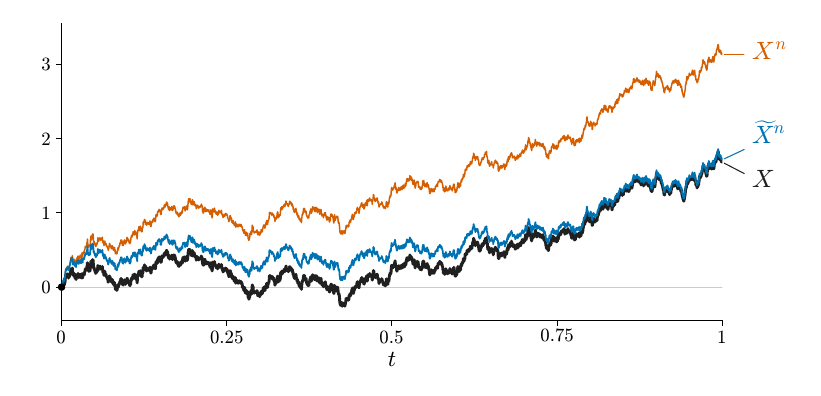}
	\caption{Generalized coupling. The  processes $X$, $X^n$, $\wt X^{n,\lambda}$ start at the same point and share the same Brownian driver. The control keeps $\wt X^{n,\lambda}$ close to $X$, while $X^n$ moves away.}
	\label{fig:coupling}
\end{figure}

To bound the first term on the right-hand side of \eqref{key}, we note  that the SDE \eqref{wtxn} can be rewritten as 
\begin{equation*}
d \wt X^{n,\lambda}_t =b^n(\wt X^{n,\lambda}_t) dt+d\wt  W_t,
\end{equation*}
where we put $\wt W_t:=\int_0^t\lambda(X_r-\wt X^{n,\lambda}_r)dr+ W_t$, $t\in[0,1]$. The Girsanov theorem implies that $\wt W$ is a Brownian motion under a new probability measure $\wt \P$ (see \cref{l:dtvlemma} for details) and therefore the pair $(\wt X^{n,\lambda},\wt W)$ is a solution to \eqref{distrsden} on $(\Omega,\wt \P)$. Hence, by Girsanov's theorem, the data processing inequality, and Pinsker's inequality (\cref{pinsker}), we get  (see again \cref{l:dtvlemma}  for more details)
\begin{equation}\label{tvbound:s}
	d_{TV}(\law(X^n),\law(\wt X^{n,\lambda}))\le d_{TV}(\P,\wt \P)\le  C \lambda\bigl\|\|X-\wt X^{n,\lambda}\|_{\C([0,1];\R^d)}\bigr\|_{L_2(\Omega)}.
\end{equation}

Next, to compare $\wt X^{n,\lambda}$ and $X$, we benefit from the control term $\lambda(X-\wt X^{n,\lambda})$ which pushes $\wt X^{n,\lambda}$ towards $X$. After a short calculation, we deduce for $t\in[0,1]$
\begin{equation}\label{keyint:s}
	|X_t-\wt X^{n,\lambda}_t|\le \Bigl|\int_0^t e^{-\lambda(t-r)}(b(X_r)-b^n(\wt X_r^{n,\lambda}))\,dr\Bigr|.
\end{equation}	
However, by this time the reader will probably already have become a master of stochastic sewing and will immediately recognize that this bound is of the type \eqref{type} and is very close to the bound \eqref{seckey}. Therefore, modifying a bit the argument of \cref{e:step1uni}, we obtain \begin{equation*}
	\sup_{t\in[0,1]}\|X_t-\wt X_t^{n,\lambda}\|_{L_2(\Omega)}\le C \lambda^{-\frac34-\frac{\gamma}2}
	\sup_{t\in[0,1]} \|X_t-\wt X^{n,\lambda}_t \|_{L_2(\Omega)}^{\frac12}+\lambda  \|b-b^n\|_{\C^{\gamma-\eps}}.
\end{equation*}
Of course, similar to \eqref{directbound:s}, the term $\|X_t - \widetilde{X}^{n,\lambda}_t \|_{L_2(\Omega)}$ appears with an exponent less than $1$, and we cannot get rid of it directly. However, this is not a problem anymore due to the presence of the additional factor $\lambda^{-\frac34-\frac{\gamma}2}$, which we can make very small. By Young's inequality
\begin{equation*}
	\sup_{t\in[0,1]}\|X_t-\wt X_t^{n,\lambda}\|_{L_2(\Omega)}\le C \lambda^{-\frac32-\gamma} +\frac12
	\sup_{t\in[0,1]} \|X_t-\wt X^{n,\lambda}_t \|_{L_2(\Omega)}+ \lambda  \|b-b^n\|_{\C^{\gamma-\eps}},
\end{equation*}
which implies 
\begin{equation*}
	\sup_{t\in[0,1]}\|X_t-\wt X_t^{n,\lambda}\|_{L_2(\Omega)}\le C \lambda^{-\frac32-\gamma} + \lambda  \|b-b^n\|_{\C^{\gamma-\eps}}.
	\end{equation*}
Now we can combine this together with \eqref{tvbound:s} and substitute into \eqref{key}. We get (up to inverting the order of the $\C([0,1])$ and $L_2(\Omega)$ norms, which should be justified separately with Kolmogorov continuity theorem-type arguments)
\begin{equation*}
\W_{\|\cdot\|\wedge1}(\law(X^n),\law(X))\le C \lambda^{-\frac12-\gamma} + \lambda^2  \|b-b^n\|_{\C^{\gamma-\eps}}.
\end{equation*}
If the power of $\lambda$ is negative, that is 
\begin{equation}\label{gammaweakcond}
-\frac12-\gamma<0,
\end{equation}
 then we can  pick $\lambda:=\|b-b^n\|_{\C^{\gamma-\eps}}^{-1/3}$ to get 
\begin{equation*}
	\W_{\|\cdot\|\wedge1}(\law(X^n),\law(X))\le C   \|b-b^n\|_{\C^{\gamma-\eps}}^\rho
\end{equation*}
for some   $\rho>0$. This yields the desired weak convergence of the fixed sequence $(X^n)_{n\in\Z_+}$ to any weak solution $X$, establishing weak uniqueness. The condition 
\eqref{gammaweakcond} is exactly the condition $\gamma>-\frac12$ of \cref{t:weak}.

Now it remains to turn this informal description of our proof strategy into a rigorous argument. In \cref{s:keyintegral} we bound the moments of the right-hand side of \eqref{keyint:s}. In \cref{s:we} we establish weak existence of solutions, and in \cref{s:stochs} we conclude.

\subsection{Key integral bounds}\label{s:keyintegral}

Let us recall that we need to bound the right-hand side of \eqref{keyint:s} and show that it is small for large $\lambda$. We split the integral into two parts:
\begin{equation}\label{boundbound}
\int_0^t e^{-\lambda(t-r)}(b(X_r)-b^n(\wt X_r^{n,\lambda}))\,dr= \Bigl(\int_0^{(t-\lambda^{-1+\eps})\vee0}+\int_{(t-\lambda^{-1+\eps})\vee0}^t\Bigr) e^{-\lambda(t-r)}(b(X_r)-b^n(\wt X_r^{n,\lambda}))\,dr.
\end{equation}
Now, if $r$ is far from $t$ (the first integral), then the weight $e^{-\lambda(t-r)}$ is tiny; it is smaller than $\exp(-\lambda^{\eps})$. On the other hand, if $r$ is close to $t$ (the second integral), the weight is of order $1$, but the time interval $[t-\lambda^{-1+\eps}, t]$ is very small. Thus, in both cases, each part of the integral is small. To make this heuristic precise, we need to generalize \cref{e:step1uni} in two directions: first, by allowing a deterministic time-dependent weight in front of $b$; and second, by extending the range of $\gamma$ to include negative values.

We first prove a general statement and then, as a corollary, bound \eqref{boundbound}. Recall that we measure the regularity of the perturbation in the norms $[\cdot]_{\C^\tau L_m([S,T])}$ introduced in~\eqref{norms}.
\begin{lemma}\label{l:finfin}
Let  $m\ge2$, $\gamma\in(-1,0)$, $\tau\in(0,1]$. Let $f\colon\R^d\to\R$, $w\colon[0,1]\to\R_+$ be bounded continuous functions.
Let  $\phi,\psi\colon\Omega\times[0,1]\to\R^d$ be continuous measurable processes adapted to the filtration $(\F_t)_{t\in[0,1]}$.
	Assume that
	\begin{equation}\label{paramcond}
		\frac\gamma2+\tau>\frac12.
	\end{equation}
Then there exists a constant $C=C(\gamma,\tau,d,m)$ such that for any $(S,T)\in\Delta_{[0,1]}$, we have 	
\begin{equation}\label{sb1}
\Bigl\| \int_S^T w_r f(W_r+\phi_r)\,dr  \Bigr\|_{L_m(\Omega)}
\le C \|f\|_{\C^{\gamma}}\|w\|_{L_\infty([S,T])}(T-S)^{1+\frac\gamma2}\bigl(1+ [\phi]_{\C^\tau L_m([S,T])} (T-S)^{\tau-\frac12}\bigr).
\end{equation}
	
Additionally, for any $\mu\in(0,1]$ such that 
	\begin{equation}\label{paramcond2}
		\gamma>\mu-1,
	\end{equation}
there exists a constant $C=C(\gamma,\mu,\tau,d,m)$  such that for any $(S,T)\in\Delta_{[0,1]}$ we have 
	\begin{align}
		&\Bigl\|\int_S^T w_r(f(W_r+\phi_r)-f(W_r+\psi_r))\,dr\Bigr\|_{L_m(\Omega)}\nn\\
		&\quad\le
		C \|f\|_{\C^{\gamma}}\|w\|_{L_\infty([S,T])}(T-S)^{1+\frac{\gamma-\mu}2}\|\phi-\psi\|^\mu_{\C^0 L_m([S,T])}\nn\\
		&\qquad+C \|f\|_{\C^{\gamma}}\|w\|_{L_\infty([S,T])} (T-S)^{\frac12+\frac\gamma2+\tau}( [\phi]_{\C^\tau L_m([S,T])}+ [\psi]_{\C^\tau L_m([S,T])}).\label{sb2}
	\end{align}
\end{lemma}

Let us note that even though we have assumed $f$ to be a bounded continuous function, the right-hand sides of \eqref{sb1} and \eqref{sb2} depend only on a much weaker norm of $f$, namely $\|f\|_{\C^\gamma}$.

\begin{proof}
First, we suppose additionally that $f\in\C^\infty$ and then as in  Step 2 of the proof  of \cref{e:step1uni}, we remove this assumption by Fatou's lemma. Fix $(S,T)\in\Delta_{[0,1]}$.

As in the proof of \cref{e:step1uni}, we rely on the stochastic sewing lemma. As before it is crucial to choose the correct germ. Inspired by our previous choice \eqref{germper}, we put
\begin{equation}\label{Aprocess}
A^\phi_{s,t}:=\int_s^t w_r \E^s f(W_r+\phi_s)\,dr;\qquad
\A^\phi_{t}:=\int_S^t w_r f(W_r+\phi_r)\,dr, \quad (s,t)\in\Delta_{[S,T]}.
\end{equation}

We begin with the proof of \eqref{sb1}. Verification of condition (i) is done in the same way as at the end of Step 1 in the proof of \cref{e:step1uni} and is left as a technical exercise to the reader. We see that $A^\phi_{s,t}$ is $\F_s$-measurable and therefore condition (ii) of \cref{T:SSLst} (SSL) holds. Next, let us verify condition (iii) of the SSL. Using the definition of the Besov space \eqref{negbesov}, we derive for $(s,t)\in\Delta_{[S,T]}$

\begin{align*}
|A^\phi_{s,t}|&\le\int_s^t w_r |P_{r-s}f(W_s+\phi_s)|\,dr\le \int_s^t w_r \|P_{r-s}f\|_{L_\infty(\R^d)}\,dr\\
&\le  C \|f\|_{\C^{\gamma}}\|w\|_{L_\infty([S,T])}\int_s^t (r-s)^{\frac\gamma2}\,dr\le  C \|f\|_{\C^{\gamma}}\|w\|_{L_\infty([S,T])}(t-s)^{1+\frac\gamma2},
\end{align*}	
where  $C=C(\gamma,d)$. This implies for any $m\ge2$
\begin{equation}\label{astb1}
\|A^\phi_{s,t}\|_{\lm}\le 	C \|f\|_{\C^{\gamma}}\|w\|_{L_\infty([S,T])}(t-s)^{1+\frac\gamma2},\quad  (s,t)\in \Delta_{[S,T]}.
\end{equation}	
Since $\gamma>-1$, condition \eqref{con:s1} of the SSL holds.
	
Next, we pass to the verification of condition \eqref{con:s2}.  We have for $(s,u,t)\in\Delta^3_{[S,T]}$
\begin{equation*}
\delta A^\phi_{s,u,t}=\E^s \int_u^t w_r f(W_r+\phi_s)\,dr-\E^u \int_u^t w_r f(W_r+\phi_u)\,dr.
\end{equation*}	
Therefore,  
\begin{align*}
|\E^s \delta A^\phi_{s,u, t}|&=\Bigl|\E^s \int_u^t w_r   \E^u[f(W_r+\phi_s)-f(W_r+\phi_u)]\,dr\Bigr|\nn\\
&=\Bigl|\E^s \int_u^t w_r   (P_{r-u}f(W_u+\phi_s)-P_{r-u}f(W_u+\phi_u))\,dr\Bigr|\nn\\
&\le\E^s \int_u^t w_r   \|P_{r-u}f\|_{\C^1(\R^d)} |\phi_s-\phi_u| \,dr.
\end{align*}
Using again the definition of a negative Besov space and  conditional Jensen's inequality in the form \eqref{useful} with $\mathcal{G}=\F_s$, we get for any $m\ge2$
\begin{align}\label{astb2}
\|\E^s \delta A^\phi_{s,u, t}\|_{L_m(\Omega)}&\le  \int_u^t w_r   \|P_{r-u}f\|_{\C^1(\R^d)} \|\phi_s-\phi_u\|_{L_m(\Omega)} \,dr\nn\\
&\le  C \|f\|_{\C^{\gamma}}\|w\|_{L_\infty([S,T])} [\phi]_{\C^\tau L_m([S,T])} (t-s)^{\tau} \int_u^t  (r-u)^{\frac\gamma2-\frac12} \,dr\nn\\
&\le  C \|f\|_{\C^{\gamma}}\|w\|_{L_\infty([S,T])} [\phi]_{\C^\tau L_m([S,T])} (t-s)^{\frac12+\frac\gamma2+\tau}
\end{align}
for $C=C(\gamma,d)$. Here we used that the singularity $(r-u)^{\frac\gamma2-\frac12}$ is integrable because $\gamma>-1$. By \eqref{paramcond} we have $\frac12+\frac\gamma2+\tau>1$ and therefore condition \eqref{con:s2} is satisfied.

Thus, all the conditions of \cref{T:SSLst} are satisfied and \eqref{sb1} follows from \eqref{est:ssl1} and \eqref{astb1}, \eqref{astb2}.

Now we move on to the proof of \eqref{sb2}. We apply now the SSL  to the processes  
\begin{equation*}
A_{s,t}:=A^\phi_{s,t}-A^\psi_{s,t};\quad 	\A_{t}:=\A^\phi_t-\A^\psi_t,\qquad (s,t)\in \Delta_{[S,T]},
\end{equation*}
	where the processes $A^\phi$, $\A^\phi$ were introduced in \eqref{Aprocess} and $A^\psi$, $\A^\psi$ are defined in exactly the same way with $\phi$ replaced by $\psi$. 
	
We again leave verification of condition (i) of the SSL to the reader and note that condition (ii) obviously holds. Next, using \cref{l:gbext}, we write for any $(s,t)\in\Delta_{[S,T]}$
\begin{align*}
|A_{s,t}|&\le\int_s^t w_r |P_{r-s}f(W_s+\phi_s)-P_{r-s}f(W_s+\psi_s)|\,dr\le |\phi_s-\psi_s|^\mu\int_s^t w_r \|P_{r-s}f\|_{\C^\mu(\R^d)}\,dr\\
&\le  C \|f\|_{\C^{\gamma}}\|w\|_{L_\infty([S,T])}|\phi_s-\psi_s|^\mu\int_s^t (r-s)^{\frac{\gamma-\mu}2}\,dr\\
&\le  C \|f\|_{\C^{\gamma}}\|w\|_{L_\infty([S,T])}|\phi_s-\psi_s|^\mu(t-s)^{1+\frac{\gamma-\mu}2}.
\end{align*}
for $C=C(\gamma,\mu,d)$. Therefore, by Jensen's inequality we get for $m\ge2$
\begin{equation}\label{s:spdec1}
\|A_{s,t}\|_{L_m(\Omega)}\le  C \|f\|_{\C^{\gamma}}\|w\|_{L_\infty([S,T])}\|\phi-\psi\|^\mu_{\C^0 L_m([S,T])}(t-s)^{1+\frac{\gamma-\mu}2}.
\end{equation}
By \eqref{paramcond2}, we have  $1+\frac{\gamma-\mu}2>1/2$. Therefore condition \eqref{con:s1} of \cref{T:SSLst} holds. 
	
Next, using \eqref{astb2}, we easily get for  $(s,u,t)\in\Delta^3_{[S,T]}$, $m\ge2$
	\begin{align}\label{asutb2}
		\|\E^s \delta A_{s,u, t}\|_{L_m(\Omega)}&\le \|\E^s \delta A^\phi_{s,u, t}\|_{L_m(\Omega)}+\|\E^s \delta A^\psi_{s,u, t}\|_{L_m(\Omega)}\nn\\
		&\le C \|f\|_{\C^{\gamma}}\|w\|_{L_\infty([S,T])} (t-s)^{\frac12+\frac\gamma2+\tau}( [\phi]_{\C^\tau L_m([S,T])}+ [\psi]_{\C^\tau L_m([S,T])}),
\end{align}
where $C=C(\gamma,d)$. We see from \eqref{paramcond} $\frac12+\frac\gamma2+\tau>1$ and thus \eqref{con:s2} holds.
	
Thus all the conditions of  the SSL hold and the desired bound \eqref{sb2} follows from \eqref{est:ssl1} and \eqref{s:spdec1}, \eqref{asutb2}.
	
To obtain  \eqref{sb1} and \eqref{sb2} for general $f$, we note that for any $\eps>0$ the function 
$f_\eps:=P_{\eps}f\in\C^\infty$. Therefore, bounds \eqref{sb1} and  \eqref{sb2} hold for $f_\eps$. Since $f_\eps$ converges pointwise to $f$ as $\eps\to0$  and $\|f_\eps\|_{\C^\gamma}\le\|f\|_{\C^\gamma}$, by Fatou's lemma we see  that these bounds hold also for $f$.	
\end{proof}

Now we can prove the corollary discussed at the beginning of the subsection. We bound  the moments of the right-hand side of \eqref{keyint:s} in terms of the difference of the processes multiplied by $\lambda$ to a negative power. 

\begin{corollary}\label{c:finfin}
Let  $m\ge2$, $\gamma\in(-1,0)$, $\tau\in(0,1]$, $\eps\in(0,\frac12)$. Let $f\colon\R^d\to\R$ be a bounded continuous function.
Let  $\phi,\psi\colon\Omega\times[0,1]\to\R^d$ be continuous measurable processes adapted to the filtration $(\F_t)_{t\in[0,1]}$.
	Assume that \eqref{paramcond} holds. Then there exists a constant $C=C(\gamma, \eps, \tau, d,m)>0$ such that for any $\lambda>1$, $(s,t)\in\Delta_{[0,1]}$  we have
	\begin{equation}\label{bound1terml}
		\Bigl\|\int_s^t e^{-\lambda(t-r)} f(W_r+\phi_r) \,dr\Bigr\|_{L_m(\Omega)}\le C \|f\|_{\C^\gamma}(t-s)^{\eps}\lambda^{-1-\frac\gamma2+3\eps}(1+ [\phi]_{\C^\tau L_m([s,t])}\lambda^{\frac12-\tau}).
	\end{equation}
If, additionally, for $\mu\in(0,1]$ condition \eqref{paramcond2} is satisfied, 
then there exists a constant $C=C(\gamma, \eps, \mu,\tau, d,m)>0$ such that for any $(s,t)\in\Delta_{[0,1]}$,
 we have 
\begin{align}\label{ourboundasregl}
&\Bigl\|\int_s^t e^{-\lambda(t-r)}(f(W_r+\phi_r)-f(W_r+\psi_r))\,dr\Bigr\|_{L_m(\Omega)}\nn\\
&\quad\le
C \|f\|_{\C^\gamma}
(t-s)^\eps\lambda^{-1-\frac{\gamma-\mu}2+3\eps}\|\phi-\psi\|^\mu_{\C^0 L_m ([s,t])}\nn\\
&\qquad+C\|f\|_{\C^\gamma} 	(t-s)^\eps \lambda^{-\frac12-\frac\gamma2-\tau+3\eps}([\phi]_{\C^\tau L_m([s,t])}+[\psi]_{\C^\tau L_m([s,t])}).
\end{align}
\end{corollary}
\begin{proof}
We only show how to deduce \eqref{bound1terml} from \eqref{sb1}. Bound \eqref{ourboundasregl} is derived from \eqref{sb2} using exactly the same argument.
	
We fix $\eps\in(0,1/2)$, $\lambda>1$, $(s,t)\in\Delta_{[0,1]}$ and, as discussed in \eqref{boundbound}, split the integral into 
\begin{align}\label{difftwoguys}
\Bigl\|\int_s^te^{-\lambda(t-r)}f(W_r+\phi_r)\,dr\Bigr\|_{L_m(\Omega)}&\le
\Bigl\|\int_s^{(t-\lambda^{-1+\eps})\vee s } e^{-\lambda(t-r)}f(W_r+\phi_r)\,dr\Bigr\|_{L_m(\Omega)}\nn\\
&\phantom{\le}+\Bigl\|\int_{(t-\lambda^{-1+\eps})\vee s}^t e^{-\lambda(t-r)}f(W_r+\phi_r)\,dr  \Bigr\|_{L_m(\Omega)}\nn\\
&=:I_1+I_2.
\end{align}
If $r\le t-\lambda^{-1+\eps}$, then $e^{-\lambda(t-r)}\le e^{-\lambda^\eps}$ and $\|e^{-\lambda (t-\cdot)}\|_{L_
\infty([s,(t-\lambda^{-1+\eps})\vee s])}\le e^{-\lambda^\eps}$. Therefore, \eqref{sb1} implies 
\begin{equation}\label{ionel}
I_1\le C e^{-\lambda^\eps} \|f\|_{\C^\gamma}
(t-s)^{\eps}(1+ [\phi]_{\C^\tau L_m([s,t])}).
\end{equation}
for $C=C(\gamma,\tau,d,m)$, where we also used that $1+\frac\gamma2>1/2>\eps$. 
	
To bound $I_2$, we use the obvious inequality  $\|e^{-\lambda (t-\cdot)}\|_{L_\infty([t-\lambda^{-1+\eps},t])}\le1$. 
Since 
\begin{equation*}
(t-((t-\lambda^{-1+\eps})\vee s))\le t-s \qquad\text{ and  }\qquad (t-((t-\lambda^{-1+\eps})\vee s))\le \lambda^{-1+\eps},
\end{equation*}
we also have for any $\rho\in[\eps,2]$
$$
(t-((t-\lambda^{-1+\eps})\vee s))^{\rho}\le (t-s)^\eps  \lambda^{(-1+\eps)(\rho-\eps)}\le
(t-s)^\eps  \lambda^{-\rho+3\eps}.
$$
Applying this inequality separately to the two terms in \eqref{sb1} with $\rho=1+\frac\gamma2$ and $\rho=\frac12+\frac\gamma2+\tau$, respectively, we obtain
\begin{align*}
I_2&\le C \|f\|_{\C^\gamma}(t-s)^{\eps}\lambda^{-1-\frac\gamma2+3\eps}+C \|f\|_{\C^\gamma}[\phi]_{\C^\tau L_m([s,t])}(t-s)^{\eps}\lambda^{-\frac12-\frac\gamma2-\tau+3\eps}\\
&=C \|f\|_{\C^\gamma}(t-s)^{\eps}\lambda^{-1-\frac\gamma2+3\eps}
(1+[\phi]_{\C^\tau L_m([s,t])}\lambda^{\frac12-\tau})
\end{align*}
for $C=C(\gamma,\tau,d,m)$. Combining this with \eqref{ionel} and substituting into \eqref{difftwoguys}, we obtain \eqref{bound1terml}; here we used that the function $\lambda\mapsto e^{-\lambda^\eps}$ decreases to $0$  faster than any negative polynomial function as $\lambda\to\infty$.
\end{proof}

\subsection{Weak existence}\label{s:we}

Before we proceed to the weak uniqueness of solutions to the SDE \eqref{distrsde}, let us first discuss the weak existence of solutions. 
In this subsection, we mostly follow the arguments of \cite[Section~8]{GG22}. 

Although the definition of a weak solution to an SDE with a distributional drift differs from the usual one (recall \cref{D:sol}), the proof strategy is similar. First, we establish an a priori regularity bound. Then we consider a sequence of solutions to approximated equations with regularized drifts and show that this sequence is tight. Finally, we show that any of its limit points is a weak solution. As always, we work with solutions to SDEs driven by a Brownian motion, but the same techniques can be easily transferred (at the cost of increased technical complexity) to SDEs driven by a non-Markovian process, for example, fractional Brownian motion.

The good news is that we do not need any new technical bounds; everything follows from the already established bound \eqref{sb1}.

Before we begin, we need to establish the following very useful technical tool.  Clearly, a uniform local bound on the H\"older norm of a function over all intervals of length at most $\ell>0$ yields a global bound on its Hölder norm, with a constant depending only on $\ell$. The next lemma establishes essentially the same result for the $[\cdot]_{\C^\tau([0,1];E)}$ seminorm, where $E$ is a Banach space. It will be later applied in the case $E=L_m(\Omega)$. 

\begin{lemma}\label{l:ltg}
	Let  $E$ be a Banach space, $\tau\in(0,1]$, $\Gamma\ge0$, $\ell>0$. Let $f$ be a measurable function $[0,1]\to E$. 
	\begin{enumerate}[(i)]
		\item 
		Assume that for every $(S,T)\in\Delta_{[0,1]}$ with $T-S\le \ell$ we have 
		\begin{equation}\label{genholder}
			[f]_{\C^\tau([S,T];E)}\le \Gamma.
		\end{equation}
		Then 
		\begin{equation}\label{res1}
			[f]_{\C^\tau([0,1];E)}\le  \Gamma ( \ell^{\tau-1} + 1).
		\end{equation}
		
		\item Let $M>0$. Suppose now that $f_0=0$ and  for every $(S,T)\in\Delta_{[0,1]}$ with $T-S\le \ell$ we have
		\begin{equation}\label{genholderm}
			[f]_{\C^\tau([S,T];E)}\le M\|f_S\|_{E}+\Gamma.
		\end{equation}
		Then there exists a constant $C=C(\tau,\ell,M)>0$ independent of $\Gamma$ such that
		\begin{equation*}
			[f]_{\C^\tau([0,1];E)}\le C\Gamma.
		\end{equation*}
	\end{enumerate}
\end{lemma}
\begin{proof}
	(i). Let $(s,t)\in\Delta_{[0,1]}$. Let $N$ be the smallest integer such that $(t-s)/N\le \ell$. We partition $[s,t]$ into $N$ subintervals $[t_i,t_{i+1}]$, $i=0,...,N-1$, each of length  $(t-s)/N$, and sum the increments of $f$ using \eqref{genholder}:
	\begin{equation*}
		\| f_t-f_s\|_{E} \le \sum_{i=0}^{N-1} \| f_{t_{i+1}}-f_{t_i}\|_{E}\le 
		N^{1-\tau} \Gamma (t-s)^\tau.
	\end{equation*}
	Dividing this by $(t-s)^\tau$ and using $N\le (t-s)/\ell + 1\le 1/\ell+1$, we get
	\begin{equation*}
		\frac{\| f_t-f_s\|_{E} }{(t-s)^\tau} \le \Gamma ( \ell^{\tau-1} + 1),
	\end{equation*}
	which yields \eqref{res1}.
	
	(ii).  Let $N=\lceil 1/\ell\rceil$. Fix a partition $0=t_0<t_1<\dots<t_N=1$ with $t_{i+1}-t_i\le \ell$. It follows from \eqref{genholderm}  that for any $i=0,\hdots,N-1$, $t\in[t_i,t_{i+1}]$
\begin{equation}\label{step1condlemma}
\|f_{t}\|_{E}\le\|f_{t_{i}}\|_{E}+(t-t_i)^\tau 
[f]_{\C^{\tau}([t_{i},t_{i+1}];E)}\le \|f_{t_{i}}\|_{E} (1+ \ell^\tau M)+\ell^\tau \Gamma.
\end{equation}
Starting from $i=0$, taking in the above bound $t=t_{i+1}$, and iterating  $k$ times, $k= 1,\hdots,N$, we get
\begin{equation*}
\|f_{t_k}\|_{E}\le \ell^\tau\Gamma\sum_{i=0}^{k-1}(1+M\ell^\tau)^i
\le \Gamma M^{-1} (1+M\ell^\tau)^k\le \Gamma M^{-1} (1+M\ell^\tau)^N=C \Gamma
\end{equation*}
for $C=C(\tau,\ell,M)$. 
Therefore, for any $t\in[t_k,t_{k+1}]$, \eqref{step1condlemma} gives
\begin{equation*}
\|f_t\|_E \le C\Gamma
\end{equation*}
for $C=C(\tau,\ell,M)$, which of course implies  $\sup_{t\in[0,1]}\|f_t\|_E\le C\Gamma$.
Substituting this into \eqref{genholderm}, we obtain, for every
$(S,T)\in\Delta_{[0,1]}$ with $T-S\le\ell$,
\begin{equation*}
[f]_{\C^\tau([S,T];E)}
\le M\|f_S\|_E+\Gamma
\le C\Gamma.
\end{equation*}
Part (i) now yields the desired global bound.
\end{proof}

Now we are ready to implement our weak existence proof strategy.

We begin with an a priori bound on the time regularity of the drift $\int b(X_r)\,dr$ of the solution to \eqref{distrsde}. Let us start with some heuristics, omitting as usual the arbitrarily small exponents. Clearly, $W \in \C^{\frac12}([0,1])$, which suggests that the solution $X$ to \eqref{distrsde} should also belong to $\C^{\frac12}([0,1])$. The function $b$ lies in $\C^{\gamma}(\R^d)$ for some negative $\gamma$, so, very informally, we may expect “$b(X)$” to be in $\C^{\frac\gamma2}([0,1])$. We gain $1$ in regularity by integration, so we expect that the drift $t \mapsto \int_0^t b(X_r)\,dr$ belongs to $\C^{1+\frac\gamma2}([0,1])$. This is essentially the statement of our first lemma, modulo the fact that, as usual, we measure regularity not a.s. in the $\C^\tau$ scale but rather in the $\C^\tau L_m$ scale; recall the corresponding definition in \eqref{norms}.

We derive a general a priori bound, which we will also reuse later in the uniqueness part of the proof.

\begin{lemma}[General a priori estimate]\label{lem.apriori}
Let $m\in[2,\infty)$, $\gamma\in(-\frac12,0)$, $x\in\R^d$. Let $(f_k)_{k\in \N}$ be a sequence of bounded continuous functions $\R^d\to\R^d$. Let $\psi, g,Z\colon\Omega\times [0,1]\to\R^d$ be measurable processes adapted to the filtration $(\F_t)_{t\in[0,1]}$ which satisfy 
\begin{equation*}
Z_t=x+\psi_t+\int_0^t g_r dr + W_t, \quad t\in[0,1]
\end{equation*}
and 
\begin{equation*}
	\sup_{t\in[0,1]}\Bigl|\psi_t- \int_0^t f_k(Z_r)\,dr\Bigr|\to0 \quad\text{in probability as } k\to\infty.
\end{equation*}
If
\begin{equation}\label{zcond}
[Z-W]_{\C^{1+\frac\gamma2}L_m([0,1])}<\infty
\end{equation} 
then 
\begin{equation}\label{psiodin}
[Z-W]_{\C^{1+\frac\gamma2}L_m([0,1])}\le C (1+ F)(F+ \|g\|_{\C^0L_m([0,1])})
\end{equation}	
for $C=C(\gamma,d,m)>0$ and $F:=\sup_k\|f_k\|_{\C^\gamma}$.
\end{lemma}

Note again that even though we assumed $f_k$ to be a sequence of bounded continuous functions, the $\C^{1+\frac\gamma2}L_m$ norm of the drift does not depend on the supremum norm of $f_k$. Instead, it depends on the much weaker norm $\|f_k\|_{\C^{\gamma}}$. This will allow us to pass to the limit later.
The condition $\gamma > -\frac12$ is precisely the condition $\frac\gamma2 + \tau > \frac12$ from \cref{l:finfin}, with $\tau = 1 + \frac\gamma2$, as motivated by the heuristics above.

\begin{proof}
The proof relies on bound  \eqref{sb1}. Fix $m\ge2$. First we note that  for any $(s,t)\in\Delta_{[0,1]}$ we get by Minkowski's inequality and Fatou's lemma
\begin{align}\label{z1step}	
\|(Z_t-W_t)-(Z_s-W_s)\|_{L_m(\Omega)}&\le \|\psi_t-\psi_s\|_{L_m(\Omega)}+ \int_s^t\|g_r\|_{L_m(\Omega)}\,dr\nn\\
&\le \liminf_{k\to\infty}\Bigl\|\int_s^t f_k(Z_r)\,dr\Bigr\|_{L_m(\Omega)} \!\!\!\!+ (t-s) \|g\|_{\C^0L_m([0,1])}.
\end{align}
For $k\in\N$ let us apply  \cref{l:finfin} with the following set of parameters: $f=f_k$, $w\equiv1$, $\tau=1+\frac\gamma2$. Since $\gamma>-\frac12$, we see that condition \eqref{paramcond} holds. Therefore we get from  \eqref{sb1} that there exists $C=C(\gamma,d,m)>0$ such that for any $(s,t)\in\Delta_{[0,1]}$
\begin{equation*}
\Bigl\|\int_s^t f_k(Z_r)\,dr\Bigr\|_{L_m(\Omega)}\le C F(t-s)^{1+\frac\gamma2}\bigl(1+[Z-W]_{\C^{1+\frac\gamma2}L_m([s,t])}(t-s)^{\frac12+\frac \gamma2}\bigr).
\end{equation*}
Substituting this into \eqref{z1step}, we get 
\begin{equation*}
\|(Z_t-W_t)-(Z_s-W_s)\|_{L_m(\Omega)}\le C(t-s)^{1+\frac\gamma2}\bigl(F+ \|g\|_{\C^0L_m([0,1])}+F[Z-W]_{\C^{1+\frac\gamma2}L_m([s,t])}(t-s)^{\frac12+\frac \gamma2}\bigr).
\end{equation*}
Fix now $(S,T)\in\Delta_{[0,1]}$.   By dividing both sides of the above inequality by  $(t-s)^{1+\frac\gamma2}$ and taking 
supremum over all $(s,t)\in\Delta_{[S,T]}$, we get (recall that $\gamma>-1/2$)
\begin{equation*}
[Z-W]_{\C^{1+\frac\gamma2}L_m([S,T])}\le  C_0 (F+ \|g\|_{\C^0L_m([0,1])}) +C_0F	[Z-W]_{\C^{1+\frac\gamma2}L_m([S,T])}(T-S)^{\frac14}
\end{equation*}
for $C_0=C_0(\gamma,d,m)>0$.  If $F=0$, then there is nothing to prove, the above inequality implies \eqref{psiodin}. Otherwise, we choose  $\ell>0$  such that
\begin{equation*}
C_0F\ell^{\frac14}= \frac12.
\end{equation*}
Recall that $[Z-W]_{\C^{1+\frac\gamma2}L_m([S,T])}$ is finite thanks to our assumption \eqref{zcond}. Then substituting this in the above inequality we get for any $(S,T)\in\Delta_{[0,1]}$ with $T-S\le \ell$
\begin{equation*}
[Z-W]_{\C^{1+\frac\gamma2}L_m([S,T])}\le  2C_0 (F+ \|g\|_{\C^0L_m([0,1])}).
\end{equation*}
By \cref{l:ltg}(i) this implies 
\begin{equation*}
[Z-W]_{\C^{1+\frac\gamma2}L_m([0,1])}\le  2C_0 (F+ \|g\|_{\C^0L_m([0,1])}) ( \ell^{-\frac14} + 1)\le  C (1+ F)(F+ \|g\|_{\C^0L_m([0,1])})
\end{equation*}
for $C=C(\gamma,d,m)$. This implies \eqref{psiodin}.
\end{proof}

Next, let us consider a sequence of strong solutions to approximated equations
\begin{equation}\label{xnseq}
X_t^n=x+\int_0^t b^n(X^n_r) dr +W_t,\quad t\in[0,1],
\end{equation}
where $b^n:=P_{1/n}b$, $n\in\N$. We see that $(b^n)_{n\in\N}$ is a sequence of smooth functions $\R^d\to\R^d$ converging to $b$ in $\C^{\gamma-\eps}$ for any $\eps>0$, see \cref{e:betamin}. Using \cref{lem.apriori}, let us now show that  the sequence $(X^n)_{n\in\N}$  is tight.

\begin{lemma}[Tightness]\label{lem.uun}
Suppose that $\gamma>-\frac12$. Then the sequence $(X^n)_{n\in\N}$ defined in \eqref{xnseq} is tight in the space $\C([0,1],\R^d)$.
\end{lemma}

\begin{proof}
Fix the initial data $x\in\R^d$ and denote for brevity $\phi^n:=X^n-W$, $n\in\N$. For a fixed $n\in\N$ we apply \cref{lem.apriori} with $Z:=X^n$, $\psi_t:=\int_0^t b^n(X^n_r)\,dr$, $g\equiv0$ and take $f_k:=b^n$ for every $k$. We see that for any $(s,t)\in\Delta_{[0,1]}$ we have
\begin{equation*}
|(X^n_t-W_t)-(X^n_s-W_s)|\le \|b^n\|_{L_\infty(\R^d)} |t-s|,
\end{equation*}
and therefore condition \eqref{zcond} is satisfied. Thus it follows from \eqref{psiodin} that for any $m\ge1$ there exists a 
constant $C=C(\gamma,d,m)>0$ such that for any $(s,t)\in\Delta_{[0,1]}$
\begin{equation*}
\|\phi^n_t-\phi^n_s\|_{L_m(\Omega)}\le C\|b^n\|_{\C^{\gamma}}  (1+ \|b^n\|_{\C^{\gamma}})|t-s|^{1+\frac\gamma2}
\le C\|b\|_{\C^{\gamma}}  (1+ \|b\|_{\C^{\gamma}})|t-s|^{1+\frac\gamma2},
\end{equation*}
where we also used that $\|b^n\|_{\C^{\gamma}}\le \|b\|_{\C^{\gamma}}$ by \cref{e:betamin}.

Therefore by the Kolmogorov continuity theorem (recall \cref{t:kolmi}), for any $\eps>0$ we have
\begin{equation}\label{kolmnorm}
\| [\phi^n]_{\C^{1+\frac\gamma2-\eps}([0,1],\R^d)}\|_{L_m(\Omega)}\le  C\|b\|_{\C^{\gamma}}  (1+ \|b\|_{\C^{\gamma}}),
\end{equation}
for $C=C(\gamma,\eps,d,m)>0$.

Denote now for $M>0$
\begin{equation*}
A_M:=\big\{f\in \C([0,1],\R^d): f(0)=x \text{ and } [f]_{\C^{\frac12+\frac\gamma2}([0,1];\R^d)}\le M\big\}.
\end{equation*}
By the Arzel\`a--Ascoli theorem, for each $M>0$ the set $A_M$ is a compact set in the space $\C([0,1],\R^d)$. Furthermore, it follows from \eqref{kolmnorm}  and  the Chebyshev inequality that 
$$
\P(\phi^n \notin A_M)\le   M^{-1}\E [\phi^n]_{\C^{\frac12+\frac\gamma2}([0,1];\R^d)}\le  C\|b\|_{\C^{\gamma}}  (1+ \|b\|_{\C^{\gamma}})M^{-1}.
$$
Thus, the sequence $(\phi^n)_{n\in\N}$ is tight in $\C([0,1],\R^d)$. Since the process $W$ does not depend on $n$, we get that 
 $(X^n)_{n\in\N}$ is also  tight in $\C([0,1],\R^d)$.
\end{proof}

Now we are finally ready to show weak existence of the solutions to SDE  \eqref{distrsde}. Namely, as promised in the beginning of this subsection, we will show that any of the limiting points of the sequence $(X^n)$ is a weak solution, that is, it satisfies the conditions of \cref{D:sol}. 
We will rely again on the very powerful a priori estimates (\cref{lem.apriori}).

\begin{theorem}\label{t:WESDE}
Let $\gamma\in(-\frac12,0)$, $x\in\R^d$, $b\in\C^\gamma$. Then equation   \eqref{distrsde} has a weak solution in the sense of \cref{D:sol}.
\end{theorem}

\begin{proof}
\textbf{Step 1}. We know from \cref{lem.uun} that the sequence $(X^n)_{n\in\N}$ is tight in $\C([0,1],\R^d)$. Obviously, the constant sequence $(W)_{n\in\N}$ is tight in this space as well. Hence the sequence $(X^n,W)_{n\in\N}$ is tight in $\C([0,1],\R^{2d})$.
Since this space is Polish, by the Prokhorov theorem (see, e.g., \cite[Theorem 3.2.1]{Shiryaevpr}) there exists
a subsequence $(n_k)_{k\in\N}$ such that $(X^{n_k}, W)_{k\in\N}$ converges weakly in the space $\C([0,1],\R^{2d})$. By passing to an appropriate subsequence and applying the Skorokhod representation theorem  (see, e.g., \cite[Theorem 3.8.1]{Shiryaevpr}), we see that there exists a random element $(\wh X,\wh W)$ and a sequence of random elements $(\wh X^n,\wh W^n)$ defined on a common probability space $(\wh\Omega, \wh\F, \wh P)$
such that $\law(\wh X^n,\wh W^n)=\law( X^n, W)$ and
\begin{equation}\label{prishli}
	\|(\wh X^n,\wh W^n)-(\wh X,\wh W)\|_{\C([0,1];\R^{2d})}\to0 \quad \text{as $n\to\infty$ a.s.}
\end{equation}

Now let us show that $(\wh X,\wh W)$ is a weak solution to  \eqref{distrsde}. We will need to check two things: a) that for  some filtration $\wh\FF$ the process  $\wh X$ is adapted to $\wh \FF$ and  $\wh W$ is an $\wh \FF$-Brownian motion; b) that $(\wh X,\wh W)$ satisfies \cref{D:sol}. 

\textbf{Step 2a}. 
Obviously, we have  $\law(\wh W)=\law(\wh W^n)=\law(W)$, and thus, $\wh W^n$, $\wh W$ are standard Brownian motions. For $t\in[0,1]$, put $\wh \F_t:=\sigma(\wh W_s,\wh X_s; s\le t)$ and consider the filtration  $\wh\FF:=(\wh \F_t)_{t\in[0,1]}$. By definition, $\wh X$ is adapted to $\wh \FF$. We claim now that $\wh W$ is an $\wh \FF$-Brownian motion.

Indeed, for any $(s,t)\in\Delta_{[0,1]}$, $n,N\in\N$, any bounded continuous functions $f\colon\R^d\to\R$, 
$g\colon\R^{2Nd}\to\R$, and any time points $0\le t_1\le\hdots\le t_N\le s$ one has
\begin{align*}
&\E f(\wh W^n(t)-\wh W^n(s)) g\bigl(\wh W^n(t_1),\hdots, \wh W^n(t_N), \wh X^n(t_1),\hdots, \wh X^n(t_N)\bigr)\\
&\quad=\E  f(\wh W^n(t)-\wh W^n(s))\E  g\bigl(\wh W^n(t_1),\hdots, \wh W^n(t_N), \wh X^n(t_1),\hdots, \wh X^n(t_N)\bigr),
\end{align*}
since $\wh W^n(t)-\wh W^n(s)$ is independent of $\wh \F^n_s:=\sigma(\wh W^n_r, \wh X^n(r); r\le s)$ . By passing to the limit in the above expression as $n\to\infty$ using \eqref{prishli}, we derive that the same identity holds also for $(\wh W,\wh X)$. Therefore $\wh W(t)-\wh W(s)$ is independent of $\wh \F_s$ for any $(s,t)\in\Delta_{[0,1]}$. Thus, $\wh W$ is an $\wh \FF$-Brownian motion.

\textbf{Step 2b}. Let us define  
\begin{equation}\label{psiprdef}
\psi_t:= \wh X_t-x- \wh W_t,\quad t\in[0,1].
\end{equation}
We see that part (i) of  \cref{D:sol} holds by construction.

To check part (ii) of  \cref{D:sol}, fix an arbitrary sequence  $(f^n)_{n\in\N}$ of $\C^\infty(\R^d,\R^d)$ functions converging to $b$ in $\C^{\gamma-}$. We put for $n,k\in\N$, $t\in[0,1]$
\begin{align}\label{Psin}
&\psi_n(t):=\int_0^t f^n(\wh  X(r))dr,\qquad \psi_{n,k}(t):=\int_0^t f^n(\wh  X^k(r))dr.
\end{align}
Our goal is to prove that
\begin{equation*}
\lim_{n\to\infty}\sup_{t\in[0,1]}\Big|\int_0^t f^n(\wh X_r)\,dr-\psi(t)\Big|=
\lim_{n\to\infty}\sup_{t\in[0,1]}\Big| \psi_n(t)-\psi(t)\Big|\stackrel{???}{=}0\,\,\text{in probability}.
\end{equation*}
We split this expression into several parts. Let $n,k\in\N$. We have 
\begin{align}\label{baza0}
\sup_{t\in[0,1]} |\psi_n(t)-\psi(t)|&\le
\sup_{t\in[0,1]} |\psi_n(t)-\psi_{n,k}(t)|+
\sup_{t\in[0,1]} \bigl|\psi_{n,k}(t)-\int_0^t b^k(\wh X^k(r))\,dr\bigr|\nn\\
&\phantom{\le} +\sup_{t\in[0,1]} \bigl|\int_0^t b^k(\wh X^k(r))\,dr-\psi(t)\bigr|\nn\\
&=:I_1(n,k)+I_2(n,k)+I_3(k).
\end{align}
It is easy to deal with the first term. We have for any fixed $n\in\N$
\begin{equation}\label{firstbaza}
\lim_{k\to\infty} I_1(n,k)\le \lim_{k\to\infty}\|f^n\|_{\C^1(\R^d,\R^d)}\|\wh X-\wh X^k\|_{\C([0,1];\R^d)}=0 
\end{equation}
by \eqref{prishli}. Similarly, it is also immediate to treat the last term. 
 Recall that  $\wh X^k=x+\int_0^\cdot b^k(\wh X^k(r))\,dr+\wh W^k$. Hence, by definition of $\psi$ in \eqref{psiprdef}, we have
\begin{equation}\label{thirdbaza}
\lim_{k\to\infty} I_3(k)= \lim_{k\to\infty} \|\wh X^k-\wh W^k-(\wh X-\wh W)\|_{\C([0,1];\R^d)}=0 
\end{equation}
thanks to the convergence  \eqref{prishli}.

It remains to bound $I_2(n,k)$. This requires a little more work, but we will rely on the a priori bound and the Kolmogorov continuity theorem to interchange the order of taking the supremum in time and the expectation. Thus, we fix $\eps\in(0,1/2)$ so that $\gamma > -\frac12 + \eps$.
 We apply \cref{l:finfin} with $f=f^n-b^k$, $\gamma-\eps$ in place of $\gamma$, $\phi=\wh X^k-\wh W^k$, $\tau=1+\frac\gamma2$, $w\equiv1$. We see that condition \eqref{paramcond} is satisfied since $\gamma>-\frac12+\eps$. Therefore it follows from  \eqref{sb1} that for any $m\ge2$ there exists a constant $C=C(\gamma,\eps,d,m)>0$ such that for any $(s,t)\in\Delta_{[0,1]}$, $n,k\in\N$ we have
\begin{align*}
&\|(\psi_{n,k}(t)-\psi_{n,k}(s))-\int_s^t b^k(\wh X^k(r))\,dr\|_{L_m(\Omega)}\\
&\quad=
\Bigl\|\int_s^t (f^n-b^k)(\wh  X^k(r))\,dr\Bigr\|_{L_m(\Omega)}\\
&\quad\le C\|f^n-b^k\|_{\C^{\gamma-\eps}}|t-s|^{1+\frac\gamma2-\eps}(1+ [\wh X^k-\wh W^k]_{\C^{1+\frac\gamma2} L_m([0,1])}\bigr)\\
&\quad\le C\|f^n-b^k\|_{\C^{\gamma-\eps}}|t-s|^{1+\frac\gamma2-\eps}
\end{align*}
for $C=C(\|b\|_{\C^\gamma},\gamma,\eps,d,m)$. Here in the last inequality we used the a priori bound in \cref{lem.apriori}. Therefore, 
by the Kolmogorov continuity theorem (\cref{t:kolmi}) we finally get 
\begin{align*}
\|I_2(n,k)\|_{L_2(\Omega)}	&=	\Bigl\|\sup_{t\in[0,1]} \Bigl|\psi_{n,k}(t)-\int_0^t b^k(\wh X^k(r))\,dr\Bigr|\,\Bigr\|_{L_2(\Omega)}\\
&\le
\bigl\|\bigl[\psi_{n,k}-\int_0^\cdot b^k(\wh X^k(r))\,dr\bigr]_{\C^{\frac12+\frac\gamma2}}\,\bigr\|_{L_2(\Omega)}\le C\|f^n-b^k\|_{\C^{\gamma-\eps}}\to0\,\,\text{as $k,n\to\infty$}.
\end{align*}
Combining this with \eqref{firstbaza} and \eqref{thirdbaza} and passing to the limit in \eqref{baza0} first as $k\to\infty$ and then as $n\to\infty$, we finally get
$$
\sup_{t\in[0,1]} |\psi_n(t)-\psi(t)|\to0,\quad \text{in probability as $n\to\infty$},
$$
and thus part (ii) of  \cref{D:sol} holds.

\textbf{Step 3}. Thus, we conclude. By Step 2a,  $\wh X$ is adapted to the filtration $\wh \FF$ and  $\wh W$ is an $\wh \FF$-Brownian motion. By Step 2b the pair  $(\wh X,\wh W)$ satisfies \cref{D:sol}. Hence $(\wh X,\wh W)$ is a weak solution to equation \eqref{distrsde}.
\end{proof}

Let us discuss the obtained result a bit. We have just shown that any limiting point of the sequence $(X^n)$ is a weak solution to \eqref{distrsde}. For this, we used the standard SSL to obtain the integral bound \cref{l:finfin}, which in turn implied tightness. In the next subsection, we show that all limiting points actually have the same law. Moreover, we will show that if $X$ is \textit{any} weak solution, then our fixed sequence of approximations $(X^n)$ converges weakly to $X$, which yields weak uniqueness. To establish this, the SSL alone will not suffice: we will also rely on the generalized coupling technique.

As is standard in the analysis of SDEs or SPDEs with distributional drift, we do not consider all solutions of \eqref{distrsde}, but rather restrict ourselves to those whose drifts have a certain regularity, and later establish a \textit{conditional} weak uniqueness (that is, weak uniqueness only among such solutions). More precisely, we consider the following class of functions.

\begin{definition} Let $\kappa\in(0,1]$. We say that a solution $(X,W)$ to \eqref{distrsde} belongs to the class $\V(\kappa)$ if for any $m\ge2$ we have for $\psi:=X-W$
\begin{equation*}
[\psi]_{\C^\kappa L_m([0,1])}=\sup_{(s,t)\in\Delta_{[0,1]}} \frac{\|\psi_t-\psi_s\|_{L_m(\Omega)}}{|t-s|^\kappa}<\infty.
\end{equation*}	
\end{definition}

We see that, as usual, it is more convenient to measure the regularity of the drift in the $\C^\kappa L_m$ norm rather than in the $\C^\kappa$ norm. It is also almost immediate to see  that any limiting point of the sequence $(X^n)$ defined in \eqref{xnseq} belongs to $\V(1+\frac\gamma2)$. Let us provide a short proof.

\begin{lemma} Let $(X,B)$ be a limiting point (in law in the space $\C([0,1];\R^{2d})$) of the sequence $(X^n,W)_{n\in\N}$ defined in \eqref{xnseq}. Then $(X,B)\in\V(1+\frac\gamma2)$. 
\end{lemma}
\begin{proof}
Let $m\ge1$. We have from \cref{lem.apriori} and \cref{e:betamin} for any $n\in\N$
\begin{equation*}
[X^n-W]_{\C^{1+\frac\gamma2}L_m([0,1])}\le C\|b\|_{\C^{\gamma}}  (1+ \|b\|_{\C^{\gamma}})
\end{equation*}
for $C=C(\gamma,d,m)$.
Therefore, passing to a subsequence converging to $(X,B)$, we get by \cref{e:lscnorm} 
\begin{equation*}
[X-B]_{\C^{1+\frac\gamma2}L_m([0,1])}\le\liminf_{n\to\infty}[X^n-W]_{\C^{1+\frac\gamma2}L_m([0,1])}\le  C\|b\|_{\C^{\gamma}}  (1+ \|b\|_{\C^{\gamma}}),
\end{equation*}
which implies that $(X,B)\in\V(1+\frac\gamma2)$.
\end{proof}

\subsection{Stochastic sewing with generalized couplings for weak uniqueness}\label{s:stochs}

We now have all the tools to rigorously implement the strategy sketched in \cref{s:sketch} and establish weak uniqueness of solutions to \eqref{distrsde} that belong to the class $\V(1+\frac\gamma2)$.

Thus, let $(X,W)$ be \textit{any} weak solution to  \eqref{distrsde} in the sense of \cref{D:sol}. Assume that $(X,W)\in\V(1+\frac\gamma2)$. Let $X^n$, $n\in\N$, be a solution to the regularized equation \eqref{distrsden}. Introduce an auxiliary process $\wt X^{n,\lambda}$ which is defined on the same space as $(X,W)$ and solves
\begin{equation}\label{Xntildedef}
\wt X^{n,\lambda}_t =x+\int_0^t b^n(\wt X^{n,\lambda}_r) dr+\lambda\int_0^t (X_r-\wt X^{n,\lambda}_r)dr+W_t,\quad t\in[0,1],
\end{equation}
where the parameter $\lambda>1$ will be fixed later. We see that since the function $b^n$ is smooth, the process $\wt X^{n,\lambda}$ is well-defined. We decompose the distance between $X$ and $X^n$ as
\begin{align}\label{decompos}
\W_{\|\cdot\|\wedge1}(\law(X^n),\law(X))&\le \W_{\|\cdot\|\wedge1}(\law(X^n),\law(\wt X^{n,\lambda}))+\W_{\|\cdot\|\wedge1}(\law(\wt X^{n,\lambda}),\law(X))\nn\\
&\le d_{TV}(\law(X^n),\law(\wt X^{n,\lambda}))+\E \|\wt X^{n,\lambda}-X\|_{\C^0([0,1];\R^d)},
\end{align}
where in the last inequality we used \eqref{wrhocon} and the definition of the Wasserstein distance. 
 
As we explained in \cref{s:sketch}, to bound $d_{TV}(\law(X^n),\law(\wt X^{n,\lambda}))$, we use the Girsanov theorem and Pinsker's inequality. The following general statement provides bounds on the KL divergence  between the solutions of two SDEs.

\begin{lemma}({\cite[Theorem~A.2]{BKS18}}, {\cite[Theorem 8]{UUU09}}, {\cite[Proposition 1]{Lehec}})\label{l:dtvlemma}
Let $f\in\C^1(\R^d,\R^d)$. Let $g\colon\Omega\times[0,1]\to\R^d$ be a measurable adapted function. 
Let $Y,Z\colon\Omega\times[0,1]\to\R^d$ be strong solutions to the following SDEs, correspondingly:
\begin{equation*}
dY_t = f(Y_t)dt +dW_t;\quad dZ_t = f(Z_t)dt +g_t dt + dW_t,\qquad t\in[0,1]
\end{equation*}
with the initial condition $Y_0=Z_0=x\in\R^d$. Then 
\begin{equation}\label{dtvboundg}
D_{KL}\bigl(\law(Z)\|\law(Y)\bigr)\le \frac12\int_0^1 \E |g_s|^2\,ds.
\end{equation}
\end{lemma}

\begin{proof}
Since the additional drift term $g$ might be unbounded, we combine the Girsanov theorem with a localization argument. 

Without loss of generality we can assume that $\int_0^1 |g_s|^2\,ds<\infty$ a.s.; otherwise bound \eqref{dtvboundg} is trivial. 
For $N>0$ consider a stopping time 
\begin{equation*}
\tau_N:= \inf\{t\in[0,1]: \int_0^t |g_r|^2\,dr \ge N\}
\end{equation*}
and put 
\begin{equation*}
\wt W_t^{N}:=W_t+\int_0^t g_r \I(r\le \tau_N)\,dr.
\end{equation*}
By definition, the process $t\mapsto g_t \I(t\le \tau_N)$ satisfies the Novikov condition and therefore 
by the Girsanov theorem, there exists an equivalent probability measure $\wt \P^{N}$ such that the process $\wt W^{N}$ is an $(\F_t)_{t\in[0,1]}$- Brownian motion under this measure and 
\begin{equation}\label{denratio}
	\frac{d \wt \P^{N}}{d \P}=\exp\Bigl(-\int_0^1 g_r  \I(r\le \tau_N) \,d W_r-\frac12\int_0^1 |g_r|^2 \I(r\le \tau_N)\,dr\Bigr).
\end{equation}	
Let $Z^N$ be a strong solution to SDE 
\begin{equation}\label{Yneq}
Z^N(t)=x+ \int_0^t f(Z^N(r))\,dr +\wt W_t^{N},\quad t\in[0,1]
\end{equation}
on the space $(\Omega,\F,\wt \P^N)$. $Z^N$ is well-defined since the function $f$ is Lipschitz by assumption. 
We see that  $( Z^N,\wt W^{N})$ is a strong solution to SDE \eqref{Yneq} on $(\Omega,\F,\wt \P^N)$ and $(Y,W)$ is a strong solution to exactly the same equation on $(\Omega,\F,\P)$. Since weak uniqueness holds for SDE \eqref{Yneq}  (recall that the function $f$ is Lipschitz)  we have $\law_{\P}(Y)=\law_{\wt \P^N}(Z^N)$.
Then, using the data processing inequality (\cref{pinsker}(ii)) and the  explicit formula for the density \eqref{denratio} we derive
\begin{align}\label{term1gir}
D_{KL}(\law_{\P}(Z^N)\|\law_{\P}(Y))&=
D_{KL}(\law_{\P}(Z^N)\|\law_{\wt \P^N}(Z^N))\le D_{KL}(\P\|\wt \P^N)\nn\\
&=\E^{\P}\log \frac{d \P}{d \wt \P^{N}}=\frac12\int_0^1 \E^\P\bigl[|g_r|^2  \I(r\le \tau_N)\bigr]\,dr\nn\\
&\le \frac12\int_0^1 \E^\P|g_r|^2\,dr.
\end{align}	
Here we used that $\E^{\P} \int_0^1 g_r  \I(r\le \tau_N) \,d W_r=0$ since $\E^{\P} \int_0^1 |g_r|^2  \I(r\le \tau_N) \,dr$ is finite by definition.

Note that on the set $\{\tau_N=\infty\}$ the processes $Z^N$ and $Z$ coincide. Therefore,
\begin{equation*}
\lim_{N\to\infty}d_{TV}(\law_\P(Z^N),\law_\P(Z))\le \lim_{N\to\infty}\P(\tau_N<\infty)=0
\end{equation*}
since $\int_0^1 |g_s|^2\,ds<\infty$ a.s. Thus, by lower semicontinuity of KL divergence (\cref{p:lowersc}), we finally get using inequality \eqref{term1gir}
\begin{equation*}
D_{KL}(\law_{\P}(Z)\|\law_{\P}(Y))\le\liminf_{N\to\infty}D_{KL}(\law_{\P}(Z^N)\|\law_{\P}(Y))\le  \frac12\int_0^1 \E^\P|g_r|^2\,dr.\qedhere 
\end{equation*}
\end{proof}

Combining \cref{l:dtvlemma} with the Pinsker inequality, we bound the first term on the right-hand side of~\eqref{decompos}.
\begin{corollary}\label{l:ks1}
For any $n\in \N$, $\lambda>1$, we have
\begin{equation*}
d_{TV}(\law(X^n),\law(\wt X^{n,\lambda}))\le \frac12 \lambda \| \|X-\wt X^{n,\lambda}\|_{\C^0([0,1];\R^d)} \|_{L_2(\Omega)}.
\end{equation*}	
\end{corollary}
\begin{proof}
We apply \cref{l:dtvlemma} with $Y=X^n$, $Z=\wt X^{n,\lambda}$, $f=b^n$, $g=\lambda (X-\wt X^{n,\lambda})$. We get 
\begin{align*}
D_{KL}\bigl(\law(\wt X^{n,\lambda})\|\law(X^n)\bigr)\le \frac12\lambda^2\E \sup_{s\in[0,1]} |X_s-\wt X^{n,\lambda}_s|^2.
\end{align*}
An application of Pinsker's inequality (\cref{pinsker}(i)) yields the claim.
\end{proof}

Now we move on to bounding the second term on the right-hand side of~\eqref{decompos}. Fortunately, we are completely prepared by now and can rely on the integral bounds in \cref{l:finfin}, which were established using the SSL.

\begin{lemma}\label{L:ks2}
For any $\eps>0$ such that 
\begin{equation}\label{epsineq}
9\eps<\frac12+\gamma.
\end{equation}
there exist constants $C=C(\eps)$, $\lambda_0=\lambda_0(\|b\|_{\C^\gamma},\gamma,\eps,d,m)>1$ such that for any
$\lambda> \lambda_0$
one has 
\begin{equation}\label{mainboundxmy}
 \bigl\| \sup_{t\in[0,1]}|X_t-\wt X^{n,\lambda}_t |\bigr\|_{L_2(\Omega)}\le    C  \lambda^{-\frac32-\gamma+9\eps}+C \|b^n-b\|_{\C^{\gamma-\eps}}\lambda^{-1-\frac\gamma2+7\eps}.
 \end{equation}	
\end{lemma}	

The key point of the lemma is that $\lambda$ appears in the first term on the right-hand side of \eqref{mainboundxmy} with an exponent  smaller than $-1$ (recall that we would like to establish weak uniqueness in the regime $\gamma>-\frac12$). Therefore, we can ``pay the price'' of $\lambda$ when we plug this bound into \cref{l:ks1} and then into inequality \eqref{decompos}. While $\lambda$ appears in the second term on the right-hand side of \eqref{mainboundxmy} with an exponent  larger than $-1$, this is not relevant, as $\|b^n-b\|_{\C^{\gamma-\eps}}$ converges to $0$ as $n\to\infty$ and this term will disappear.

\begin{proof}
To bound the $\E \|\wt X^{n,\lambda}-X\|_{\C([0,1];\R^d)}$ , we rely, as usual, on the Kolmogorov continuity theorem. That is, we first bound $\|\wt X^{n,\lambda}-X\|_{\C^\eps L_m([0,1])}$ for some very small $\eps > 0$ and sufficiently large $m$, and then interchange the order of taking the expectation and the $\C^\eps([0,1])$ norm.

We take $m:=2/\eps\ge2$ and our goal is to show that $\|\wt X^{n,\lambda} - X\|_{\C^\eps L_m([0,1])}$ can be bounded by a half of itself plus $\lambda^{-1-\delta}$ for some $\delta>0$. To achieve this  we will follow the strategy of \cref{s:sketch} with the following technical modification: since $b$ is a distribution, we cannot even define the right-hand side of \eqref{keyint:s}. To make everything rigorous, we introduce yet another layer of approximations. Thus, for $k\in\N$ we put
\begin{equation*}
Z^k_t:=x+\int_0^t  b^k(X_r)\,dr +W_t, \quad t\in[0,1].
\end{equation*}
By definition of a solution to SDE \eqref{distrsde}, (recall \cref{d:apprsec}) and by passing to a subsequence if necessary, we have 
\begin{equation}\label{conv}
\|Z^k - X\|_{\C([0,1];\R^d)}\to0\quad a.s.,\quad  \text{as $k\to\infty$}. 	
\end{equation}	
We should think of $k$ as being much larger than $n$ which is fixed; in the  proof we will pass to the limit as $k\to\infty$. Recall that the process $\wt X^{n,\lambda}$ is defined in \eqref{Xntildedef}.

\textbf{Step~1}. By definition of $Z^k$ and $\wt X^{n,\lambda}$, we have 
\begin{equation*}
d(Z^k_t-\wt X^{n,\lambda}_t)=-\lambda (X_t-\wt X^{n,\lambda}_t)dt+ (b^k(X_t)-b^n(\wt X^{n,\lambda}_t))\,dt.
\end{equation*}
	
Hence, by the chain rule
\begin{align*}
d[e^{\lambda t}(Z^k_t-\wt X^{n,\lambda}_t)]&=e^{\lambda t} (b^k(X_t)-b^n(\wt X^{n,\lambda}_t))\,dt+
e^{\lambda t} \lambda (\wt X_t^{n,\lambda}-X_t+Z_t^k-\wt X_t^{n,\lambda})\,dt\\
&=e^{\lambda t} (b^k(X_t)-b^n(\wt X^{n,\lambda}_t))\,dt+
e^{\lambda t} \lambda (Z_t^k-X_t)dt.
\end{align*}
Let $(s,t)\in\Delta_{[0,1]}$. Integrating the above identity in time from $s$ to $t$ and dividing both sides of the equation by $\exp(\lambda t)$,  we derive
\begin{align*}
&|(Z^k_t-\wt X^{n,\lambda}_t)-e^{-\lambda(t-s)}(Z^k_s-\wt X^{n,\lambda}_s)|\\
&\qquad\le\Bigl|\int^t_s e^{-\lambda (t-r)} (b^k(X_r)-b^n(\wt X^{n,\lambda}_r))\,dr\Bigr|+\lambda
\int^t_s|Z_r^k-X_r|\,dr\\
&\qquad\le \Bigl|\int^t_s e^{-\lambda (t-r)} (b^n(X_r)-b^n(\wt X^{n,\lambda}_r))\,dr\Bigr|		+\Bigl|\int_s^t e^{-\lambda (t-r)}(b^n(X_r)-b^k(X_r)) \, dr\Bigr|\\
&\qquad\phantom{\le}+\lambda \|Z^k - X\|_{\C[0,1]}\\
&\qquad=: I_{1,n}(s,t)+I_{2,n,k}(s,t)+\lambda \|Z^k - X\|_{\C([0,1];\R^d)}.
\end{align*}
Let us pass to the limit in this inequality as $k\to\infty$. Recalling \eqref{conv}, we get	
\begin{equation*}
|(X_t-\wt X^{n,\lambda}_t)-e^{-\lambda(t-s)}(X_s-\wt X^{n,\lambda}_s)|\le  I_{1,n}(s,t)+\liminf_{k\to\infty} I_{2,n,k}(s,t).
\end{equation*}
Therefore,
\begin{align}\label{limeq}
|(X_t-\wt X^{n,\lambda}_t)-(X_s-\wt X^{n,\lambda}_s)|&\le |(X_t- \wt X^{n,\lambda}_t)-e^{-\lambda(t-s)}(X_s-\wt X^{n,\lambda}_s)|+(1-e^{-\lambda(t-s)})|X_s-\wt X^{n,\lambda}_s|\nn\\
&\le I_{1,n}(s,t)+\liminf_{k\to\infty} I_{2,n,k}(s,t)+\lambda^{\eps}|t-s|^{\eps}|X_s-\wt X^{n,\lambda}_s|,
\end{align}
where we used the elementary inequality $1-e^{-a}\le a^\rho$ valid for any $a\ge0$, $\rho\in[0,1]$. 

\textbf{Step~2a}.  Let us now analyze each of the terms in the above inequality. We begin with $I_{1,n}$. This is the term for which we have tailored \cref{c:finfin}, and we apply it now with  $\tau=1+\frac\gamma2$, $\mu=\frac12$,  $f=b^n$. We see that conditions \eqref{paramcond} and 
\eqref{paramcond2} hold thanks to our choice of parameters and the standing assumption $\gamma>-\frac12$. Therefore we get by bound~\eqref{ourboundasregl}, using as usual that $\|b^n\|_{\C^\gamma}\le \|b\|_{\C^\gamma}$:
\begin{align}\label{ivanpred}
\|I_{1,n}(s,t)\|_{L_m(\Omega)}&\le C (t-s)^\eps\lambda^{-\frac34-\frac\gamma2+3\eps}\|X-\wt X^{n,\lambda}\|_{\C^0L_m([0,1])}^{\frac12}\nn\\
&\quad+C (t-s)^\eps\lambda^{-\frac32-\gamma+3\eps}\bigl([X-W]_{ \C^{1+\frac\gamma2}L_m([0,1])}+[\wt X^{n,\lambda}-W]_{ \C^{1+\frac\gamma2}L_m([0,1])}\bigr)
\end{align}
for $C=C(\|b\|_{\C^\gamma},\gamma,\eps,d,m)$ independent of $\lambda$. 

To bound $[X-W]_{ \C^{1+\frac\gamma2}L_m([0,1])}$
we apply an a priori bound from \cref{lem.apriori}. We take there $X=Z$, $f_k=b^k$, $g\equiv0$ and we see that condition \eqref{zcond} is precisely our assumption $X\in\V(1+\frac\gamma2)$. Therefore we get for any $m\ge2$ 
\begin{equation}\label{phib}
[X-W]_{ \C^{1+\frac\gamma2}L_m([0,1])}\le C.
\end{equation}	
for $C=C(\|b\|_{\C^\gamma},\gamma,d,m)>0$.

To bound $[\wt X^{n,\lambda}-W]_{ \C^{1+\frac\gamma2}L_m([0,1])}$ we also apply  \cref{lem.apriori}. We take now $Z=\wt X^{n,\lambda}$, $f_1=f_2=...=b^n$, $g=\lambda (X-\wt X^{n,\lambda})$. It is easy to see that  $[\wt X^{n,\lambda}-W]_{\C^1 L_m([0,1])}<\infty$ and thus condition \eqref{zcond} holds. Hence, it follows from \eqref{psiodin} that 
\begin{equation}\label{psib}
[\wt X^{n,\lambda}-W]_{ \C^{1+\frac\gamma2}L_m([0,1])}\le C(1+\lambda \|X-\wt X^{n,\lambda}\|_{\C^0 L_m([0,1])})
\end{equation}
for $C=C(\|b\|_{\C^\gamma},\gamma,d,m)>0$.

Now we substitute \eqref{phib} and \eqref{psib} back into \eqref{ivanpred}. Recall that we are allowed to bound $I_{1,n}$ only by terms of the form $\lambda$ to a negative power or $\|X-\wt X^{n,\lambda}\|_{\C^0L_m([0,1])}$ to the power $1$. Since the first term on the right-hand side of \eqref{ivanpred} is not of that form, we bound it using the inequality $xy \le  x^2 + \frac14 y^2$ for $x=C\lambda^{-\frac34-\frac\gamma2+3\eps}$, $y=  \|X-\wt X^{n,\lambda}\|_{\C^0 L_m([0,1])}^{\frac12}$. We derive
\begin{align}\label{ivanpred2}
\|I_{1,n}(s,t)\|_{L_m(\Omega)}&\le C(t-s)^\eps \lambda^{-\frac32-\gamma+6\eps }+\frac14(t-s)^\eps \|X-\wt X^{n,\lambda}\|_{\C^0L_m([0,1])}\nn\\
&\phantom{\le}+C(t-s)^\eps \lambda^{-\frac12-\gamma+3\eps}\|X-\wt  X^{n,\lambda}\|_{\C^0L_m([0,1])}\nn\\
&\le C (t-s)^\eps \lambda^{-\frac32-\gamma+6\eps}+(t-s)^\eps \|X-\wt  X^{n,\lambda}\|_{\C^0L_m([0,1])}(\frac14+C\lambda^{-2\eps})
\end{align}
for $C=C(\|b\|_{\C^\gamma},\gamma,\eps,d,m)$ and in the last step, we used the inequality \eqref{epsineq} for our choice of $\eps$. We see that all the terms are now in the desired form.

\textbf{Step~2b}. Next, we move to the term $I_{2,n,k}$ from \eqref{limeq}. We apply \cref{c:finfin} with $\gamma-\eps$ in place of $\gamma$, $\tau=1+\frac\gamma2$, $f=b^n-b^k$. We see that condition \eqref{paramcond} holds and  we get  
from \eqref{bound1terml} and \eqref{phib}
\begin{align}\label{i2nk}
\|I_{2,n,k}(s,t)\|_{L_m(\Omega)}&\le 	C \|b^n-b^k\|_{\C^{\gamma-\eps}}(t-s)^\eps \lambda^{-1-\frac\gamma2+4\eps}\Bigl(1+[X-W]_{ \C^{1+\frac\gamma2}L_m([0,1])}\lambda^{-\frac12-\frac\gamma2}\Bigr)\nn\\
&\le C \|b^n-b^k\|_{\C^{\gamma-\eps}}(t-s)^\eps \lambda^{-1-\frac\gamma2+4\eps}
\end{align}
for $C=C(\|b\|_{\C^\gamma},\gamma,d,m)$ independent of $\lambda$, $n$, $k$.

\textbf{Step~2c}. Now we substitute \eqref{ivanpred2} and \eqref{i2nk} into \eqref{limeq}. We note that $\|b^n -b^k\|_{\C^{\gamma-\eps}}\to \|b^n -b\|_{\C^{\gamma-\eps}}$ as $k\to\infty$. Therefore, by Fatou's lemma we derive
for any $(s,t)\in\Delta_{[0,1]}$
\begin{align}\label{xmy1}
&\|(X_t-\wt X^{n,\lambda}_t)-(X_s-\wt X^{n,\lambda}_s)\|_{L_m (\Omega)}\nn\\
&\qquad\le  C_0(t-s)^\eps  \lambda^{-\frac32-\gamma+6\eps}+(\frac14+C_0\lambda^{-2\eps}) (t-s)^\eps \|X-\wt X^{n,\lambda}\|_{\C^0L_m([0,1])}\nn\\
&\qquad\phantom{\le}+C_0(t-s)^\eps \|b^n-b\|_{\C^{\gamma-\eps}}\lambda^{-1-\frac\gamma2+4\eps}	+\lambda^{\eps}|t-s|^{\eps}\|X_s-\wt X^{n,\lambda}_s\|_{L_m(\Omega)},
\end{align}
where $C_0=C_0(\|b\|_{\C^\gamma},\gamma,\eps,d,m)$.	

\textbf{Step~3: buckling for supremum norm}. 
 Choose now any $\lambda\ge1$ such that 
\begin{equation}\label{xmy2}
C_0\lambda^{-2\eps}\le\frac14,
\end{equation}
where $C_0$ is as in \eqref{xmy1}.
By choosing  now $s=0$ in \eqref{xmy1} and taking supremum over all $t\in[0,1]$ we deduce with the help of \eqref{xmy2}
\begin{equation}\label{diffeq}
\|X-\wt X^{n,\lambda}\|_{\C^{0}L_m ([0,1])}\le \frac12\|X-\wt X^{n,\lambda}\|_{\C^{0}L_m ([0,1])}+ \lambda^{-\frac32-\gamma+8\eps}+\|b^n-b\|_{\C^{\gamma-\eps}}\lambda^{-1-\frac\gamma2+6\eps}.
\end{equation}	
Note that 	$\|X\|_{\C^{0}L_m ([0,1])}<\infty$ because $X$ belongs to the class $\V(1+\frac\gamma2)$. Furthermore, it is also easy to see that 
$\|\wt X^{n,\lambda}\|_{\C^{0}L_m ([0,1])}<\infty$. Hence  $\|X-\wt X^{n,\lambda}\|_{\C^{0}L_m ([0,1])}<\infty$. Thus we can subtract the term $\frac12\|X-\wt X^{n,\lambda}\|_{\C^{0}L_m ([0,1])}$ from both sides of \eqref{diffeq} and derive 	
\begin{equation}\label{supfinal}
\|X-\wt X^{n,\lambda}\|_{\C^{0}L_m ([0,1])}\le  2 \lambda^{-\frac32-\gamma+8\eps}+ 2\|b^n-b\|_{\C^{\gamma-\eps}}\lambda^{-1-\frac\gamma2+6\eps}.
\end{equation}
	
\textbf{Step~4: bounding the H\"older norm}. 
Now we substitute \eqref{supfinal} into \eqref{xmy1}. We  note that $\|X_s-\wt X^{n,\lambda}_s\|_{L_m(\Omega)}\le 	\|X-\wt X^{n,\lambda}\|_{\C^{0}L_m ([0,1])}$. We derive
for any $(s,t)\in\Delta_{[0,1]}$
\begin{equation*}
\|(X_t-\wt X^{n,\lambda}_t)-(X_s-\wt X^{n,\lambda}_s)\|_{L_m (\Omega)}\le C (t-s)^\eps  \lambda^{-\frac32-\gamma+9\eps}+C(t-s)^\eps \|b^n-b\|_{\C^{\gamma-\eps}}\lambda^{-1-\frac\gamma2+7\eps}
\end{equation*}
for some $C>0$. 
Dividing both sides of the above inequality by $(t-s)^\eps$ and taking the supremum over all $(s,t)\in\Delta_{[0,1]}$, we finally obtain
\begin{equation*}
[X-\wt X^{n,\lambda}]_{\C^{\eps}L_m ([0,1])}\le C  \lambda^{-\frac32-\gamma+9\eps}+C \|b^n-b\|_{\C^{\gamma-\eps}}\lambda^{-1-\frac\gamma2+7\eps}
\end{equation*}
Since $m\eps>1$ and the processes $X$ and $\wt X^{n,\lambda}$ are continuous, we can apply the Kolmogorov continuity theorem (\cref{t:kolmi}) to derive	
	\begin{equation*}
	\|\sup_{t\in[0,1]}|X_t-\wt X^{n,\lambda}_t|\|_{L_m (\Omega)}\le C  \lambda^{-\frac32-\gamma+9\eps}+C \|b^n-b\|_{\C^{\gamma-\eps}}\lambda^{-1-\frac\gamma2+7\eps}.
\end{equation*}
for $C=C(\eps)$, which is  the desired bound \eqref{mainboundxmy}.
\end{proof}

Now, to show weak uniqueness, it just remains to collect the bounds from \cref{l:ks1} and \cref{L:ks2} and plug them into \eqref{decompos}. Let us now provide the precise  statement, formalizing \cref{t:weak}.

\begin{theorem}\label{t:420}
Let $d\in\N$, $x\in\R^d$, $b\in\C^\gamma(\R^d,\R^d)$, $\gamma\in(-\frac12,0)$. Then SDE \eqref{distrsde} has a unique weak solution in the class  $\V(1+\frac\gamma2)$.
\end{theorem}

\begin{proof}
Fix $b\in\C^\gamma$, $x\in\R^d$ and let $(b^n)_{n\in\Z_+}$ be a sequence of $\C^\infty(\R^d,\R^d)$ functions converging to $b$ in $\C^{\gamma-}$. 
Let $(X^n,W)$ be the strong solution to the regularized equation \eqref{distrsden}. Let $(X,W)$ be a weak solution to  \eqref{distrsde} in the class $\V(1+\frac\gamma2)$.

Fix $\eps>0$ such that condition  \eqref{epsineq} holds. Let $\lambda_0=\lambda_0(\|b\|_{\C^\gamma},\gamma,\eps,d,m)>1$ be as in \cref{L:ks2}. Then for any $\lambda>\lambda_0$ we have by  \eqref{decompos},  \cref{l:ks1} and \cref{L:ks2} 
\begin{align}\label{mainproof1}
\W_{\|\cdot\|\wedge1}(\law(X^n),\law(X))&\le  d_{TV}(\law(X^n),\law(\wt X^{n,\lambda}))+\E \|\wt X^{n,\lambda}-X\|_{\C([0,1];\R^d)}\nn\\
&\le C \lambda \| \|\wt X^{n,\lambda}- X\|_{\C([0,1];\R^d)} \|_{L_2(\Omega)}\nn\\
&\le C  \lambda^{-\frac12-\gamma+9\eps}+C \|b^n-b\|_{\C^{\gamma-\eps}}\lambda^{-\frac\gamma2+7\eps}
\end{align}
for $C=C(\eps)$. By passing to the limit as $n\to\infty$ in the above inequality, we get
\begin{equation*}
\limsup_{n\to\infty}\W_{\|\cdot\|\wedge1}(\law(X^n),\law(X))\le C  \lambda^{-\frac12-\gamma+9\eps}
\end{equation*}
for  $C=C(\eps)>0$, and any $\lambda>\lambda_0$. Passing now to the limit as $\lambda\to\infty$, we derive
\begin{equation*}
\lim_{n\to\infty}\W_{\|\cdot\|\wedge1}(\law(X^n),\law(X))=0,
\end{equation*}
where we used that by assumption $\gamma-9\eps>-\frac12$.

Now if  $(\overline{X},\overline{W})$ is  another weak solution to \eqref{distrsde} and $\overline{X}\in\V(1+\frac\gamma2)$, then by the above 
\begin{equation*}
	\lim_{n\to\infty}\W_{\|\cdot\|\wedge1}(\law(X^n),\law(\overline{X}))=0,
\end{equation*}
which implies that $\law(X)=\law(\overline{X})$ and thus weak uniqueness holds. 
\end{proof}

\subsection{Exercises}

First, let us better understand the notion of a Besov space and discuss its properties. As mentioned in \cref{s:41}, the standard definition of a Besov space is different from \cref{d:besov}. Before we state it we require some preliminary work.

For a function $f\in\S(\R^d)$ define its Fourier and inverse Fourier transforms by
\begin{equation*}
\F f(\lambda):=\frac1{(2\pi)^{\frac{d}2}}\int_{\R^d} e^{-i x^\t \lambda}f(x)\,dx,\qquad
\F^{-1} f(x):=\frac1{(2\pi)^{\frac{d}2}}\int_{\R^d} e^{i x^\t \lambda}f(\lambda)\,d\lambda,
\end{equation*}
where $x,\lambda\in\R^d$. It is known that both $\F$ and $\F^{-1}$ map $\S(\R^d)$ to itself and for any $f\in \S(\R^d)$ we have $\F^{-1}\F f= \F\F^{-1} f=f$. 

The Fourier and inverse Fourier transforms  of a tempered distribution 
$f\in\S'(\R^d)$ as well as its product with a function $\phi\in\S(\R^d)$ are defined by duality: if $g\in\S(\R^d)$ then
\begin{equation}\label{manybesov}
\langle \F f, g\rangle:=\langle f, \F g\rangle;\qquad \langle \F^{-1} f, g\rangle:=\langle f, \F^{-1} g\rangle;\qquad \langle f\phi, g\rangle:=  \langle f,\phi g\rangle.
\end{equation}
The convolution between a tempered distribution $f\in\S'(\R^d)$ and a Schwartz function $\phi\in\S(\R^d)$ is defined as 
\begin{equation*}
f*\phi(x):=\langle f,\phi(x-\cdot)\rangle,\quad x\in\R^d.
\end{equation*}

\begin{exercise}\label{e:besov1} Let $f\in\S'(\R^d)$, $\phi\in\S(\R^d)$.
\begin{enumerate}[(i)] 
\item Show that the definitions of $\F f$, $ \F^{-1} f$, $f\phi$, $f*\phi$ are consistent with \cref{c:conv}. That is, if $f=T_u$ for some function $u\in\S(\R^d)$ (recall \eqref{convention}), then
\begin{equation*}
 \F T_u=T_{\F u},\quad \F^{-1} T_u=T_{\F^{-1} u},\quad T_u\phi=T_{u\phi},\quad T_u*\phi=u*\phi.
\end{equation*}
\item Show that the convolution $f*\phi$ has at most polynomial growth and thus it can be identified with a distribution (recall \eqref{convention}).

\hn Use the inequality $1+|y|\le  (1+|x|)(1+|x-y|)$ valid for any $x,y\in\R^d$.
\item Show that $(2\pi)^{-d/2}\F(f*\phi)= \F(f)\F(\phi)$.
\item Deduce from part (iii) that 
\begin{equation}\label{invfour}
\F^{-1}(f\phi)= (2\pi)^{-d/2}\F^{-1}(f)*\F^{-1}(\phi)
\end{equation}
and thus the distribution $\F^{-1}(f\phi)$ can be identified with a function.

\hn Apply part (iii) of the exercise with $\F^{-1}f$ in place of $f$ and $\F^{-1}\phi$ in place of $\phi$.
\end{enumerate}
\end{exercise}

Following \cite[Section~1.3]{Tr3} let us fix two functions $\phi_0,\phi^0\in \S(\R^d)$ satisfying the following property: there exists $\eps>0$ such that
\begin{equation}\label{pairrec}
|\phi_0(x)|>0\quad  \text{if $|x|\le 2\eps$};\qquad |\phi^0(x)|>0\quad  \text{if $\frac\eps2\le|x|\le 2\eps$}.
\end{equation}
For $j\in\N$ we put $\phi_j(x):=\phi^0(2^{-j} x)$. 

\begin{definition}\label{d:besov2}
Let $\alpha<0$. We say that a tempered distribution $f\in\S'(\R^d)$ belongs to the Besov space $\C^\alpha(\R^d)$ if  
\begin{equation}\label{negbesovphi}
\|f\|_{\C^\alpha(\R^d),\phi_0,\phi^0}:=\sup_{j\in\Z_+} 2^{j\alpha}\|\F^{-1}(\phi_j \F f)\|_{L_\infty(\R^d)}<\infty.
\end{equation}
\end{definition}	

Note that since $f\in\S'(\R^d)$, we have $\F f \in\S'(\R^d)$. Since $\phi_j\in\S(\R^d)$, we see that, thanks to \cref{e:besov1}(iv), the tempered distribution $\F^{-1}(\phi_j \F f)$ can be identified with a function, and thus its $L_\infty(\R^d)$ norm is well-defined (though it could be of course infinite).

Note that the definition of the Besov space in \cref{d:besov2} seems to depend on the pair $(\phi_0,\phi^0)$. A key result in the theory of Besov spaces establishes that all such norms are equivalent.

\begin{theorem}[{\cite[Theorem 1.7(i)]{Tr3}}]\label{t:triebel} Let $\alpha<0$. Assume that a pair of Schwartz functions $(\wt\phi_0,\wt\phi^0)$ satisfies \eqref{pairrec} for some $\wt\eps>0$ in place of $\eps$. Then the norms 
${\|\cdot\|_{\C^\alpha(\R^d),\phi_0,\phi^0}}$ and $\|\cdot\|_{\C^\alpha(\R^d),\wt \phi_0,\wt\phi^0}$ are equivalent.
\end{theorem}

Therefore, from now on we will denote this norm by $\|\cdot\|_{\C^\alpha(\R^d)}$, dropping the pair $(\phi_0,\phi^0)$.

\begin{exercise}\label{e:besovdef} Let $\alpha<0$.
\begin{enumerate}[(i)] 
\item Let $h\in\S(\R^d)$ satisfy $\int_{\R^d} h(x)dx\neq0$. Define for $t>0$, $h_t(x):=t^{-\frac{d}2}h(t^{-\frac12}x)$. Show that the norm 
\begin{equation*}
\|f\|_{\C^\alpha(\R^d),h}:=\sup_{j\in\Z_+}2^{j\alpha}\|h_{2^{-2j}}* f\|_{L_\infty(\R^d)}
\end{equation*}
is equivalent to the norm $\|f\|_{\C^\alpha(\R^d)}$ defined in \cref{d:besov2}.

\hn Use \cref{t:triebel} with $\phi^0=\phi_0=\F h$ and identity \eqref{invfour}.
\item Deduce from part (i) that the norm used in this section, namely
\begin{equation*}
\|f\|_{\C^\alpha(\R^d)}:=\sup_{t\in(0,1]}t^{-\frac\alpha2}\|P_t f\|_{L_\infty(\R^d)}
\end{equation*}
is equivalent to the norm defined in \cref{d:besov2}.
\item In the setting of part (i) of the exercise, show that the norm 
$$
\sup_{t\in(0,1]}t^{-\frac\alpha2}\|h_t *f\|_{L_\infty(\R^d)}
$$
is also equivalent to the Besov norm.

\hn One direction is obvious. To show the reverse direction, use the identity $(h_t*f)(x)=(h*f^{(t)})(t^{-1/2}x)$, where $\langle f^{(t)},g\rangle:=t^{-\frac{d}2}\langle f,g(t^{-1/2}\cdot)\rangle$, $g\in\S(\R^d)$. Then conclude applying part (i) of the exercise for $j=0$ and then part (ii) of the exercise.
\end{enumerate}
\end{exercise}

Now let us provide some examples of tempered distributions belonging to a certain Besov space. For a measurable function $f\colon\R^d\to\R$ with at most polynomial growth we define its distributional derivative $\nabla f$ by $\langle \nabla f, g\rangle:=-\langle f,  \nabla g\rangle$, $g\in\S(\R^d)$.
\begin{exercise}\label{e:easyprop} 
\begin{enumerate}[(i)] 
\item Let $f\in\C^\alpha(\R^d)$, where $\alpha\in(0,1)$. Show that $\nabla f\in \C^{\alpha-1}(\R^d)$.

\item (Toy version of Besov embedding) Let $f\in L_p(\R^d)$, $p\in[1,\infty)$. Show that it can be identified with a tempered distribution (recall \cref{c:conv}) and $f\in \C^{-\frac{d}p}(\R^d)$.

\item Let $\mu$ be a finite measure on $(\R^d,\mathscr{B}(\R^d))$. Show that $\langle R_\mu,g\rangle:=\int_{\R^d} gd\mu$, $g\in\S(\R^d)$, defines a tempered distribution, and $R_\mu\in\C^{-d}(\R^d)$. 

\end{enumerate}
\end{exercise}

In the next exercise we prove that, as claimed in \cref{s:41}, if $\beta<0$ and $f\in\C^\beta$, then the sequence $P_{s}f$ converges to $f$ in $\C^{\beta-}$ as $s\to0$ and this convergence cannot be improved, in general, to convergence in $\C^\beta$.
\begin{exercise}\label{e:betamin} Let $\beta<0$ and $f\in\C^\beta$.
\begin{enumerate}[(i)] 
\item Show that for any $s>0$ we have $\|P_s f\|_{\C^\beta}\le \|f\|_{\C^\beta}$.

\hn  Use \cref{d:besov} and a trivial bound $\|P_s g\|_{L_\infty}\le  \|g\|_{L_\infty}$ for a bounded measurable $g$.
\item\label{p2betam} Let $\gamma<0$. Show that for some constant $C=C(\gamma,d)$ we have for any $t\in(0,1]$ 
\begin{equation*}
\|P_t f\|_{\C^2(\R^d)}\le C t^{\frac{\gamma-2}2}\|f\|_{\C^\gamma}.
\end{equation*}

\hn Write $P_t f= P_{\frac{t}2}P_{\frac{t}2}f$. Then apply  \eqref{ineq} and  use \cref{d:besov}.

\item\label{p3betam} Show that for any $x\in\R^d$, $s,t>0$ we have
\begin{equation*}
|P_{t+s} f(x)-P_tf(x)|\le \int_{t}^{t+s} |\Delta P_r f(x)|\,dr.
\end{equation*}
\item 
Let $\eps\in(0,2]$, $t\in(0,1]$, $s>0$. Deduce from parts \ref{p2betam} and \ref{p3betam} that  for $C=C(\beta,\eps,d)$
\begin{equation*}
\|P_{t+s} f-P_tf\|_{L_\infty(\R^d)}\le C\|f\|_{\C^{\beta-\eps}} t^{\frac\beta2-\eps}s^{\frac\eps2}.
\end{equation*}
\item Conclude that $P_{s}f$ converges to $f$ in $\C^{\beta-}$ as $s\to0$.

\item Let $d=1$. Recall that thanks to \cref{e:easyprop}(iii) we have that the Dirac delta measure at $0$, $\delta_0\in\C^{-1}$. Show that for any  smooth function $f\in\C_b^\infty(\R,\R)$ we have 
\begin{equation*}
\|f-\delta_0\|_{\C^{-1}}\ge\frac{1}{\sqrt{2\pi}}.
\end{equation*}
\hn Use that by definition, for any value of the parameter $t\in(0,1]$ we have $\|f-\delta_0\|_{\C^{-1}}\ge t^{1/2} P_{t}\delta_0(0)-t^{1/2} P_t f(0)$.
\end{enumerate}	
\end{exercise}

The next group of exercises develops the properties of the KL divergence and total variation distance. We begin with the proof of Pinsker's inequality, \cref{pinsker}(i).

\begin{exercise}\label{e:pinsker}
\begin{enumerate}[(i)]
\item First, suppose that the space $E$ has only two points $E=\{A,B\}$. Let $\mu(A)=p$, $\mu (B)=1-p$,  $\nu(A)=q$, $\nu (B)=1-q$, where $p,q\in[0,1]$. Show that  \cref{pinsker}(i) holds in this case.
\item Now let us do the general case. Assume without loss of generality that $\mu\ll\nu$ (otherwise, the right-hand side of Pinsker's inequality is infinite). Fix any measurable set $S\subset E$ with $\mu(S)\neq0$, $\nu(S)\neq0$. Define probability measures on $(S,\mathcal{E}_S:=\{A\cap S:A\in\mathcal E\})$:
\begin{equation*}
\mu_S(A):=\frac{\mu(A)}{\mu(S)},\quad
\nu_S(A):=\frac{\nu(A)}{\nu(S)},\quad A\in\mathcal{E}_S.
\end{equation*}
Show that $\mu_S\ll\nu_S$ and 
\begin{equation*}
\frac{d \mu_S}{d \nu_S}=\frac{\nu(S)}{\mu(S)}\frac{d \mu}{d \nu}.
\end{equation*}
\item Deduce that 
\begin{equation*}
\int_S \log \bigl(\frac{d \mu}{d \nu}\bigr)d\mu\ge 
\log \bigl(\frac{\mu(S)}{\nu(S)}\bigr)\mu(S).
\end{equation*}
\hn Use \cref{p:basicdkl}(ii).

\item Show that for any  measurable set $S\subset E$ we have 
\begin{equation*}
\int_E \log \bigl(\frac{d \mu}{d \nu}\bigr)d\mu\ge 
	\log \bigl(\frac{\mu(S)}{\nu(S)}\bigr)\mu(S)+\log \bigl(\frac{\mu(E\setminus S)}{\nu(E\setminus S)}\bigr)\mu(E\setminus S).
\end{equation*}

\item Use parts (i) and (iv) of the exercise to conclude that \cref{pinsker}(i) holds.
\end{enumerate} 
\end{exercise}

The ideas for the next three exercises were suggested by the referee. \cref{e:data} is also partially  inspired by \cite[Proof of Lemma~A.2]{Cyril}.

\begin{exercise}\label{e:data}
Let $(E,\EE)$ be a measurable space, $\mu,\nu\in\PP(E)$. For parts (i)–(iv), assume that ${D_{KL}(\mu\|\nu)<\infty}$ and hence $\mu\ll\nu$.
\begin{enumerate}[(i)] 
\item Let $A:=\{\frac{d\mu}{d\nu}>0\}$. Show that $\mu(E\setminus A)=0$.
\item  Show that for any bounded measurable function $g\colon E\to\R$ one has
\begin{equation*}
D_{KL}(\mu\|\nu)\ge\int_E g d\mu-\log \int_E e^g d\nu.
\end{equation*}

\hn Apply Jensen's inequality to $\int_A \log \bigl(e^g(\frac{d\mu}{d\nu})^{-1}\bigr) d\mu$.
\item Consider a sequence of bounded measurable functions $g_n:=\log \bigl( (\frac{d\mu}{d\nu}\wedge e^n)\vee e^{-n}\bigr)$. Show that
\begin{equation*}
\int_E g_n d\mu\to D_{KL}(\mu\|\nu);\quad \int_E e^{g_n} d\nu\to 1,\qquad \text{as $n\to\infty$.}
\end{equation*} 

\hn Use the dominated convergence theorem. In the first limit use that $|g_n|\le |\log \frac{d\mu}{d\nu}|$ and recall \cref{p:basicdkl}(i).
\item Derive from parts (ii) and (iii) of the exercise the variational characterization of relative entropy:
\begin{equation}\label{varchar}
D_{KL}(\mu\|\nu)=\sup_{g}\Bigl(\int_E g d\mu-\log \int_E e^g d\nu\Bigr),
\end{equation}
where the supremum is taken over all bounded measurable functions $g\colon E\to\R$.
\item Show that identity \eqref{varchar} holds in the case ${D_{KL}(\mu\|\nu)=\infty}$ as well.

\hn If $\mu\nll \nu$ take $g_n:= n \I_B$, where $B\in\EE$ is the set such that $\mu(B)>0$, $\nu(B)=0$. If $\mu\ll \nu$ take the same $g_n$ as in part (iii). 

\item Deduce from parts (iv) and (v) the data processing inequality (\cref{pinsker}(ii)). Namely, show that if
 $(E_1,\mathcal{E}_1)$ is another measurable space and  $\Phi\colon E\to E_1$ is a measurable function, then 
\begin{equation*}
D_{KL}(\Phi^{\#}\mu\|\Phi^{\#}\nu)\le 
D_{KL}(\mu\|\nu).
\end{equation*}	
\hn Use the change of variable formula:
$\int_{E_1} g d(\Phi^\#\mu)=\int_E g\circ \Phi d \mu$.

\item Show that total variation distance also satisfies the data processing inequality. That is, in the setting of part (vi), we have
\begin{equation*}
d_{TV}(\Phi^{\#}\mu,\Phi^{\#}\nu)\le 
d_{TV}(\mu,\nu).
\end{equation*}	

\hn Use just the definition \eqref{dtv}.
\end{enumerate}
\end{exercise}

\begin{exercise}\label{e:lsc}
Let $(E,\rho)$ be a metric space, $\EE=\mathscr{B}(E)$, $\mu,\nu\in\PP(E)$.
\begin{enumerate}[(i)] 
\item Let $g$ be a bounded measurable function $E\to\R$. Show that there exists a sequence of bounded continuous functions $g_n\colon E\to\R$ such that
\begin{equation*}
\sup_{x\in E} |g_n(x)|\le  \sup_{x\in E} |g(x)|\quad\text{and}\quad \int_{E}|g-g_n|d (\mu+\nu)\to0
\text{ as $n\to\infty$}.
\end{equation*}

\hn Show this statement first for indicator functions, then for the finite sums of indicator functions, then in the general case. Use that finite Borel measures on metric spaces are regular, that is, for any $A\in\EE$, $\eps>0$, there exists a closed set $F\subset A$ such that $(\mu+\nu)(A\setminus F)<\eps$.

\item Show that the variational characterization formula \eqref{varchar} holds if the supremum is taken over all bounded \textbf{continuous} functions $g\colon E\to\R$.

\hn Use part (i) of the exercise.
\item Show that KL divergence is lower semicontinuous in both arguments. That is, if $(\mu_n)_{n\in\N}$, $(\nu_n)_{n\in\N}$ are two sequences of measures in $\PP(E)$ weakly converging to $\mu$ and $\nu$, respectively, as $n\to\infty$, then
\begin{equation*}
D_{KL}(\mu\|\nu)\le\liminf_{n\to\infty} D_{KL}(\mu_n\|\nu_n).
\end{equation*}

\hn Use part (ii) of the exercise.
\end{enumerate}
\end{exercise}

Now let us see how the proof of \cref{t:420} can be extended to get stability bounds for weak solutions of SDE \eqref{distrsde}.

\begin{exercise} Let $M>0$, $\gamma\in(-\frac12,0)$, $b\in\C^\gamma$. Let $X$ be the unique weak solution to SDE \eqref{distrsde} in the class $\V(1+\frac\gamma2)$.  Let $(b^n)_{n\in\N}$ be a sequence of smooth functions converging to $b$ in $\C^{\gamma-}$, and denote by $X^n$ the unique solution to SDE \eqref{distrsden}. Suppose that
\begin{equation*}
\|b\|_{\C^\gamma}\le M;\qquad\sup_{n\in\N}\|b^n\|_{\C^\gamma}\le M.
\end{equation*}

\begin{enumerate}[(i)] 
\item Modify the last part of the proof of
\cref{t:420} to show that for any $\eps>0$ there exist $\delta=\delta(\gamma,\eps)>0$, $C=C(\gamma,\eps,d,M)>0$ such that for any $n\in\N$
\begin{equation*}
\W_{\|\cdot\|\wedge1}(\law(X^n),\law(X))
\le C\|b^n- b\|_{\C^{\gamma-\delta}}^{\frac{1+2\gamma}{1+\gamma}-\eps}.
\end{equation*}
\item Consider now SDE \eqref{distrsde} with drift $\overline{b}$ in place of $b$ and let $\overline{X}$ be its unique weak solution in the class $\V(1+\frac\gamma2)$. Suppose that $\|\overline{b}\|_{\C^\gamma}\le M$. 

Extend the results of part~(i) and show that for any $\eps>0$ there exist $\delta=\delta(\gamma,\eps)>0$, $C=C(\gamma,\eps,d,M)>0$  such that
\begin{equation}\label{genbounddr}
\W_{\|\cdot\|\wedge1}(\law(\overline X),\law(X))
\le C\|b-\overline b\|_{\C^{\gamma-\delta}}^{\frac{1+2\gamma}{1+\gamma}-\eps}.
\end{equation}

\hn Argue as in the proof of  \cref{t:420}, but replace there $X$ with $\overline{X}$.

\item  Use bound \eqref{genbounddr} to show that the sequence $(\law(X^n))_{n\in\N}$ is Cauchy in the Polish space $\PP(\C([0,1];\R^d))$ equipped with the distance $\W_{\|\cdot\|\wedge1}$. Deduce that it converges to a limit and the limit is the law of a weak solution to SDE \eqref{distrsde}.

This provides an alternative way of proving weak existence of solutions to this SDE (recall \cref{t:WESDE}).
\end{enumerate}
\end{exercise}

Finally, let us just outline how SDEs with irregular diffusion coefficients can be treated. We refer to \cite{MP22,matsuda2023pathwise} for further related results. 

\begin{exercise}[\cite{Ksch}] Consider an SDE
\begin{equation}\label{kuliks}
dX_t = \sigma(X_t)\, dW_t,\quad X_0=x,
\end{equation}
where $x\in\R^d$, $\sigma\in\C^\gamma(\R^d,\R^{d\times d})$, $\gamma>0$. Assume that for each $x\in\R^d$ there exists a right inverse $\sigma(x)^{-1}$ of the matrix $\sigma(x)$, and
\begin{equation}\label{inverse}
\sup_{x\in\R^d}
\bigl|\!\bigl|\!\bigl|\sigma(x)^{-1}\bigr|\!\bigr|\!\bigr|
<\infty,
\end{equation}
where $\bigl|\!\bigl|\!\bigl|\cdot\bigr|\!\bigr|\!\bigr|$ denotes the Frobenius norm. Use the methodology of this section to show that SDE \eqref{kuliks} has a unique weak solution if $\gamma>1/2$.

\hn Correspondingly modify the generalized coupling process $\wt X^{n,\lambda}$. Use It\^o's formula in place of stochastic sewing.
\end{exercise}	

Note that the celebrated Stroock-Varadhan theorem (see, e.g., \cite[Theorem 7.2.1]{SVbook}) establishes weak uniqueness of SDE \eqref{kuliks} in a much larger class of diffusions, namely if $\sigma$ is a bounded continuous function (and condition \eqref{inverse} holds). It is not known if one could extend the generalized coupling argument to cover the whole Stroock-Varadhan regime of regularity of $\sigma$  (or at least up to an arbitrarily small loss of regularity).

\section{Analysis of numerical methods and further applications of stochastic sewing}\label{s:5}
In this section, we see how stochastic sewing can be very efficiently applied in stochastic numerics.
As always, we present the main idea in the basic setting of Brownian noise (\cref{s:numbd}) and leave the extension to the fractional Brownian case to the reader. However, when the driving noise is an $\alpha$-stable process, certain fundamental difficulties arise that require new tools: the John–Nirenberg inequality (\cref{s:JN}), new norms (\cref{sec:44}), and the shifted stochastic sewing lemma (\cref{s:sssl}). When the drift is a function in a Sobolev space rather than a H\"older continuous function, we need an additional tool: the taming singularities lemma (\cref{s:levyex}). We will cover all of these in this section.

\subsection{Numerics for SDEs driven by Brownian noise}\label{s:numbd}
We begin with SDEs driven by standard Brownian motion, with the goal of establishing \cref{t:mainnumt}. In this subsection, we mainly follow \cite{BDG}.

Recall the setup of \cref{s:12} and the definition of $\kappa_n(r)$ in \eqref{kappanr}. Fix $\gamma\in(0,1]$, $b\in\C^\gamma$. Let $X$ be a solution to \eqref{mainSDE}, and for $n\in\N$ let $X^n$ be the Euler scheme for this equation as defined in \eqref{mainSDEn}. Similar to \cref{s:SU}, we denote 
\begin{equation*}
\phi:=X-W,\quad \phi^n=X^n-W.
\end{equation*}
We would like to bound the approximation error 
\begin{equation*}
\|X-X^n\|_{\C^0L_2([0,1])}=\|\phi-\phi^n\|_{\C^0L_2([0,1])}
\end{equation*}
by half of itself plus terms of the form $C n^{-\rrr}$. However, this is not possible to achieve directly, because when we apply \cref{e:step1uni}, a term with a higher seminorm, namely $[\phi-\phi^n]_{\C^{1/2}L_2([0,1])}$, appears. Therefore, instead we aim to bound 
\begin{equation*}
[\phi-\phi^n]_{\C^{1/2}L_2([S,T])}
\end{equation*}
by half of itself plus $C n^{-\rrr}$. This is possible whenever the interval $[S,T]$ is small enough, and then we apply \cref{l:ltg} that transforms this local bound into a bound on the entire interval.

Thus, for $(s,t)\in\Delta_{[0,1]}$ we need to bound 
\begin{align}
(\phi_t-\phi^n_t)-(\phi_s-\phi^n_s)&=\int_s^t (b(X_r)-b(X^n_r))\,dr+\int_s^t (b(X^n_{r})-b(X^n_{\kappa_n(r)}))\,dr
\nn\\
&=\int_s^t (b(W_r+\varphi_r)-b(W_r+\varphi^n_{r}))\,dr
	\nn\\
	&\phantom{=}+\int_s^t (b(W_r+\phi^n_{r})-b(W_{\kappa_n(r)}+\phi^n_{\kappa_n(r)}))\,dr.\label{eq:error-decomposition}
\end{align}
We already know how to deal with the first integral, see \cref{e:step1uni}. Therefore, we focus on the second one. As usual, to warm up, we begin with a simplified result, where $\phi^n\equiv0$.
\begin{theorem}\label{t:appradf}
Let $\gamma\in(0,1]$, $f\in \C^\gamma(\R^d,\R)$, $m\ge2$, $\eps\in(0,\frac12)$. Then there exists a constant $C=C(\gamma,\eps,d,m)>0$ such that for any $n\in\N$, $(S,T)\in\Delta_{[0,1]}$ one has 
\begin{equation}\label{nummain}
\Bigl\|\int_S^T (f(W_r)-f(W_{\kappa_n(r)}))\,dr\Bigr\|_{L_m(\Omega)}\le  C  \| f\|_{\C^\gamma}(T-S)^{\frac12+\eps}n^{-\frac12-\frac\gamma2+\eps}.
\end{equation}
\end{theorem}
\begin{remark}
\Cref{t:appradf} is also of independent interest. It provides the rate of convergence for an approximation of an additive functional of Brownian motion based on high-frequency observations. It extends the corresponding results of \cite{Alt}.
\end{remark}

To prove \cref{t:appradf}, we need the following technical lemma on heat kernel bounds.
\begin{lemma}\label{l:gb2}
	Let $f\in\C^\gamma(\R^d,\R)$, $\gamma\in[0,1]$, $0< s\le t\le 1$. Then
	\begin{equation*}
		\|P_tf-P_s  f\|_{\C^0}\le  C\| f\|_{\C^\gamma} s^{\frac\gamma2-1}(t-s)
	\end{equation*} 	
	for $C=C(\gamma,d)>0$.
\end{lemma}
\begin{proof}
Fix $0<s\le t$. For any $x\in\R^d$, using \cref{l:gb} we have
\begin{align*}
|P_tf(x)-P_sf(x)|&= \Bigl|\int_s^t \frac{\d}{\d r} P_r f(x)\,dr\Bigr|=\frac12\Bigl|\int_s^t \Delta P_r f(x)\,dr\Bigr|\\
&\le C\int_s^t \|\nabla P_r f\|_{\C^1}\,dr\le C\| f\|_{\C^\gamma}  \int_s^t  r^{\frac\gamma2-1}\,dr\le C\| f\|_{\C^\gamma} s^{\frac\gamma2-1}(t-s).\qedhere
\end{align*}
\end{proof}

\begin{proof}[Proof of \cref{t:appradf}]
As before, we rely on  the stochastic sewing lemma. Fix $\eps\in(0,\frac12)$. Consider the germ 
\begin{equation}\label{astnumdef}
A_{s,t}:=\E^s \int_s^t (f(W_r)-f(W_{\kappa_n(r)}))\, dr,\quad  (s,t)\in\Delta_{[S,T]}.
\end{equation}
Then,  for any $(s,u,t)\in\Delta^3_{[S,T]}$
\begin{equation*}
	\delta A_{s,u,t}=\E^s \int_u^t (f(W_r)-f(W_{\kappa_n(r)}))\, dr-\E^u \int_u^t(f(W_r)-f(W_{\kappa_n(r)}))\, dr
\end{equation*}
and as in the proof of \cref{t:firstk}, we have 
\begin{equation*}
\E^s \delta A_{s,u,t}=0.
\end{equation*}
Thus, condition \eqref{con:s2} trivially holds. 

Now let us verify  \eqref{con:s1}. Note the following potential danger in  the germ $A_{s,t}$ in~\eqref{astnumdef}: if $r$ is very close to $s$, it may happen that $\kappa_n(r) < s$. In this case, $\E^s f(W_{\kappa_n(r)})$ is just  $f(W_{\kappa_n(r)})$, and we do not gain any additional smoothing from the heat kernel (in contrast, the first term in the integral $\E^s f(W_r) = P_{r-s} f(W_s)$ always contains smoothing).

Therefore, we distinguish between the cases where $t-s$ is small or large relative to $\frac1n$. In the latter case, we split the integral in \eqref{astnumdef} into two parts: when $r$ is far from $s$ (i.e., $r \ge\kappa_n(s)+\frac2n$), we obtain sufficient smoothing and proceed as before; when $r$ is close to $s$, the entire integral will be very small.

Thus, take any $(s,t)\in\Delta_{[S,T]}$. First, suppose that $t-s\le \frac2n$. Then, we immediately get
\begin{align}\label{easynum}
\|A_{s,t}\|_{L_m(\Omega)}&\le  \int_s^t \| f(W_r)-f(W_{\kappa_n(r)}) \|_{L_m(\Omega)} \, dr\le  \| f\|_{\C^\gamma}\int_s^{t} \bigl(\E |W_r-W_{\kappa_n(r)}|^{m\gamma}\bigr)^{\frac1m} \, dr\nn\\
&\le C \| f\|_{\C^\gamma}(t-s) n^{-\frac\gamma2} \le C \| f\|_{\mathcal{C}^\gamma} n^{-\frac12-\frac\gamma2+\eps}(t-s)^{\frac12+\eps},
\end{align}
for $C=C(\gamma,d,m)>0$.
Here, in the penultimate inequality, we used that $r-\kappa_n(r)\le\frac1n$. In the final inequality, we used that  $(t-s)\le \frac2n$, which allowed us to transfer some powers from $(t-s)$ to $\frac1n$. Note that we kept $(t-s)^{\frac12+\eps}$, which is the bare minimum required for the stochastic sewing lemma to be applicable.

If $t-s\ge \frac2n$, let $s': =\kappa_n(s) + \frac2n$ be the second gridpoint to the right of $s$. We write
\begin{align}\label{step1num}
\|A_{s,t}\|_{L_m(\Omega)}&\le \Bigl\|\int_s^{s'}
	\E^s \big( f(W_r)-f(W_{\kappa_n(r)})\big) dr\Bigr\|_{L_m(\Omega)}\!\!\!+
	\Bigl\|\int_{s'}^t
	\E^s \big( f(W_r)-f(W_{\kappa_n(r)})\big) dr\Bigr\|_{L_m(\Omega)}\nn\\
	& =:I_1+I_2
\end{align}
We treat $I_1$ exactly as in \eqref{easynum} and get 
\begin{equation}\label{ivantr}
I_1 \le C \| f\|_{\C^\gamma}\int_s^{s'} \bigl(\E |W_r-W_{\kappa_n(r)}|^{m\gamma}\bigr)^{\frac1m} \, dr
	\le C \| f\|_{\C^\gamma}n^{-1-\frac\gamma2}
\end{equation}
for $C=C(\gamma,d,m)>0$. To bound $I_2$, we use \cref{l:gb2} and derive 
\begin{align}
I_2&=\Bigl\|\int_{s'}^t(P_{r-s}f(W_s)-P_{\kappa_n(r)-s}f(W_s))\,dr
\Bigr\|_{L_m(\Omega)}\le \int_{s'}^t\|P_{r-s}f-P_{\kappa_n(r)-s}f\|_{\C^0}\,dr\label{seethis}\\
&\le C\| f\|_{\C^\gamma}\int_{s'}^t(\kappa_n(r)-s)^{\frac\gamma2-1}n^{-1}\,dr\nn\\
&\le C\| f\|_{\C^\gamma}\int_{s'}^t(r-s')^{\frac\gamma2-1}n^{-1}\,dr\le C\| f\|_{\C^\gamma}(t-s)^{\frac\gamma2}n^{-1},\nn
\end{align}
for $C=C(\gamma,d)>0$. Here in the penultimate inequality we used that $r-\kappa_n(r)\le\frac1n\le s'-s$ and hence $\kappa_n(r)-s\ge r-s'$. Substituting this and \eqref{ivantr} into \eqref{step1num} and using that $t-s\ge\frac1n$, we get 
\begin{equation}\label{hardnum}
\|A_{s,t}\|_{L_m(\Omega)}\le C \| f\|_{\C^\gamma} n^{-1-\frac\gamma2} +C\| f\|_{\C^\gamma}(t-s)^{\frac\gamma2}n^{-1}\le C  \| f\|_{\C^\gamma}(t-s)^{\frac12+\eps}n^{-\frac12-\frac\gamma2+\eps}.
\end{equation}
Note how carefully we traded powers of $n^{-1}$ for powers of $t - s$ to ensure that condition \eqref{con:s1} is satisfied. 

Condition (ii) of \cref{T:SSLst} is satisfied by the definition of $A_{s,t}$ in \eqref{astnumdef}, and condition \eqref{con:s3} follows from the BDG inequality by exactly the same argument as in the proof of \cref{t:firstk}. Thus, all conditions of the SSL are satisfied, and the bound \eqref{est:ssl1} together with \eqref{easynum} and \eqref{hardnum} yields precisely \eqref{nummain}.
\end{proof}

\begin{remark}
The transfer of exponents from the factor $n^{-1}$ to $(t-s)$ in \eqref{easynum} and \eqref{hardnum} to ensure the applicability of the SSL may seem totally wasteful. Indeed, were the SSL applicable with a time exponent smaller than $1/2$, we would obtain a better rate (we have made no claims about the optimality of the SSL). It is therefore very interesting that  the resulting rate $n^{-\frac12-\frac\gamma2+\eps}$ is optimal up to an arbitrarily small loss; see \cite{ellinger2025optimal}. Thus, we did not lose almost anything with these transfers of powers.
\end{remark}	

Next, an easy application of the Girsanov theorem allows to insert the drift into inequality \eqref{nummain}. Recall that the process $X^n$ solves
\begin{equation*}
	X^n_t=x+\int_0^t b(X^n_{\kappa_n(s)})\,ds+W_t,\quad  t\in[0,1].
\end{equation*}

\begin{corollary}\label{c:appradf}
Let $\gamma\in(0,1]$, $f\in \C^\gamma(\R^d,\R)$, $m\ge2$, $\eps\in(0,\frac12)$.  Let $b\in\C^\gamma(\R^d,\R^d)$, $x\in\R^d$. 	
Then there exists a constant $C=C(\|b\|_{\C^0},\gamma,\eps,d,m)$ such that for any $n\in\N$, $(S,T)\in\Delta_{[0,1]}$ one has 
	\begin{equation}\label{nummain2}
		\Bigl\|\int_S^T (f(X^n_r)-f(X^n_{\kappa_n(r)}))\,dr\Bigr\|_{L_m(\Omega)}\le  C  \| f\|_{\C^\gamma}(T-S)^{\frac12+\eps}n^{-\frac12-\frac\gamma2+\eps}.
	\end{equation}
\end{corollary}
\begin{proof}
In the proof we write for brevity $h^n_t:=b(X^n_{\kappa_n(t)})$. Consider a probability measure $\P^n$ on $(\Omega,\F)$ defined by 
\begin{equation*}
	\frac{d\P^n}{d\P}=\exp\bigl(-\int_0^1 h^n_rd W_r-\frac12\int_0^1 |h^n_r|^2d r\bigr).
\end{equation*}	
Since the function $b$ is bounded, we see that $\P^n$ is a probability measure and Girsanov's theorem implies that the process
\begin{equation*}
W_t^n:=X^n_t-x=\int_0^t b(X^n_{\kappa_n(r)})dr+W_t,\quad t\in[0,1]
\end{equation*}
is a Brownian motion under the measure $\P^n$. Hence for any $(S,T)\in\Delta_{[0,1]}$, we get 
\begin{align}
&\E^\P \Bigl|\int_S^T (f(X^n_r)-f(X^n_{\kappa_n(r)}))\,dr\Bigr|^m\nn\\
&\quad=
\E^{\P^n} \Bigl|\int_S^T (f(X^n_r)-f(X^n_{\kappa_n(r)}))\,dr\Bigr|^m\frac{d\P}{d\P^n}\nn\\
&\quad\le \Bigl(\E^{\P^n} \Bigl|\int_S^T (f(X^n_r)-f(X^n_{\kappa_n(r)}))\,dr\Bigr|^{2m}\Bigr)^{\frac12} \Bigl(\E^{\P^n}\Bigl(\frac{d\P}{d\P^n}\Bigr)^2\Bigr)^{\frac12}\nn\\
&\quad= \Bigl(\E^{\P^n} \Bigl|\int_S^T (f(W^n_r+x)-f(W^n_{\kappa_n(r)}+x))\,dr\Bigr|^{2m}\Bigr)^{\frac12} \Bigl(\E^{\P}\frac{d\P}{d\P^n}\Bigr)^{\frac12}\nn\\
&\quad= \Bigl(\E^{\P} \Bigl|\int_S^T (f(W_r+x)-f(W_{\kappa_n(r)}+x))\,dr\Bigr|^{2m}\Bigr)^{\frac12} \Bigl(\E^{\P}\frac{d\P}{d\P^n}\Bigr)^{\frac12}\label{last}
\end{align}
Using the Cauchy-Schwarz inequality, we get 
\begin{align*}
\E^{\P}\frac{d\P}{d\P^n}&=\E^\P \exp\bigl(\int_0^1 h^n_rd W_r-\int_0^1 |h^n_r|^2\,dr+\frac32\int_0^1 |h^n_r|^2d r\bigr)\\
&\le\Bigl(\E^\P \exp\bigl(\int_0^1 2h^n_rd W_r-\frac12\int_0^1 |2h^n_r|^2\,dr\bigr)\Bigr)^{\frac12}\Bigl(\E^\P \exp\bigl(3\int_0^1 |h^n_r|^2\,d r\bigr)\Bigr)^{\frac12}\\
&\le \exp(\frac32\|h^n\|^2_{\C^0}),
\end{align*}
where in the last inequality we used that the process $t\mapsto e^{\int_0^t 2h^n_rd W_r-\frac12\int_0^t |2h^n_r|^2\,dr}$ is a martingale, because $h^n$ is bounded. Substituting this into \eqref{last} and applying \cref{t:appradf} to treat the first factor in the right-hand side of \eqref{last}, we get \eqref{nummain2}.
\end{proof}

Now we have all the ingredients to prove the rate of convergence of the Euler scheme~\eqref{Eulerscheme} to the solution of SDE \eqref{mainSDE}.  The first part of the proof is similar to the proof of strong uniqueness in \cref{t:VK} in \cref{s:SU}: for small $T-S$, we bound $[\phi-\phi^n]_{\C^{\frac12} L_m([S,T])}$ by half of the same quantity. However, this time the bound also involves an additional term, which will eventually determine the convergence rate. Therefore, we pass from small $T-S$ to arbitrary $T-S$ using \cref{l:ltg}. Finally, we use the Kolmogorov continuity theorem to bound the supremum of the error between the Euler scheme and the solution.

\begin{proof}[Proof of \cref{t:mainnumt}]
Without loss of generality, we assume $m>2$. Indeed, for ${1\le m\le2}$, the estimate \eqref{mbt35} follows from the case $m=3$, since $\|\cdot\|_{L_m(\Omega)}\le \|\cdot\|_{L_3(\Omega)}$.

\textbf{Step 1: Estimate on short time intervals}. We will use decomposition~\eqref{eq:error-decomposition}. Fix $\eps>0$. Let $(S,T)\in\Delta_{[0,1]}$. Using \cref{c:appradf} and \cref{e:step1uni} with $\tau=\frac12$, $\Gamma_0=\|b\|_{\C^0}$, we get for any $(s,t)\in\Delta_{[S,T]}$
\begin{align*}
\|(X_t-X^n_t)-(X_s-X^n_s)\|_{L_m(\Omega)}&=\|(\phi_t-\phi^n_t)-(\phi_s-\phi^n_s)\|_{L_m(\Omega)}\\
&\le\Bigl\|\int_s^t (b(W_r+\varphi_r)-b(W_r+\varphi^n_{r}))\,dr\Bigr\|_{L_m(\Omega)}
	\nn\\
	&\phantom{=}+\Bigl\|\int_s^t (b(W_r+\phi^n_{r})-b(W_{\kappa_n(r)}+\phi^n_{\kappa_n(r)}))\,dr\Bigr\|_{L_m(\Omega)}\\
	&\le C(t-s)^{\frac12+\frac\gamma2}\|\phi-\phi^n\|_{\C^0L_m([S,T])}\\
	&\phantom{=}+  C(t-s)^{1+\frac\gamma2}[\phi-\phi^n]_{\C^{\frac12} L_m([S,T])}+C(t-s)^{\frac12+\eps}n^{-\frac12-\frac\gamma2+\eps}
\end{align*}
for $C=C(\|b\|_{\C^\gamma},\gamma,\eps,d,m)>0$.
Dividing both sides by $(t-s)^{1/2}$ and taking supremum over all $(s,t)\in\Delta_{[S,T]}$ we get 
\begin{align}
&[\phi-\phi^n]_{\C^{\frac12} L_m([S,T])}\nn\\
&\,\,\le C(T-S)^{\frac\gamma2}\|\phi-\phi^n\|_{\C^0L_m([S,T])}+  C(T-S)^{\frac12+\frac\gamma2}[\phi-\phi^n]_{\C^{\frac12} L_m([S,T])}+Cn^{-\frac12-\frac\gamma2+\eps}\nn\\
&\,\,\le C_0(T-S)^{\frac\gamma2}\|\phi_S-\phi^n_S\|_{L_m(\Omega)}+C_0(T-S)^{\frac12+\frac\gamma2}[\phi-\phi^n]_{\C^{\frac12} L_m([S,T])}+C_0n^{-\frac12-\frac\gamma2+\eps}\label{bucklnum}
\end{align}
for $C_0=C_0(\|b\|_{\C^\gamma},\gamma,\eps,d,m)>0$, where in the last inequality we used \eqref{norm2normgen}. 
Fix now $\ell>0$ such that
\begin{equation*}
C_0\ell^{\frac12+\frac\gamma2}\le\frac12
\end{equation*}
(we stress that $\ell$ does not depend on $n$). Then we get from \eqref{bucklnum} for any $(S,T)\in\Delta_{[0,1]}$ with $T-S\le \ell$
 \begin{equation}\label{diffbound}
[\phi-\phi^n]_{\C^{\frac12} L_m([S,T])}\le 2C_0\|\phi_S-\phi^n_S\|_{L_m(\Omega)}+2C_0n^{-\frac12-\frac\gamma2+\eps}.
\end{equation}

\textbf{Step 2: Passage to arbitrary time intervals}.
We apply \cref{l:ltg}(ii) to $f:=\phi-\phi^n$. 
Bound \eqref{diffbound} implies that there exists $C=C(\ell,C_0)$ independent of $n$ such that 
\begin{equation}\label{finbigint}
[\phi-\phi^n]_{\C^{\frac12} L_m([0,1])} \le CC_0n^{-\frac12-\frac\gamma2+\eps}\le Cn^{-\frac12-\frac\gamma2+\eps},
\end{equation}
for $C=C(\|b\|_{\C^\gamma},\gamma,\eps,d,m)>0$.

\textbf{Step 3: Supremum bound via Kolmogorov's theorem}.
By the Kolmogorov continuity theorem (\cref{t:kolmi}), inequality \eqref{finbigint} implies
for any $\delta\in(1/m,1/2)$
 \begin{equation*}
\|[X-X^n]_{\C^{\frac12-\delta}([0,1];\R^d)}\|_{L_m(\Omega)}=\|[\phi-\phi^n]_{\C^{\frac12-\delta}([0,1];\R^d)}\|_{L_m(\Omega)} \le Cn^{-\frac12-\frac\gamma2+\eps},
\end{equation*}
for $C=C(\|b\|_{\C^\gamma},\gamma,\delta,\eps,d,m)>0$. Recalling that $X_0-X^n_0=0$, we get the desired bound~\eqref{mbt35}.
\end{proof}

We note that the convergence rate in \eqref{mbt35} does not deteriorate for large $m$. This allows to get a bound on a.s. convergence rate by the Borel–Cantelli arguments.
\begin{corollary}
Under the conditions of \cref{t:mainnumt}, there exists a random variable $\eta$ such that for any $n\in\N$ 
\begin{equation*}
\sup_{t\in[0,1]}|X_t-X^n_t| \le \eta n^{-\frac12-\frac\gamma2+\eps}\qquad \text{$\P$-a.s.}
\end{equation*}
\end{corollary}
\begin{proof}
The following simple proof was suggested by the referee. Fix $\eps>0$ and let 
\begin{equation}\label{defznbc}
\rho:=\frac12+\frac\gamma2;\quad Z_n:=n^{\rho-\eps}\sup_{t\in[0,1]}|X_t-X^n_t|, n\in \N.
\end{equation}
Then applying \cref{t:mainnumt} with $\eps/2$ in place of $\eps$, we deduce for any $m\ge1$
\begin{equation*}
\E\sup_{n\in\N}|Z_n|^m \le\sum_{n=1}^\infty  \E|Z_n|^m\le  C\sum_{n=1}^\infty n^{m\rho-m\eps}n^{-m\rho+\frac12m\eps}= C\sum_{n=1}^\infty  n^{-\frac{m\eps}2}
\end{equation*}
for  $C=C(\|b\|_{\C^\gamma},\gamma,\eps,d,m)>0$.
Choosing $m=4/\eps$, we see that the above sum is finite. Thus $\eta:=\sup_{n\in\N}|Z_n|<\infty$ a.s. Recalling the definition of $Z_n$ in \eqref{defznbc}, we get the desired bound.
\end{proof}

\subsection{John--Nirenberg inequality}\label{s:JN}

Let us move on to the next problem and study well-posedness for SDEs driven by a symmetric $\alpha$-stable process and the strong rate of convergence of the Euler scheme, ${\alpha \in (0,2)}$. It turns out that certain parts of the proof in the Gaussian case do not work in this setting at all and must be replaced by new arguments. This highlights a common situation: while the original stochastic sewing lemma of L\^e, in the formulation of \cite{LeSSL}, might not be applicable to a given problem, one can often modify it or complement it with different ideas to make things work. In the upcoming subsections, we will see in detail how the proof strategy of \cref{t:VK,t:mainnumt} must be reworked in the setting of L\'evy processes.

Thus, we fix a probability space $(\Omega,\F,(\F_t)_{t\in[0,1]},\P)$ and let $\alpha\in(0,2)$. We say that a stochastic process $L\colon\Omega\times[0,1]\to\R^d$ is  an $(\F_t)$-$\alpha$-stable process if $L$ is adapted to $(\F_t)$, c\'adl\'ag,  $L_t-L_s$ is independent of $\F_s$ whenever $(s,t)\in\Delta_{[0,1]}$, and the characteristic function of $L$ is given by  \eqref{charlevy}. \cref{fig:example2} shows typical trajectories of an $\alpha$-stable process for various $\alpha\in(0,2)$.
\begin{figure}[h]
\centering
\subfloat[\centering $\alpha=0.8$]{{\includegraphics[width=4.63cm]{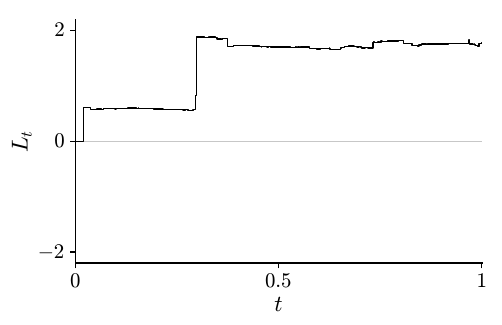} }}%
\qquad
\subfloat[\centering $\alpha=1.5$]{{\includegraphics[width=4.63cm]{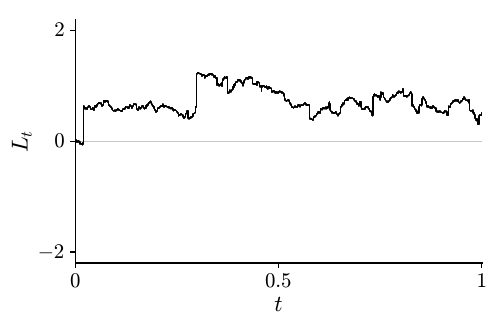} }}%
\qquad
\subfloat[\centering $\alpha=1.99$]{{\includegraphics[width=4.63cm]{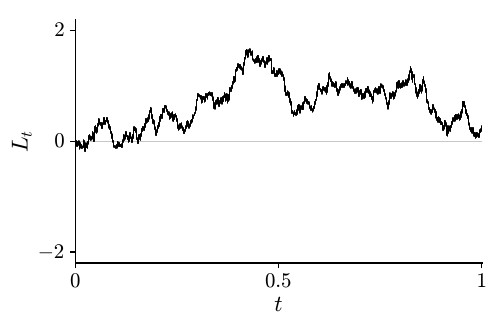} }}
\caption{Typical trajectories of a 1D symmetric $\alpha$-stable process for various~$\alpha$.}%
\label{fig:example2}%
\end{figure}
We see that as $\alpha$ approaches $2$ the trajectories start to look more like a trajectory of a standard Brownian motion (whose characteristic function satisfies \eqref{charlevy} with $\alpha=2$). On the other hand, as $\alpha$ approaches $0$, the trajectories start to have more and more large jumps and become smoother. Indeed, it is known that the trajectories of an $\alpha$-stable process have a.s. finite $p$-variation for any $p>\alpha$. In contrast to Brownian motion, $L_t$ has finite moments only of order strictly less than $\alpha$ and
it holds that
\begin{equation}\label{scalinglevy}
\E |L_t|^m\le C t^{\frac{m}\alpha},\quad t\in[0,1] 
\end{equation}
for $C=C(\alpha,d,m)>0$ whenever $m\in(0,\alpha)$.  We refer to \cite[Chapter~1]{Apple} for a detailed treatment of the theory of L\'evy processes.

If we try to prove an analogue of \cref{t:appradf} for an $\alpha$-stable process (in order to establish \cref{t:levy} and obtain strong convergence of the Euler scheme), we immediately encounter the following difficulty. The proof of \cref{t:appradf} relies heavily on the fact that for $r\in[0,1]$, $n\in\N$, $m\ge2$, $\gamma\in(0,1]$
\begin{equation*}
 \| f(W_r)-f(W_{\kappa_n(r)}) \|_{L_m(\Omega)} \le \|f\|_{\C^\gamma}\| |W_r-W_{\kappa_n(r)}|^\gamma \|_{L_m(\Omega)}\le C \|f\|_{\C^\gamma} n^{-\frac\gamma2},
\end{equation*}
see \eqref{easynum} and \eqref{ivantr}.
However, for an $\alpha$-stable L\'evy process, this is no longer true, as $L$ does not have moments of order  $\alpha$ and higher. Thus, the best we can obtain is just
\begin{equation}\label{observation}
	\| f(L_r)-f(L_{\kappa_n(r)}) \|_{L_m(\Omega)} \le \|f\|_{\C^\gamma}\| |L_r-L_{\kappa_n(r)}|^\gamma\wedge1 \|_{L_m(\Omega)}\le C\|f\|_{\C^\gamma} n^{-(\frac\gamma\alpha\wedge\frac1m)+\eps},
\end{equation}
so the bound deteriorates for large $m$. Consequently, we would get an $L_m(\Omega)$ convergence rate of order  $n^{-(\frac12 +\frac\gamma\alpha\wedge\frac1m)}$ and a.s. rate of order $n^{-\frac12}$. The correct rate, however, is of order  $n^{-((\frac12+\frac\gamma\alpha)\wedge1)}$. 

Note that for $m=2$, the bound \eqref{observation} is good enough, as it yields the correct convergence rate $n^{-(\frac{1}{2} + \frac{\gamma}{\alpha} \wedge \frac{1}{2})}$. Therefore, we need a tool that can transfer an $L_2$-moment bound to an arbitrary $L_m$-moment bound without any loss. This is indeed possible and is achieved via the John–Nirenberg inequality.
\begin{theorem}[John--Nirenberg inequality, {\cite[Exercise~A.3.2]{SVbook}},  {\cite[Theorem~2.3]{le2022}},
	{\cite[Appendix B]{BDGLevy}}] 
Let\label{t:john}
$0\le S\le T$ and let $\A\colon\Omega\times[S,T]\to\R^d$ be a continuous process adapted to the filtration $(\F_t)_{t\in[0,T]}$. Assume that there exists a constant $\Gamma>0$ such that
 for any $(s,t)\in\Delta_{[S,T]}$ one has 
\begin{equation}\label{BMOcon}
\E^s|\A_t-\A_s|\le \Gamma,\quad \text{a.s.}
\end{equation}
Then for any $m\in[1,\infty)$, there exists a constant $C>0$ independent of $d,m,S,T$ such that for any 
$s\in [S,T]$ one has 
\begin{equation}\label{jnres}
\|\sup_{r\in[s,T]}|\A_r-\A_s|\|_{L_m(\Omega)}\le C m\Gamma.
\end{equation}	
\end{theorem}

\begin{remark}
We emphasize that the continuity of $\A$ is crucial in \cref{t:john}. Indeed, if $\A$ is a standard Poisson process, then clearly $\E^s|\A_t-\A_s|\le t-s$, but of course it is not true that $\E|\A_t-\A_s|^m\le C(t-s)^m$ for $m>1$. Similarly, if $\A$ is an $\alpha$-stable process, $\alpha\in(1,2)$, then $\E^s|\A_t-\A_s|\le C(t-s)^{1/\alpha}$, but  $\E|\A_t-\A_s|^m=\infty$ for $m\ge\alpha$. We refer though to \cite{le2022} for a version of \cref{t:john} which allows jumps (under certain strong assumptions).
\end{remark}

We prove \cref{t:john} by arguing similarly to \cite{le2022maximal}. First, we establish 
 the so-called ``good $\lambda$  inequality'' of Burkholder \cite{Bur73}.
\begin{proposition}[Good $\lambda$  inequality]\label{p:khoaineq}
Let $X$, $Y$ be nonnegative random variables. Assume that for any $\lambda>0$, $\theta\in(0,1)$, $h>0$
one has 
	\begin{equation}\label{khoaineq}
		\P(X\ge(1+h)\lambda)\le \theta\P(X\ge\lambda )+\P(Y\ge\theta h\lambda).
	\end{equation}
	Then for any $m\ge1$ there exists a universal constant $C>0$ independent of $m$ such that
	\begin{equation}\label{rezmain}
		\|X\|_{L_m(\Omega)}\le Cm\|Y\|_{L_m(\Omega)}.
	\end{equation}	
\end{proposition}	
\begin{proof} Using Fubini's theorem and the identity $x^m=m\int_0^x \lambda^{m-1}\,d\lambda$ valid for any $x\ge0$, $m\ge1$, we derive for any $\theta\in(0,1)$, $M>0$, $h>0$
\begin{align*}
\E (X\wedge M)^m&=m\int_0^M \P(X\ge\lambda)\lambda^{m-1}\,d\lambda=(1+h)^{m}m\int_0^{\frac{M}{1+h}} \P\bigl(X\ge(1+h)\lambda\bigr)\lambda^{m-1}\,d\lambda\\
&\le 
(1+h)^{m}m\theta \int_0^{\frac{M}{1+h}} \P(X\ge\lambda)\lambda^{m-1}\,d\lambda
+(1+h)^{m} m \int_0^{\frac{M}{1+h}} \P(Y\ge\theta h\lambda)\lambda^{m-1}\,d\lambda\\
&\le (1+h)^{m} \theta \E (X\wedge M)^m+(1+h)^{m}\theta^{-m}h^{-m}\E Y^m,
\end{align*}
where in the penultimate inequality we use our main assumption \eqref{khoaineq}. Take now $\theta:=\frac12(1+h)^{-m}$. Then, by above
\begin{equation*}
		\frac12\E (X\wedge M)^m\le (1+h)^{m^2+m}2^mh^{-m}\E Y^m.
	\end{equation*}
Passing to the limit as $M \to \infty$ and using the monotone convergence theorem, we obtain
\begin{equation*}
\frac12\E X^m\le (1+h)^{m^2+m}2^mh^{-m}\E Y^m.
\end{equation*}
By taking now $h=\frac1m$ and using that $(1+\frac1m)^m\le e$, we get  \eqref{rezmain}.
\end{proof}	

Now we are ready to prove the John-Nirenberg inequality.
\begin{proof}[Proof of \cref{t:john}]
Fix $0\le S\le T$. Put 
\begin{equation}\label{vstar}
V^*(\omega):=\sup_{r\in[S,T]}|\A_r(\omega)-A_S(\omega)|,\qquad \omega\in\Omega.
\end{equation}	
We would like to apply \cref{p:khoaineq} with $X=V^*$, $Y=2\Gamma$. 
	
\textbf{Step~1}. We claim that condition \eqref{BMOcon} can be extended and for any stopping times $\tau\le\eta$ taking values in $[S,T]$ we have a.s.
\begin{equation}\label{tauetast}
	\E^{\tau}|\A_\eta-\A_\tau|\le 2\Gamma.
\end{equation}
First, we prove \eqref{tauetast} for the case when $\eta=T$ and the stopping time $\tau$ takes finitely many values ${S\le t_1<\hdots<t_n= T}$.
Recall the identity $\E^{\tau}(X\I_{\tau=t})=\E^t(X\I_{\tau=t})$, valid for any integrable random variable $X$, see, e.g., \cite[Problem~1.2.17(i)]{KS91}.
Then, using  \eqref{BMOcon}, we deduce 
\begin{equation}\label{mainthingjn1}
	\E^\tau|\A_T-\A_\tau|=\sum_{i=1}^n\E^{t_i}\bigl[|\A_{T}-\A_{t_i}| \I(\tau=t_i)\bigr]=
	\sum_{i=1}^n \I(\tau=t_i) \E^{t_i}|\A_{T}-\A_{t_i}|\le\Gamma.
\end{equation}	

Next, for a general stopping time $\tau \in [S,T]$, we consider a sequence of stopping times $\tau_n$ taking finitely many values in $[S,T]$ and converging to $\tau$ from above. Then, using the continuity of $\A$ and \eqref{mainthingjn1}, we deduce  (note that the continuity is crucial here)
	\begin{equation}\label{mainthingjn2}
		\E^\tau|\A_T-\A_{\tau}|\le \liminf_{n\to\infty} \E^\tau|\A_T-\A_{\tau_n}|=\liminf_{n\to\infty} \E^\tau\E^{\tau_n}|\A_T-\A_{\tau_n}|\le \Gamma.
	\end{equation}
	Finally, let $\tau\le \eta$ be arbitrary stopping times taking values in $[S,T]$. Then
	\eqref{mainthingjn2} yields
	\begin{equation*}
		\E^\tau|\A_\eta-\A_{\tau}|\le \E^\tau|\A_T-\A_{\tau}|+\E^\tau|\A_T-\A_{\eta}|\le\Gamma +\E^\tau\E^\eta |\A_T-\A_{\eta}|\le 2\Gamma,
	\end{equation*}
	which is \eqref{tauetast}.
	
\textbf{Step~2}.
Let us verify that condition \eqref{khoaineq} holds for $X=V^*$, $Y=2\Gamma$. Take arbitrary $\lambda,h>0$ and consider two stopping times:
\begin{equation*}
\tau:=T\wedge \inf\{r\in[S,T]: |\A_r-\A_S|\ge\lambda\};\qquad 
\eta:=T\wedge \inf\{r\in[S,T]: |\A_r-\A_S|\ge(1+h)\lambda\}.
\end{equation*}
Then, by  continuity of $\A$, we get
\begin{align}\label{mainineq}
\P\bigl(V^*\ge(1+h)\lambda\bigr)&=\P\bigl(|\A_\eta-\A_S|=(1+h)\lambda\bigr)\nn\\
&=\P\bigl(|\A_\tau-\A_S|=\lambda, |\A_\eta-\A_S|-|\A_\tau-\A_S|=h\lambda)\nn\\
&\le\P\bigl(|\A_\tau-\A_S|=\lambda, |\A_\eta-\A_\tau|\ge h\lambda)\nn\\
&=\E [\I_{\{|\A_\tau-\A_S|=\lambda\}}\E^\tau \I_{\{|\A_\eta-\A_\tau|\ge h\lambda\}}]\nn\\
&\le \E \Bigl[\I_{\{|\A_\tau-\A_S|=\lambda\}}\E^\tau \frac{|\A_\eta-\A_\tau|}{h\lambda}\Bigr]\nn\\
&\le 2\Gamma h^{-1}\lambda^{-1}\E \I_{\{|\A_\tau-\A_S|=\lambda\}}=2\Gamma h^{-1}\lambda^{-1}\P(V^*\ge \lambda),
\end{align}
where in the last inequality we used \eqref{tauetast}.

Now let $\theta\in(0,1)$. If $\theta\ge \frac{2\Gamma}{h\lambda}$, then \eqref{mainineq} implies
\begin{equation*}
\P\bigl(V^*\ge(1+h)\lambda\bigr)
\le \theta \P(V^*\ge \lambda),
\end{equation*}
If $\theta\le \frac{2\Gamma}{h\lambda}$, then $\P(2\Gamma\ge \theta h\lambda)=1$ and we trivially have 
\begin{equation*}
	\P\bigl(V^*\ge(1+h)\lambda\bigr)
	\le 1=\P(2\Gamma\ge \theta h\lambda).
\end{equation*}
Since $h,\lambda>0$ and $\theta\in(0,1)$ were arbitrary, we see that in both cases condition \eqref{khoaineq} holds. Hence, by the good $\lambda$ inequality (\cref{p:khoaineq}), we  get
\begin{equation*}
\|V^*\|_{L_m(\Omega)}\le C m \Gamma,
\end{equation*}
for $C>0$ independent of $m,d$, 	which is the desired bound \eqref{jnres}.
\end{proof}

\subsection{Integral functionals of a L\'evy process}\label{s:levyint}
Now, equipped with the John--Nirenberg inequality, we can extend  \cref{t:appradf} to the $\alpha$-stable case and overcome the issue that the rate deteriorates for large $m$. In this and the next subsections, we follow \cite{BDGLevy}.

Let $d\in\N$, $L$ be a $d$-dimensional symmetric $\alpha$-stable process with the characteristic function \eqref{charlevy}, $\alpha\in(0,2)$. We denote its density by $\pa_t$ and let $\Pa_t$, $t\ge0$, be  the corresponding semigroup. 

The following extension of \cref{t:appradf} holds.

\begin{theorem}\label{t:levydif}
Let $\alpha\in(0,2)$, $\gamma\in(0,1]$, $f\in \C^\gamma(\R^d,\R)$, $m\ge2$, $\eps>0$. Then there exists a constant $C=C(\alpha,\gamma,\eps,d,m)>0$ such that for any $n\in\N$ one has 
\begin{equation}\label{numlevy}
\Bigl\|\sup_{t\in[0,1]}\Bigl|\int_0^t (f(L_r)-f(L_{\kappa_n(r)}))\,dr\Bigr|\,\Bigr\|_{L_m(\Omega)}\le  C  \| f\|_{\C^\gamma}n^{-((\frac12+\frac\gamma\alpha)\wedge1)+\eps}.
\end{equation}
\end{theorem}

Note that the convergence rate in the right-hand side of \eqref{numlevy} does not decrease for large $m$.
To obtain this result, we need  analogues of Gaussian bounds provided in \cref{l:gb,l:gb2} for the L\'evy case.
\begin{lemma}\label{l:fractional}
	Let $f\in\C^\gamma(\R^d,\R)$, $\alpha\in(0,2)$, $\gamma\in[0,1]$, $\eps>0$. Then there exists ${C=C(\alpha,\gamma,\eps,d)>0}$ such that for any  $0<s\le t\le 1$ we have 
	\begin{align}\label{ineqlevy}
		&\|\Pa_t  f\|_{\C^1}\le C t^{\frac{\gamma-1}\alpha}\|f\|_{\C^\gamma};\quad
		\|\nabla \Pa_t  f\|_{\C^1}\le C t^{\frac{\gamma-2}\alpha}\|f\|_{\C^\gamma}; \\
		&\|\Pa_tf-\Pa_s  f\|_{\C^0}\le  C\| f\|_{\C^\gamma} s^{(\frac\gamma\alpha\wedge1)-1-\eps}(t-s).\label{ineqlevy2}
	\end{align} 	
\end{lemma}

The proof of  \cref{l:fractional} follows from a scaling argument and is left to the reader, see \cref{levyboundex}.  

Since $L$ has independent increments, it is also easy to see that the following analogue of \eqref{1bound} holds: for any bounded measurable $f\colon\R^d\to\R$, $0\le s< t$, $x\in\R^d$ we have
\begin{equation}\label{l-1bound}
	\E^s f(L_t+x)=\Pa_{t-s}f (L_s+x)
\end{equation}

\begin{proof}[Proof of \cref{t:levydif}]
	
Let us see how the John--Nirenberg inequality helps us to overcome the obstacle described at the beginning of \cref{s:JN}. The idea is to obtain \eqref{numlevy} first for $m=2$ using a similar strategy as before (stochastic sewing), and then extend it to the case of arbitrary $m>2$ using the John--Nirenberg inequality.
Fix $\eps>0$ small enough. In particular, we assume $\eps<\frac\gamma\alpha$.

\textbf{Step 0}. Without loss of generality, and to simplify the notation, we additionally assume that 
\begin{equation}\label{wlog}
\gamma<\frac\alpha2.
\end{equation}
Indeed, if $\gamma\ge\frac\alpha2$, then $f\in\C^{\gamma}\subset\C^{\frac\alpha2-\alpha\eps}$, and we may prove the theorem with $\frac\alpha2-\alpha \eps$ in place of~$\gamma$. Since in this case the right-hand side of \eqref{numlevy} remains unchanged when $\gamma$ is replaced by $\frac\alpha2-\alpha \eps$, this implies the desired estimate also for $\gamma$.

\textbf{Step 1: Case $m=2$}.  We apply the stochastic sewing lemma to the germ 
\begin{equation*}
	A_{s,t}:=\E^s \int_s^t (f(L_r)-f(L_{\kappa_n(r)}))\, dr,\quad  (s,t)\in\Delta_{[S,T]}.
\end{equation*}
It is easy to see that $\E^s \delta A_{s,u,t}=0$ for any $(s,u,t)\in\Delta^3_{[S,T]}$ and therefore  condition \eqref{con:s2}  holds. 

To verify  \eqref{con:s1}, as in the proof of \cref{t:appradf}, for $s\in[S,T]$ we define  $s': =\kappa_n(s) + \frac2n$; that is $s'$ is the second gridpoint to the right of $s$. Then we have for any $(s,t)\in\Delta_{[S,T]}$
\begin{align}\label{step1numle}
	\|A_{s,t}\|_{L_2(\Omega)}&\le \Bigl\|\int_s^{s'\wedge t}
	\E^s \big( f(L_r)-f(L_{\kappa_n(r)})\big) dr\Bigr\|_{L_2(\Omega)}\!\!\!+
	\Bigl\|\int_{s'\wedge t}^t
	\E^s \big( f(L_r)-f(L_{\kappa_n(r)})\big) dr\Bigr\|_{L_2(\Omega)}\nn\\
	& =:I_1+I_2
\end{align}

As we explained in \cref{s:JN}, the first term on the right-hand side of \eqref{step1numle} is dangerous in the L\'evy case:  for large $m$, it reduces the convergence rate because $L$ has moments only up to order $\alpha$, and we need the moment of order $m\gamma$. However, for $m=2$, we are safe, recall assumption \eqref{wlog}. Using \eqref{scalinglevy}, we derive
\begin{align}\label{s2levy}
I_1&\le  \int_s^{s'\wedge t} \| f(L_r)-f(L_{\kappa_n(r)}) \|_{L_2(\Omega)} \, dr\le  \| f\|_{\C^{\gamma}}\int_s^{s'\wedge t} \bigl(\E |L_r-L_{\kappa_n(r)}|^{2\gamma}\bigr)^{\frac12} \, dr\nn\\
	&\le C \| f\|_{\C^\gamma}((s'\wedge t)-s) n^{-\frac\gamma\alpha} \le C \| f\|_{\mathcal{C}^\gamma} n^{-\frac12-\frac\gamma\alpha+\eps}(t-s)^{\frac12+\eps},
\end{align}
for $C=C(\alpha,\gamma,d)>0$. Here in the last inequality we used that $(s'\wedge t)-s \le t-s$ and 
$(s'\wedge t)-s \le 2n^{-1}$ by definition of $s'$.

Next, we treat $I_2$ using \cref{l:fractional}. If $s'\wedge t=t$, then there is nothing to prove, $I_2=0$. Otherwise,  if $s'\wedge t=s'$, then $t\ge s'\ge s + \frac1n$ and we get 
\begin{align}
	I_2&=\Bigl\|\int_{s'}^t(\Pa_{r-s}f(L_s)-\Pa_{\kappa_n(r)-s}f(L_s))\,dr
	\Bigr\|_{L_2(\Omega)}\le \int_{s'}^t\|\Pa_{r-s}f-\Pa_{\kappa_n(r)-s}f\|_{\C^0}\,dr\nn\\
	&\le C\| f\|_{\C^\gamma}\int_{s'}^t(\kappa_n(r)-s)^{\frac\gamma\alpha-1-\eps}n^{-1}\,dr\nn\\
	&\le C\| f\|_{\C^\gamma}\int_{s'}^t(r-s')^{\frac\gamma\alpha-1-\eps}n^{-1}\,dr\le C\| f\|_{\C^\gamma}(t-s')^{\frac\gamma\alpha-\eps}n^{-1}\label{four34}\\
	&\le C  \| f\|_{\C^\gamma}(t-s)^{\frac12+\eps}n^{-\frac12-\frac\gamma\alpha+2\eps}\nn
\end{align}
for $C=C(\alpha,\gamma,\eps, d)>0$. Here in \eqref{four34} we used that $r-\kappa_n(r)\le\frac1n\le s'-s$ and hence $\kappa_n(r)-s\ge r-s'$ and in the last inequality we used that $t-s\ge\frac1n$. Now, combining this with \eqref{s2levy} and \eqref{step1numle}, we get 
\begin{equation*}
	\|A_{s,t}\|_{L_2(\Omega)}\le C  \| f\|_{\C^\gamma}(t-s)^{\frac12+\eps}n^{-\frac12-\frac\gamma\alpha+2\eps}
\end{equation*}
and thus condition \eqref{con:s1} is satisfied. 

It is easy to see that both of the remaining conditions ((i) and (ii)) of \cref{T:SSLst} hold, and
we get for any $(S,T)\in\Delta_{[0,1]}$
\begin{equation}\label{numlevy2}
\Bigl\|\int_S^T (f(L_r)-f(L_{\kappa_n(r)}))\,dr\Bigr\|_{L_2(\Omega)}\le  C \| f\|_{\C^\gamma}(T-S)^{\frac12+\eps}n^{-\frac12-\frac\gamma\alpha+2\eps},
\end{equation}
for $C=C(\alpha,\gamma,\eps, d)>0$.

\textbf{Step 2: Case $m>2$}. Now we extend $L_2(\Omega)$ bound to $L_m(\Omega)$ bound.
We apply \cref{t:john} (John--Nirenberg inequality) with $S=0$, $T=1$ to the process 
\begin{equation*}
\A_t:=\int_0^t (f(L_r)-f(L_{\kappa_n(r)}))\,dr, \quad t\in[0,1].
\end{equation*}
Obviously, the process $\A$ is continuous because $f$ is bounded. 

Now let us verify condition \eqref{BMOcon}. As before, we run into an issue with conditioning: since the expression for $\A_t$ includes $L_{\kappa_n(r)}$, it may happen that $s \le r$ but $s > \kappa_n(r)$, in which case no smoothing occurs. To overcome this, we use a similar trick as before, distinguishing two cases. If $s$ is a grid point, then $s \le r$ automatically implies $s \le \kappa_n(r)$. Otherwise, we split the integral into two parts: a small part, where we apply a rough bound using $\|f\|_{\C^0}$, and a larger part, where we can exploit the conditioning.

Thus, we take arbitrary $(s,t)\in\Delta_{[0,1]}$. If $s$ is a gridpoint (so that $\kappa_n(s)=s$), we get
\begin{align}\label{casegrid}
\E^s|\A_t-\A_s|&=\E^s\Bigl|\int_s^t  (f(L_r)-f(L_{\kappa_n(r)}))\,dr\Bigr|\nn\\
&=\E\Bigl|\int_s^t  (f(L_r-L_s+x)-f(L_{\kappa_n(r)}-L_s+x))\,dr\Bigr|\evalat{x=L_s}.
\end{align}
It follows from \eqref{numlevy2} and stationarity of the increments of $L$ that for any $x\in\R^d$
\begin{align*}
\E\Bigl|\int_s^t  (f(L_r-L_s+x)-f(L_{\kappa_n(r)}-L_s+x))\,dr\Bigr|&=
\E \Bigl|\int_0^{t-s}  (f(L_{r}+x)-f(L_{\kappa_n(r)}+x))\,dr\Bigr|\\
&\le  C \|f\|_{\C^\gamma}n^{-\frac12-\frac\gamma\alpha+2\eps},
\end{align*}
for $C=C(\alpha,\gamma,\eps, d)>0$, and we used again that $s$ is a gridpoint. Substituting this into \eqref{casegrid}, we get 
\begin{equation}\label{casegridfin}
	\E^s|\A_t-\A_s|\le C \|f\|_{\C^\gamma}n^{-\frac12-\frac\gamma\alpha+2\eps}.
\end{equation}

If $s$ is not a gridpoint, we denote as before $s':=\kappa_n(s)+\frac2n$ --- the second gridpoint to the right of $s$.  We get
\begin{align}
\E^s|\A_t-\A_s|&\le \E^s|\A_{s'\wedge t}-\A_s|+\E^s|\A_t-\A_{s'\wedge t}|\le C \|f\|_{\C^0}|s'-s|+\E^s\E^{s'}|\A_t-\A_{s'\wedge t}|\nn\\
&\le C\|f\|_{\C^0}n^{-1}+C \|f\|_{\C^\gamma}n^{-\frac12-\frac\gamma\alpha+2\eps}\le C \|f\|_{\C^\gamma}n^{-\frac12-\frac\gamma\alpha+2\eps},\label{wheretoargue}
\end{align}
for $C=C(\alpha,\gamma,\eps, d)>0$
where the penultimate inequality follows from \eqref{casegridfin}. Thus, all conditions of \cref{t:john} are satisfied and we get for any $m\ge2$
\begin{equation*}
	\Bigl\|\sup_{t\in[0,1]}\Bigl|\int_0^t (f(L_r)-f(L_{\kappa_n(r)}))\,dr\Bigr|\,\Bigr\|_{L_m(\Omega)}=
	\|\sup_{t\in[0,1]}|\A_t-\A_0|\,\|_{L_m(\Omega)}\le   C \|f\|_{\C^\gamma}n^{-\frac12-\frac\gamma\alpha+2\eps}
\end{equation*}
for $C=C(\alpha,\gamma,\eps, d,m)>0$.
This implies the statement of the theorem.
\end{proof}

\subsection{Estimates in conditional H\"older norms}\label{sec:44}

The next milestone on our way toward establishing \cref{t:levy} is extending the well-posedness results of \cref{t:VK} to SDEs driven by $\alpha$-stable noise. Let $X$ be a solution to SDE
\begin{equation*}
X_t=x+\int_0^t b(X_r)dr +L_t,
\end{equation*}
where $L$ is a symmetric $d$-dimensional $\alpha$-stable process, $x\in\R^d$, $b\in\C^\gamma(\R^d,\R^d)$, $\gamma>1-\frac\alpha2$, $\alpha\in(0,2)$.
If we try to apply the same strategy as in \cref{s:s3}, two important new obstacles appear.

The first problem is related to \cref{r:d}. When proving \cref{e:step1uni}, we had to impose the following regularity condition on the drift $\phi$:
\begin{equation}\label{regcond}
|\phi_t-\phi_s|\le \Gamma_0 |t-s|\quad a.s.
\end{equation}	
for some $\Gamma_0>0$.
In the setting of \cref{t:VK}, this condition was optimal: indeed, if $\phi_t=\int_0^t b(X_r)\,dr$, where the function $b$ is bounded, then we do not expect $\phi$ to have better a.s. regularity. However, in our new setting, we work with much nicer drifts belonging to a certain  H\"older space. Therefore, while  condition \eqref{regcond} still obviously holds, it does not capture the full information about $\phi$, and would therefore inevitably lead to suboptimal results. 

The second problem was described in \cref{r:vd}. If we perform exactly the same calculations as in the proof of \cref{e:step1uni} but now for $\alpha < 2$, then the first term on the right-hand side of \eqref{verydangerous} becomes
\begin{equation}\label{intprblem}
C\Gamma_0\|\phi-\psi\|_{\C^0L_m([S,T])}|u-s|\|b\|_{\C^\gamma}\int_u^t (r-u)^{\frac{\gamma-2}\alpha}\,dr
\end{equation}
For the integral to be finite, we need to require $\frac{\gamma-2}\alpha>-1$, that is 
$\gamma>2-\alpha$, which is much worse than our desired condition $\gamma>1-\frac\alpha2$, see \eqref{levycond}.

In this subsection, we will deal with the first problem and postpone the solution of the second problem to the next subsection.

To solve the first problem, let us look at  the regularity of
 $\phi_t:=x+\int_0^t b(X_r)dr$. Since $X$ is not continuous, we do not expect that $\phi\in\C^\theta$ a.s. for any $\theta>1$. Note however, that $X$ is stochastically continuous (that is, $X_{t+\eps}\to X_t$ in probability as $\eps\to0$) and it is not difficult to see that 
\begin{equation*}
\E	|\phi_t-\phi_s-(t-s)\phi'_s|\le C \|b\|_{\C^\gamma} |t-s|^{1+\frac\gamma\alpha},
\end{equation*}
so the ``stochastic regularity'' of $\phi$ is $1+\frac\gamma\alpha$, provided that $\frac\gamma\alpha<1$, see \cref{c:413} below.

Next, observe that there was an additional smoothing when we passed from \eqref{crucial} to \eqref{verydangerous}, and what we actually need is an a.s. bound on $\E^s|\phi_t-\phi_s|$. Unfortunately, this still does not give us the desired exponent $1+\frac\gamma\alpha$, as the best we get is only
\begin{equation*}
	\E^s|\phi_t-\phi_s|\le C(t-s).
\end{equation*}
However, if we change the approximation of $L_r+\phi_r$ in the germ  \eqref{germper} from $L_r+\phi_s$ to $L_r+\phi_s+(r-s)\phi'_s$, then when passing from \eqref{crucial} to \eqref{verydangerous} we would need to bound $\E^s	|\phi_u-\phi_s-(u-s)\phi'_s|$, which is much smaller. 

We now take one step further and observe that the best possible approximation of a random variable by an $\F_s$-measurable random variable in $L_2(\Omega)$ is its conditional expectation. Therefore, motivated by the above discussion, we introduce the following seminorm. Let $m\ge1$, $\tau>0$, $(S,T)\in\Delta_{[0,1]}$, and let $f$ be a measurable function $\Omega\times[S,T]\to\R^d$ such that $\E |f_t|<\infty$ for any $t\in[S,T]$. We  put
\begin{equation}
	\dnew{f}{\tau}{m}{[S,T]}:=\sup_{(s,t)\in\Delta_{[S,T]}}\frac{\|f(t)-\E^s f(t)\|_{L_m(\Omega)}}{| t-s|^\tau}.\label{newnorm}
\end{equation}

\begin{remark}\label{r:restr}
For $\alpha\in(1,2)$, the restriction $\frac\gamma\alpha<1$ in the stochastic regularity of $\phi$
is harmless, since $\gamma\le1$. For $\alpha<1$, this condition is not so innocent. Recall that we are interested
in the regime $\gamma>\gamma_{\min}:=1-\frac\alpha2$. If $\alpha>\frac23$, we have $\frac{\gamma_{\min}}\alpha<1$
and we are safe. However, if $\alpha<\frac23$, then $\gamma_{\min}>{\alpha}$ and thus the stochastic regularity of $\phi$ deteriorates.
Strong well-posedness of SDE \eqref{mainSDElevy} (proven without stochastic sewing) holds in the whole regime $\alpha\in(0,2)$, $\gamma>1-\frac\alpha2$; see  \cite{bib:csz15}. However, the only available numerical approximation results with optimal
convergence rates are obtained via stochastic sewing. Therefore they are known only in the regime $\alpha\in(\frac23,2)$, $\gamma>1-\frac\alpha2$, see \cite{BDGLevy}. For $\alpha\in(0,\frac23)$, numerical approximation in this whole
regularity range remains open.
\end{remark}
Thus, from now on, to highlight the main ideas and simplify the
presentation, we assume $\alpha>1$, which, as discussed above,
yields $\frac\gamma\alpha<1$.

The following useful lemma shows that the conditional expectation is the best approximation (up to a constant) not only in the $L_2(\Omega)$ sense, but in any $L_m(\Omega)$, $m\ge1$. It implies that the seminorm $\dnew{\cdot}{\tau}{m}{[S,T]}$ is weaker than $[\cdot]_{\C^\tau L_m([S,T])}$ used in \cref{s:s3}.

To simplify the notation we write here $\E^\G[\cdot]:=\E[\cdot|\G]$.

\begin{lemma}\label{lem:useful-lemma}
 Let $m\ge1$. Let $\G\subset \F$ be a $\sigma$-algebra. Let random variables $Y,Z\in L_m(\Omega)$ and suppose that $Z$ is $\G$--measurable. Then  
\begin{align}\label{YZcond}
&\|Y-\E^\G Y\|_{L_m(\Omega)}\le 2\|Y-Z\|_{L_m(\Omega)};\\
&\E^\G|Y-\E^\G Y|\le 2\E^\G |Y-Z|;\label{YZcond2}
\end{align}
\end{lemma}

\begin{proof}
Fix $m\ge1$. Using that $Z$ is $\G$-measurable, we derive
\begin{align*}
\|Y- \E^\G Y\|_{L_m(\Omega)}&\le \|Y-Z\|_{L_m(\Omega)}+\|\E^\G Y-Z\|_{L_m(\Omega)}\\
&= \|Y-Z\|_{L_m(\Omega)}+\|\E^\G (Y-Z)\|_{L_m(\Omega)}\le 2\|Y-Z\|_{L_m(\Omega)},
\end{align*}
where the last inequality follows from Jensen's inequality \eqref{useful}. 

Very similarly,
\begin{equation*}
\E^\G|Y-\E^\G Y|\le  \E^\G |Y-Z|+\E^\G|Z-\E^\G Y|=\E^\G |Y-Z|+|\E^\G(Y-Z)|\le 2\E^\G |Y-Z|.\qedhere
\end{equation*}
\end{proof}

\begin{corollary}\label{c:normbound}
Let $(S,T)\in \Delta_{[0,1]}$, $m\ge1$,  $\tau>0$. Let $f$ be a measurable function $\Omega\times [S,T]\to\R^d$ adapted to the filtration $(\F_t)$ such that $\E |f_t|<\infty$ for any $t\in[S,T]$. Then
\begin{equation*}
\dnew{f}{\tau}{m}{[S,T]}\le2[f]_{\C^\tau L_m([S,T])}.
\end{equation*}
\end{corollary}

\begin{proof}
We just note that for any $(s,t)\in\Delta_{[S,T]}$ we have thanks to \cref{lem:useful-lemma}
\begin{equation*}
\|f(t)-\E^s f(t)\|_{L_m(\Omega)}\le 2 \|f(t)-f(s)\|_{L_m(\Omega)},
\end{equation*}
because $f_s$ is $\F_s$-measurable by assumption. Thus
\begin{equation*}
\dnew{f}{\tau}{m}{[S,T]}\le 
2\sup_{(s,t)\in\Delta_{[S,T]}}\frac{\|f(t)- f(s)\|_{L_m(\Omega)}}{| t-s|^\tau}=2[f]_{\C^\tau L_m([S,T])}\qedhere
\end{equation*}
\end{proof}

Finally, let us also bound the conditional seminorm of a solution to our SDE \eqref{mainSDElevy}.

\begin{corollary}\label{c:413}
Let $\alpha\in(1,2)$, $x\in\R^d$, $b\in\C^\gamma(\R^d,\R^d)$, $\gamma\in[0,1]$.
Let $X$ be a solution to \eqref{mainSDElevy} and denote $\phi:=X-L$. Then 
there exists a constant $C=C(\alpha,d)>0$ such that for any $(s,t)\in\Delta_{[0,1]}$
\begin{equation}\label{eq:useful-bound}
\|\E^s|\phi_t-\E^s\phi_t|\|_{L_\infty(\Omega)}\le C \|b\|_{\C^\gamma}(1+\|b\|_{\C^0})(t-s)^{1+\frac\gamma\alpha}.
\end{equation}
\end{corollary}
\begin{proof}
Let $(s,t)\in\Delta_{[0,1]}$. Since the random variable $\phi_s+(t-s) b(X_s)$ is $\F_s$ measurable, by \cref{lem:useful-lemma}, we get
\begin{align*}
\E^s|\phi_t-\E^s \phi_t|&\le 2\E^s |\phi_t-\phi_s-(t-s)b(X_s)|=2 \E^s\Bigl|\int_s^t (b(X_r)-b(X_s))\,dr\Bigr|\\
&\le 2 \|b\|_{\C^\gamma}\int_s^t (\E^s |L_r-L_s|^\gamma+\E^s |\phi_r-\phi_s|^\gamma)\,dr\le C 
\|b\|_{\C^\gamma}(1+\|b\|_{\C^0})|t-s|^{1+\frac\gamma\alpha},
\end{align*}
for $C=C(\alpha,d)>0$ and we used that $|\phi_r-\phi_s|\le \|b\|_{\C^0}|r-s|$. This proves \eqref{eq:useful-bound}.
\end{proof}

\subsection{Shifted stochastic sewing lemma}\label{s:sssl}

To overcome the second problem discussed in \cref{sec:44} (see also \cref{r:vd}), we use Gerencs\'er's idea to replace the stochastic sewing lemma (\cref{T:SSLst}) with a \textit{shifted} stochastic sewing lemma. Essentially, this allows us to change the factor $(r-u)$  in the integrand of \eqref{intprblem}  to 
$r-(u-(t-s))$. As a result, the  integral becomes finite and of order $(t-s)^{\frac{\gamma-2}\alpha+1}$ as desired, even when $\frac{\gamma-2}\alpha<-1$.

This is also the first example where we see a general principle in action: SSL is not a one-size-fits-all method. Its components can be adjusted to overcome a specific obstacle without losing the core structure. For further extensions of the shifted SSL, we refer to \cite{MP22} and \cite[Lemma~2.7]{anzeletti2025density}.

For $0\le S\le T$ we introduce  a modified simplex 
\begin{equation*}
\ms{[S,T]}:=\{(s, t)\in[S,T]^2:  s< t,\,\, s - (t - s) \ge S\}.
\end{equation*}

\begin{lemma}[Shifted stochastic sewing lemma, {\cite[Lemma~2.2]{Gerreg22}}]
\label{lem:shiftedmodifiedSSL}

Let $m\in[2,\infty)$, $0\le S< T$. 
Suppose that there exist measurable functions ${\A\colon \Omega\times[S,T]\to \R^d}$, $A\colon\Omega\times \ms{[S,T]}\to \R^d$ and a complete filtration $(\F_t)_{t\in[S,T]}$ such that the following holds:
\begin{enumerate}[(i)]
\item 
one has
\begin{equation}\label{s-limpart}
	\A_T-\A_S=\lim_{n\to\infty} \sum_{i=1}^{n-1} A_{S+i\frac{T-S}n,S+(i+1)\frac{T-S}n}\,\, \text{in probability.}
\end{equation}

	\item for any $(s,t)\in \ms{[S,T]}$, the random variable $A_{s,t}$ is $\F_t$--measurable;
	\item there exist  $\Gamma_1, \Gamma_2\ge0$, $\gamma_1>\frac12$, $\gamma_2>1$ such that for every $(s,t)\in\ms{[S,T]}$ and $u=(s+t)/2$ we have 
	\begin{align}	
		&\| A_{s, t}\|_{L_m(\Omega)}\le \Gamma_1|t-s|^{\gamma_1},\label{s-con:s1}	\\
		&\|\E^{s-(t-s)} \delta A_{s,u,t}\|_{L_m(\Omega)}\le \Gamma_2|t-s|^{\gamma_2}.\label{s-con:s2}
	\end{align}
\end{enumerate}	
Then there exists a constant $C=C(\gamma_1,\gamma_2,d,m)>0$ independent of $S$, $T$ such that 
	\begin{equation}
		\|\A_{T}-\A_{S}\|_{L_m(\Omega)} \le C \Gamma_1 |T-S|^{\gamma_1}+
		C \Gamma_2 |T-S|^{\gamma_2}.
		\label{s-est:ssl1}
	\end{equation}
\end{lemma}

\begin{remark}
\begin{enumerate}[(i)]
\item The key difference between the shifted SSL and the usual SSL is condition \eqref{s-con:s2}, which allows for additional smoothing by replacing the conditional expectation $\E^s$ with $\E^{s - (t - s)}$. We will see the benefits of this later in \cref{s:levywp}.
\item Similarly to \cref{r:SSL}, let us point out that condition (i) actually follows from condition (ii)  and a stronger version
of condition (iii). Since this weaker version of the lemma is enough for our applications, we impose condition (i) here to significantly simplify the proof.
\end{enumerate}
\end{remark}

The proof relies on the following simple bound, which is a corollary of the BDG inequality.
\begin{proposition}\label{p:cBDG}
Let $n\in\N$. Let  $(X_i)_{i=2,\hdots,n}$ be a sequence of random vectors in $\R^d$ adapted to 
the filtration $(\G_i)_{i=0,\hdots,n}$. Then for any  $m\ge2$ there exists a constant $C=C(m)>0$ independent of $n$ such that
\begin{equation}\label{matebound}
\Bigl\|\sum_{i=2}^n X_i\Bigr\|_{L_m(\Omega)}
\le C\sum_{i=2}^n\| \E[X_i|\mathcal{G}_{i-2}]\|_{L_m(\Omega)}+
C\Bigl(\sum_{i=2}^n\| X_i\|^2_{L_m(\Omega)}\Bigr)^{1/2}.
\end{equation}
\end{proposition}
\begin{proof}
We have
\begin{align}
\Bigl\|\sum_{i=2}^n X_i \Bigr\|_{L_m(\Omega)}&\le\Bigl\|\sum_{\substack{i\in[2,n]\\\text{ $i$ is even}}} (X_i- \E [X_i|\G_{i-2}])\Bigr\|_{L_m(\Omega)}+\Bigl\|\sum_{\substack{i\in[2,n]\\\text{ $i$ is odd}}} (X_i- \E [X_i|\G_{i-2}])\Bigr\|_{L_m(\Omega)}\nn\\
&\phantom{\le}+\sum_{i=2}^n\| \E[X_i|\mathcal{G}_{i-2}]\|_{L_m(\Omega)}\label{matebdg1}.
\end{align}
The sequence $(X_i- \E [X_i|\G_{i-2}])_{i\in[2,n],\text{ $i$ is even}}$ is a martingale difference sequence. Therefore, using  \eqref{BDG} and Jensen's inequality, we derive
\begin{align*}
\Bigl\|\sum_{\substack{i\in[2,n]\\\text{ $i$ is even}}} (X_i- \E [X_i|\G_{i-2}])\Bigr\|^2_{L_m(\Omega)}&\le 
C\sum_{\substack{i\in[2,n]\\\text{ $i$ is even}}} \|X_i- \E [X_i|\G_{i-2}]\|^2_{L_m(\Omega)}\\
&\le 
C\sum_{i\in[2,n]} (\|X_i\|^2_{L_m(\Omega)}+ \|\E [X_i|\G_{i-2}]\|^2_{L_m(\Omega)})\\
&\le 
C\sum_{i\in[2,n]} \|X_i\|^2_{L_m(\Omega)}
\end{align*}
for $C=C(m)$. Clearly, the same bound holds for the second term in the right-hand side of \eqref{matebdg1}. Substituting this into \eqref{matebdg1}, we get the desired inequality \eqref{matebound}.
\end{proof}

\begin{proof}[Proof of \cref{lem:shiftedmodifiedSSL}]
Let us modify appropriately the proof of the SSL. Let $m\ge2$. Fix $0\le S< T$ and consider the dyadic partition $t^n_i=S+\frac{i}{2^{n}}(T-S)$, $i=0, 1,\hdots, 2^{n}$, $n\in\Z_+$, of the interval $[S,T]$. Put 
\begin{equation*}
A^{(n)}:=\sum_{i=1}^{2^n-1} A_{t^n_i,t^n_{i+1}},\,n\in\N.
\end{equation*}
Note that the very first element, $A_{t_0^n,t_1^n}$ is not included in the sum as the pair $(t_0^n,t_1^n)=(S,S+2^{-n}(T-S))$ does not belong to $\ms{[S,T]}$. On the other hand, if $i\in[1,2^n-1]$ and $i$ is an integer, then $(t_i^n,t_{i+1}^n)\in \ms{[S,T]}$ and thus $A^{(n)}$ is well-defined.

As in the proof of \cref{T:SSLst}  we deduce that by Fatou's lemma and condition~\eqref{s-limpart}
\begin{equation}\label{s-whatwehaveSSL}
\|\A_T-\A_S\|_{L_m(\Omega)}\le \liminf_{n\to\infty}\| A^{(n)}\|_{L_m(\Omega)}\le \|A_{\frac{S+T}2,T}\|_{L_m(\Omega)}+ \sum_{n=1}^\infty\| A^{(n+1)}-A^{(n)}\|_{L_m(\Omega)}.
\end{equation}
For $n\in\N$ we apply bound \eqref{matebound} with $2^n$ in place of $n$, $X_i:=\delta A_{t^n_{i-1},t^{n+1}_{2i-1},t^n_{i}}$, $\G_i:=\F_{t^n_{i}}$. Using assumptions \eqref{s-con:s1} and \eqref{s-con:s2}, we derive 
\begin{align}\label{threesplit}
&\|A^{(n+1)}-A^{(n)}\|_{L_m(\Omega)}\nn\\
&\quad \le\Bigl\|\sum_{i=2}^{2^n} \delta A_{t^n_{i-1},t^{n+1}_{2i-1},t^n_{i}}\Bigr\|_{L_m(\Omega)}+\|A_{t^{n+1}_1,t^{n+1}_2}\|_{L_m(\Omega)}\nn\\
&\quad\le C\sum_{i=2}^{2^n} \|\E^{t^n_{i-2}}\delta A_{t^n_{i-1},t^{n+1}_{2i-1},t^n_{i}}\|_{L_m(\Omega)}+C\Bigl(\sum_{i=2}^{2^n} \|\delta A_{t^n_{i-1},t^{n+1}_{2i-1},t^n_{i}}\|^2_{L_m(\Omega)}\Bigr)^{1/2}+\|A_{t^{n+1}_1,t^{n+1}_2}\|_{L_m(\Omega)}\nn\\
&\quad\le C \Gamma_2 2^{-n(\gamma_2-1)}|T-S|^{\gamma_2}+ C \Gamma_1 2^{-n(\gamma_1-\frac12)}|T-S|^{\gamma_1}+\Gamma_1 2^{-n\gamma_1}|T-S|^{\gamma_1},
\end{align}
where $C=C(d,m)>0$. To get the first term in the last inequality we used that, by construction, 
${t^n_{i-1}-(t^n_{i}-t^n_{i-1})}=t^n_{i-2}$
and therefore bound \eqref{s-con:s2} is applicable. To get the second term in the last inequality, we used \eqref{s-con:s1} and the simple inequality  
\begin{equation*}
\|\delta A_{s,u,t}\|_{L_m(\Omega)}\le \| A_{s,t}\|_{L_m(\Omega)}+
\|A_{s,u}\|_{L_m(\Omega)}+\| A_{u,t}\|_{L_m(\Omega)}\le 3 \Gamma_1 |t-s|^{\gamma_1}, 
\end{equation*}
valid for any $(s,t)\in\ms{[S,T]}$, $u\in(s,t)$.

Inequality \eqref{threesplit} is where the key improvement of the shifted SSL over the usual SSL occurs. The updated BDG-type bound \eqref{matebound} allows us to insert additional smoothing and replace $\E^{t^n_{i-1}}[\dots]$, which appeared in the proof of the SSL, with $\E^{t^n_{i-2}}[\dots]$ at no cost.

Summing up inequalities \eqref{threesplit} over $n$ and substituting them back into \eqref{s-whatwehaveSSL}, we get 
\begin{equation*}
\|\A_T-\A_S\|_{L_m(\Omega)}\le C \Gamma_2 |T-S|^{\gamma_2}+C\Gamma_1|T-S|^{\gamma_1}
\end{equation*}
for $C=C(\gamma_1,\gamma_2,d,m)>0$.  Here we used once again  \eqref{s-con:s1} to bound $\|A_{\frac{S+T}2,T}\|_{L_m(\Omega)}$ in \eqref{s-whatwehaveSSL}. This shows \eqref{s-est:ssl1}.
\end{proof}

\subsection{Strong well-posedness for SDEs with L\'evy noise}\label{s:levywp}

Now, equipped with the tools developed in the previous two subsections, we are ready to establish the well-posedness of SDEs with irregular drift driven by an $\alpha$-stable process, that is, \cref{t:VKlevy}.
The idea is for $(S,T)\in\Delta_{[0,1]}$ to bound moments of
\begin{equation*}
\A_t:=\int_0^t (b(L_r+\phi_r)-b(L_r+\psi_r))\,dr,\quad t\in[S,T],
\end{equation*}
where $\phi$, $\psi$ are  relatively regular perturbations, in terms of $\|\phi-\psi\|$, that is, to extend \cref{e:step1uni} to the L\'evy case. To achieve this, we change the germ, replace some of the seminorms used in \cref{e:step1uni} by the conditional seminorms introduced in \cref{sec:44}, and we replace SSL by a shifted SSL. For $\alpha\in(1,2)$, this refined approach yields \cref{t:VKlevy} in the full regime  $b\in\C^\gamma$, $\gamma>1-\frac\alpha2$, recall condition \eqref{levycond}.

Let us now discuss what to take as the germ in the shifted SSL to approximate the process $\A$ introduced above. A natural first idea is to proceed as in the proof of \cref{e:step1uni} and set
\begin{equation*}
A_{s,t}:=\E^s\int_s^t (b(L_r+\phi_s)-b(L_r+\psi_s))\,dr,\quad (s,t)\in\ms{[S,T]}.
\end{equation*}
However, if we now recall the discussion in \cref{sec:44}, in order to capture the fact that the stochastic regularity of $\phi$ exceeds $1$ (see \cref{c:413}), we need a better approximation of $\phi_r$ than simply $\phi_s$. At the same time, this approximation must be $\F_s$-measurable, since we require access to the conditional law $\law(L_r+\text{approximation}|\F_s)$. The best candidate is $\E^s\phi_r$. Thus, the second idea would be to set
\begin{equation*}
	A_{s,t}:=\E^s\int_s^t (b(L_r+\E^s\phi_r)-b(L_r+\E^s\psi_r))\,dr,\quad (s,t)\in\ms{[S,T]}.
\end{equation*}
The problem now is that, to make all the integrals in the calculation of $\E^s \delta A_{s,u,t}$ finite, we need some extra smoothing, see the discussion in \cref{r:vd} and at the beginning of \cref{sec:44}. Therefore, we introduce our final choice of the germ:
\begin{equation}\label{astfinal}
	A_{s,t}:=\E^{s-(t-s)}\int_s^t (b(L_r+\E^{s-(t-s)}\phi_r)-b(L_r+\E^{s-(t-s)}\psi_r))\,dr,\quad (s,t)\in\ms{[S,T]}.
\end{equation}
Note that we have also changed our approximation to $\E^{s-(t-s)}\phi_r$ in order to keep access to
$\law(L_r+\E^{s-(t-s)}\phi_r|\F_{s-(t-s)})$.

Now we are ready to prove the following key bound. Recall the definition of the seminorm  $\dnew{\cdot}{\tau}{m}{[S,T]}$  in \eqref{newnorm}.
\begin{lemma}\label{L:32}
Let $m\ge2$, $\alpha\in(1,2)$, $\gamma\in(1-\frac\alpha2,1]$, $\tau\in(0,1]$, $\theta>0$. Let $b\in\C^\gamma(\R^d,\R)$. Let $\phi,\psi\colon\Omega\times[0,1]\to\R^d$ be bounded measurable functions adapted to $(\F_t)_{t\in[0,1]}$. Assume that 
\begin{equation}\label{l-gammatau}
\tau+\frac{\gamma-1}{\alpha}>0,\quad \theta+\frac{\gamma-2}{\alpha}>0.
\end{equation}
Suppose further that
there exists a constant $\Gamma_0\ge0$ such that for any $(s,t)\in\Delta_{[0,1]}$
\begin{equation}\label{l-gamma0cond}
\|\E^s|\phi_t-\E^s\phi_t|\|_{L_\infty(\Omega)}\le \Gamma_0|t-s|^\theta.
\end{equation}
Then  there exists a constant $C=C(\alpha,\gamma,\theta,\tau,d,m)>0$ such that 
for any  $(S,T)\in\Delta_{[0,1]}$ one has
\begin{align}\label{l-seckey}
&\Bigl \| \int_S^T \bigl(b(L_r+\phi_r)-b(L_r+\psi_r)\bigr)\,dr \Bigr\|_{L_m(\Omega)}\nn\\
&\,\,\le C\|b\|_{\C^\gamma}(T-S)^{1+\frac{\gamma-1}\alpha}\bigl((1+\Gamma_0)\|\phi-\psi\|_{\C^0L_m([S,T])}
+ \dnew{\phi-\psi}{\tau}{m}{[S,T]} (T-S)^{\tau}\bigr).
\end{align}
\end{lemma}

\begin{remark}
We see that compared with \cref{e:step1uni}, regularity condition \eqref{gamma0cond} is now replaced by stochastic regularity condition \eqref{l-gamma0cond}.
\end{remark}
\begin{proof}
Fix $m\ge2$, $(S,T)\in\Delta_{[0,1]}$. Let us verify that all the conditions of the shifted SSL (\cref{lem:shiftedmodifiedSSL}) are satisfied for the process
\begin{equation*}
\A_{t}:=\int_0^t (b(L_r+\phi_r)-b(L_{r}+\psi_r))\,dr,\qquad t\in[S,T]
\end{equation*}
and the germ $A_{s,t}$ defined in \eqref{astfinal}. 

Let $(s,t)\in\ms{[S,T]}$. Recall identity \eqref{l-1bound} and the definition of the semigroup $\Pa$ in \cref{s:levyint}. We get
\begin{align*}
|A_{s,t}|&\le \int_s^t |\Pa_{r-(s-(t-s))} b(L_{s-(t-s)}+\E^{s-(t-s)}\phi_r)-\Pa_{r-(s-(t-s))} b(L_{s-(t-s)}+\E^{s-(t-s)}\psi_r)|\,dr\\
	& \le\int_s^t \|\Pa_{r-(s-(t-s))} b\|_{\C^1}|\E^{s-(t-s)}(\phi_r-\psi_r)|\,dr.
\end{align*}
Thus, by \eqref{ineqlevy} and Jensen's inequality, we have
\begin{align}\label{astlevy}
	\|A_{s,t}\|_{L_m(\Omega)}& \le C\|b\|_{\C^\gamma}\int_s^t (r-s)^{\frac{\gamma-1}{\alpha}}\|\E^{s-(t-s)}(\phi_r-\psi_r)\|_{L_m(\Omega)}\,dr\nn\\
	& \le C\|b\|_{\C^\gamma} (t-s)^{1+\frac{\gamma-1}{\alpha}} \sup_{r\in[S,T]}\|\phi_r-\psi_r\|_{L_m(\Omega)}\nn\\
	& = C \|b\|_{\C^\gamma} \|\phi-\psi\|_{\C^0L_m([S,T])}(t-s)^{1+\frac{\gamma-1}{\alpha}}
\end{align}
for $C=C(\alpha,\gamma,d)>0$. Note that by the assumption $\gamma>1-\alpha/2$, we have $1+\frac{\gamma-1}{\alpha}>1/2$ and therefore, condition \eqref{s-con:s1} is satisfied.

Now let us verify condition \eqref{s-con:s2}. As required, we take $(s,t)\in\ms{[S,T]}$, and $u:=(t+s)/2$. It will be convenient to denote $s_1:=s-(t-s)$, $s_2:=s-(u-s)$, $s_3:=s$, $s_4:=u$, $s_5:=t$. One has $s_1\le s_2\le s_3\le s_4\le s_5$, see \cref{fig:arrow-ticks}. We have 
\begin{figure}
	\centering
	\includegraphics[width=0.99\textwidth, trim=0 35 0 35, clip]{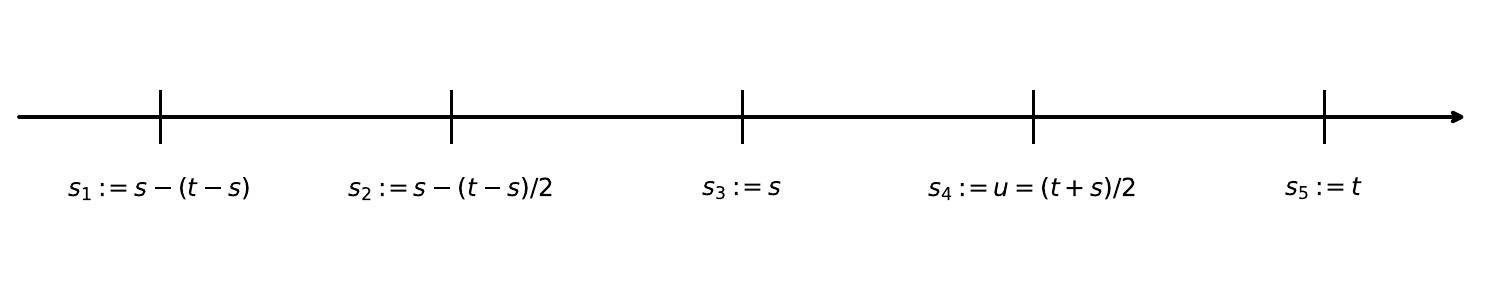}
	\caption{Time points for bounding $\E^{s-(t-s)}\delta A_{s,u,t}$}
	\label{fig:arrow-ticks}
\end{figure} 
\begin{align*}\delta A_{s,u,t}&=\delta A_{s_3,s_4,s_5}\nn\\
&=\E^{s_1}\int_{s_3}^{s_5}(b(L_r+\E^{s_1}\phi_r)-b(L_{r}+\E^{s_1}\psi_r))\,dr\\ 
&\phantom{=}-\E^{s_2}\int_{s_3}^{s_4}(b(L_r+\E^{s_2}\phi_r)-b(L_{r}+\E^{s_2}\psi_r))\,dr\\
&\phantom{=}-\E^{s_3}\int_{s_4}^{s_5}(b(L_r+\E^{s_3}\phi_r)-b(L_{r}+\E^{s_3}\psi_r))\,dr,
\end{align*}
where we used the identity $u-(t-u)=s=s_3$.
Therefore,
\begin{align}\label{step1levy}
		&\E^{s-(t-s)}\delta A_{s,u,t}\nn\\
		&\quad=\E^{s_1}\delta A_{s_3,s_4,s_5}\nn\\
		&\quad=\E^{s_1}\E^{s_2}\int_{s_3}^{s_4}(b(L_r+\E^{s_1}\phi_r)-b(L_{r}+\E^{s_1}\psi_r)-b(L_r+\E^{s_2}\phi_r)+b(L_{r}+\E^{s_2}\psi_r))\,dr\nn\\
		&\qquad+\E^{s_1}\E^{s_3} \int_{s_4}^{s_5}(b(L_r+\E^{s_1}\phi_r)-b(L_{r}+\E^{s_1}\psi_r)-b(L_r+\E^{s_3}\phi_r)+b(L_{r}+\E^{s_3}\psi_r))\,dr\nn\\
		&\quad=:I_1+I_2.
	\end{align}
We begin with the analysis of $I_1$. This is where the benefits of the shifted sewing and conditional norms become apparent. Please compare it with the similar analysis we carried out in the Gaussian case in \eqref{step1}--\eqref{cond2pert}. Using the identity \eqref{l-1bound} and \cref{l:l29}, we derive
\begin{align*}
&|I_1|\le \E^{s_1}\int_{s_3}^{s_4}\bigl|\Pa_{r-s_2}b(L_{s_2}+\E^{s_1}\phi_r)-\Pa_{r-s_2}b(L_{s_2}+\E^{s_1}\psi_r)\nn\\
&\hskip13ex-\Pa_{r-s_2}b(L_{s_2}+\E^{s_2}\phi_r)+
\Pa_{r-s_2} b(L_{s_2}+\E^{s_2}\psi_r)\bigr|\,dr\nn\\
&\quad\le \int_{s_3}^{s_4} \|\nabla\Pa_{r-s_2}b\|_{\C^1}\bigl|\E^{s_1}(\phi_r-\psi_r)\bigr|\,\E^{s_1}[|\E^{s_1}\phi_r-\E^{s_2}\phi_r|]\,dr\nn\\
&\qquad+\int_{s_3}^{s_4} \|\Pa_{r-s_2}b\|_{\C^1}\E^{s_1}|\E^{s_1}(\phi_r-\psi_r)-\E^{s_2}(\phi_r-\psi_r)|\,dr.
\end{align*}
Now we can use stochastic regularity of $\phi$ and Jensen's inequality to get
\begin{align*}
\E^{s_1}|\E^{s_1}\phi_r-\E^{s_2}\phi_r|&=\E^{s_1} |\E^{s_2}(\E^{s_1}\phi_r-\phi_r)|\le 
\E^{s_1}\E^{s_2} |\E^{s_1}\phi_r-\phi_r|=\E^{s_1}|\E^{s_1}\phi_r-\phi_r|\\
&\le \Gamma_0 |r-s_1|^\theta,
\end{align*}
recall assumption \eqref{l-gamma0cond}. Very similarly,
\begin{equation*}
\E^{s_1}|\E^{s_1}(\phi_r-\psi_r)-\E^{s_2}(\phi_r-\psi_r)|\le \E^{s_1}|\E^{s_1}(\phi_r-\psi_r)-(\phi_r-\psi_r)|.
\end{equation*}
Therefore, using the heat kernel bounds \eqref{ineqlevy}, we get
\begin{align}\label{I1fin}
\|I_1\|_{L_m(\Omega)}&\le C\Gamma_0\|b\|_{\C^\gamma}\int_{s_3}^{s_4}(r-s_2)^{\frac{\gamma-2}\alpha}\bigl\|\E^{s_1}(\phi_r-\psi_r)\bigr\|_{L_m(\Omega)}(r-s_1)^\theta\,dr\nn\\
&\phantom{\le}+C\|b\|_{\C^\gamma}\int_{s_3}^{s_4} (r-s_2)^{\frac{\gamma-1}\alpha}\|\E^{s_1}(\phi_r-\psi_r)-(\phi_r-\psi_r)\|_{L_m(\Omega)}\,dr\nn\\ 
&\le C\Gamma_0\|b\|_{\C^\gamma}\|\phi-\psi\|_{\C^0 L_m([S,T])}(s_3-s_2)^{\frac{\gamma-2}\alpha}(s_4-s_3)(s_4-s_1)^{\theta}\nn\\
&\phantom{\le}+C\|b\|_{\C^\gamma}\dnew{\phi-\psi}{\tau}{m}{[S,T]}(s_3-s_2)^{\frac{\gamma-1}\alpha}(s_4-s_3)(s_4-s_1)^{\tau}\nn\\
&\le C\Gamma_0\|b\|_{\C^\gamma}\|\phi-\psi\|_{\C^0 L_m([S,T])}(t-s)^{\frac{\gamma-2}{\alpha}+1+\theta}\nn\\
&\phantom{\le}+ C\|b\|_{\C^\gamma}\dnew{\phi-\psi}{\tau}{m}{[S,T]}(t-s)^{\tau+\frac{\gamma-1}{\alpha}+1},
\end{align}
for $C=C(\alpha,\gamma,\theta, d,m)>0$, where the last inequality follows from the fact $s_3-s_2=s_4-s_3=(t-s)/2$ and $s_4-s_1=(u-s)+(t-s)=\frac32 (t-s)$. Note that the integral in \eqref{I1fin} is finite even though $\frac{\gamma-2}{\alpha}$ may be less than $-1$. This is where the shifted stochastic sewing played a crucial role; otherwise, the integrand would be $(r - s_3)^{\frac{\gamma - 2}{\alpha}}$, and the integral would diverge.

By exactly the same argument (we just need to take $s_3$ in place of $s_2$, $s_4$ in place of $s_3$ and $s_5$ in place of $s_4$), we get
the exact same bound for $I_2$, and then by \eqref{step1levy}, for $\E^{s-(t-s)}\delta A_{s,u,t}$ as well:
\begin{align}\label{levydasut}
\|\E^{s-(t-s)}\delta A_{s,u,t}\|_{L_m(\Omega)}&\le C\Gamma_0\|b\|_{\C^\gamma}\|\phi-\psi\|_{\C^0 L_m([S,T])}(t-s)^{\frac{\gamma-2}{\alpha}+1+\theta}\nn\\
	&\phantom{\le}+ C\|b\|_{\C^\gamma}\dnew{\phi-\psi}{\tau}{m}{[S,T]}(t-s)^{\tau+\frac{\gamma-1}{\alpha}+1}.
\end{align}
Since by the assumptions of the lemma $\frac{\gamma-2}{\alpha}+1+\theta>1$ and $\tau+\frac{\gamma-1}{\alpha}+1>1$, we see that condition \eqref{s-con:s2} is satisfied.

Next, it is immediate that condition (ii) of the shifted SSL is satisfied and verification of  \eqref{s-limpart} is done along the same lines as in the proof of \cref{e:step1uni} and is left as an exercise to the reader. 	

Thus, all the conditions of \cref{lem:shiftedmodifiedSSL} are   satisfied, and bound \eqref{l-seckey} follows  from~\eqref{astlevy} and \eqref{levydasut}, using also the inequality
\begin{equation*}
(T-S)^{\frac{\gamma-2}{\alpha}+1+\theta}\le 
(T-S)^{1+\frac{\gamma-1}{\alpha}},
\end{equation*}
valid since $1+\frac{\gamma-1}{\alpha}$ is less than $1$, while $\frac{\gamma-2}{\alpha}+1+\theta$ is larger than $1$ by \eqref{l-gammatau}.
\end{proof}

Now we can transform \cref{L:32} into a stability bound for solutions of the SDE~\eqref{mainSDElevy}. The idea is to bound the difference between the two solutions in the $[\cdot]_{\C^{1/2}L_2([S,T])}$ seminorm by half of itself plus the difference between the two drifts. First, we do this for the case when the interval $[S,T]$ is small enough, and then we use \cref{l:ltg} to extend the bound to the whole interval.

\begin{corollary}\label{c:stab}
Let $\alpha\in(1,2)$, $\gamma\in(1-\frac\alpha2,1]$. 
Let $b^1,b^2\in\C^\gamma(\R^d,\R^d)$, $x\in\R^d$.  Let $X^1$, $X^2$ be two solutions on the same filtered  probability space $(\Omega,\F,(\F_t),\P)$  to the following SDEs: 
\begin{equation*}
X_t^1=x+\int_0^t b^1(X_r^1)\,dr+L_t,\quad X_t^2=x+\int_0^t b^2(X_r^2)\,dr+L_t,\quad t\in[0,1].
\end{equation*}	
Then there exists a constant $C=C(\alpha,\gamma,d,\|b^1\|_{\C^\gamma})>0$
such that
\begin{equation}\label{stabineqfin}
\|X^1-X^2\|_{\C^{0}L_2([0,1])} \le C\|b^1-b^2\|_{\C^0}.
\end{equation}
\end{corollary}
\begin{proof}
Denote $\phi:=X^1-L$, $\psi:=X^2-L$, $M:=\|b^1\|_{\C^\gamma}$.  

\textbf{Step 1: Estimate on short time intervals}. 
It follows from \cref{c:413} that there exists a constant $C=C(\alpha,d)>0$ such that 
\begin{equation}\label{stochreg}
	\E^s|\phi_t-\E^s\phi_t|\le C M(1+M)(t-s)^{1+\frac\gamma\alpha},\quad (s,t)\in\Delta_{[0,1]}.
\end{equation}
We apply now \cref{L:32} with $\theta= 1+\frac\gamma\alpha$, $\tau=\frac12$, $m=2$. It is easy to see that condition \eqref{l-gammatau} holds by assumption $\gamma>1-\alpha/2$ and condition \eqref{l-gamma0cond} holds thanks to the stochastic regularity result \eqref{stochreg}. Therefore, it follows from \eqref{l-seckey} that there exists $C=C(\alpha,\gamma,d,M)>0$ such that for any  $(s,t)\in\Delta_{[0,1]}$ one has
\begin{align*}
&\|(\phi_t-\phi_s)-(\psi_t-\psi_s)\|_{L_2(\Omega)}\\
&\quad\le\Bigl \| \int_s^t \bigl(b^1(L_r+\phi_r)-b^1(L_r+\psi_r)\bigr)\,dr \Bigr\|_{L_2(\Omega)}+
 \Bigl\| \int_s^t (b^1-b^2)(L_r+\psi_r)\,dr \Bigr\|_{L_2(\Omega)}
\nn\\
&\quad\le C (t-s)^{1+\frac{\gamma-1}\alpha}\bigl(\|\phi-\psi\|_{\C^0L_2([s,t])}+ [\phi-\psi]_{\C^{\frac12}L_2([s,t])} (t-s)^{\frac12}\bigr)+(t-s)\|b^1-b^2\|_{\C^0}.
\end{align*}
Note that here we just used the rough bound  $\dnew{\cdot}{\frac12}{2}{[s,t]}\le2[\cdot]_{\C^\frac12 L_2([s,t])}$ from \cref{c:normbound}.  However, later in \cref{s:levynum}, when deriving bounds for the convergence rate of the Euler scheme, we will need a more precise bound involving $\dnew{\cdot}{\frac12}{2}{[s,t]}$. 

Let $(S,T)\in\Delta_{[0,1]}$. We divide both sides of the above inequality by $(t-s)^{1/2}$ and take supremum over all $(s,t)\in\Delta_{[S,T]}$. Note that the condition $\gamma>1-\alpha/2$ implies $1+\frac{\gamma-1}\alpha>\frac12$. Therefore, we get 
\begin{align}\label{phipsistab}
[\phi-\psi]_{\C^{\frac12}L_2([S,T])}&\le  C \|\phi-\psi\|_{\C^0L_2([S,T])}+ C(T-S)^{\frac12}[\phi-\psi]_{\C^{\frac12}L_2([S,T])}+\|b^1-b^2\|_{\C^0}\nn\\
	&\le C_0\|\phi_S\!-\!\psi_S\|_{L_2(\Omega)}+ C_0 (T\!-\!S)^{\frac12} [\phi-\psi]_{\C^{\frac12}L_2([S,T])}+\|b^1\!-b^2\|_{\C^0}.
\end{align}
where $C_0=C_0(\alpha,\gamma,d,M)>0$ and in the last inequality we used \eqref{norm2normgen}.
Take now $\ell\in(0,1]$ small enough such that
\begin{equation*}
 C_0 \ell^{\frac12}\le \frac12.
\end{equation*}
Then we derive from \eqref{phipsistab} that for any $(S,T)\in\Delta_{[0,1]}$ with $T-S\le \ell$ we have 
\begin{equation}\label{sboundstab}
[\phi-\psi]_{\C^{\frac12}L_2([S,T])}\le 2 C_0\|\phi_S-\psi_S\|_{L_2(\Omega)}+ 2\|b^1-b^2\|_{\C^0}.
\end{equation}
\textbf{Step 2: Passage to arbitrary time intervals}.
We apply \cref{l:ltg}(ii) to $f:=\phi-\psi$. 
It follows from \eqref{sboundstab} that condition \eqref{genholderm} holds and therefore  there exists $C=C(\ell,C_0)$  such that 
\begin{equation*}
	[\phi-\psi]_{\C^{\frac12} L_2([0,1])} \le C\|b^1-b^2\|_{\C^0}.
\end{equation*}
Since $\phi_0=\psi_0$ and $\phi-\psi=X^1-X^2$, we get the desired bound \eqref{stabineqfin}.
\end{proof}

The obtained stability bound implies, almost immediately, existence and uniqueness of solutions to the SDE~\eqref{mainSDElevy}. We note that, in contrast to the proof of the existence part of \cref{t:VK}, we can no longer use Girsanov's theorem to obtain weak existence ``for free'', so we have to establish strong existence directly.

\begin{proof}[Proof of \cref{t:VKlevy}]
	
\textbf{Pathwise uniqueness}.
Let $X,Y$ be two solutions to SDE \eqref{mainSDElevy} on the same filtered  probability space $(\Omega,\F,(\F_t),\P)$. By \cref{c:stab}, we have 
\begin{equation*}
	\|X-Y\|_{\C^0L_2([0,1])}=0.
\end{equation*}
Therefore, $X_t=Y_t$ a.s. for all $t\in[0,1]$.

\textbf{Strong existence}.
Let $(b^n)_{n\in\N}$ be a sequence of smooth functions $\R^d\to\R^d$ converging to $b$ in supremum norm, such that $\|b^n\|_{\C^\gamma}\le 
\|b\|_{\C^\gamma}$ (for example, one can take $b_n:=P_{1/n}b$). Consider an SDE
\begin{equation}\label{nsde}
X^n_t=x+\int_0^t b^n(X^n_r)\,dr +L_t,\quad t\in[0,1].
\end{equation}
Since $b^n$ is smooth, this equation has a strong solution. Thus, $X^n$ is adapted to the completed natural filtration of the Lévy process, $\F^L_t:=\sigma(L_s,\mathcal{N}; s\le t)$, where $\mathcal{N}$ denotes all $\P$-null sets.

Let $\phi^n:=X^n-L$. By \cref{c:stab}, there exists a constant $C=C(\alpha,\gamma,d,\|b\|_{\C^\gamma})>0$ such that for any $n,m\in\N$, $t\in[0,1]$ we have 
\begin{equation*}
\|\phi^n_t-\phi^m_t\|_{L_2(\Omega)} \le C\|b^n-b^m\|_{\C^0}.
\end{equation*}
Since $\|b^n-b^m\|_{\C^0}\to0$ as $n,m\to\infty$, we see that there exists a process $\phi\colon\Omega\times[0,1]\to\R^d$ adapted to the filtration $(\F^L_t)_{t\in[0,1]}$ such that
\begin{equation}\label{limphistr}
\lim_{n\to\infty}\|\phi_t-\phi^n_t\|_{L_2(\Omega)}=0.
\end{equation}
Furthermore, by Fatou's lemma, for any $(s,t)\in\Delta_{[0,1]}$ we have 
\begin{equation*}
\|\phi_t-\phi_s\|_{L_2(\Omega)}\le 	\lim_{n\to\infty}\|\phi^n_t-\phi^n_s\|_{L_2(\Omega)}\le |t-s|\|b\|_{\C^0}.
\end{equation*}
Therefore, by the Kolmogorov continuity theorem (recall \cref{t:kolmi}), we see that the process $\phi$ has a continuous modification, which we will denote $\wt\phi$. We claim now that $X:=\wt\phi+L$ is a strong solution to SDE \eqref{mainSDElevy}.

Indeed, $X$ is c\`adl\`ag by construction and $X$ is adapted to $(\F^L_t)$. Note also that for any $t\in[0,1]$
\begin{align*}
\E \Bigl|\int_0^t(b(X_r)- b^n(X_r^n))dr\Bigr|&\le \E \Bigl|\int_0^t (b-b^n)(X^n_r)dr\Bigr|+
\E\Bigl|\int_0^t(b(X_r)- b(X_r^n))dr\Bigr|\\
&\le \|b-b^n\|_{\C^0} +2\|b\|_{\C^\gamma}\int_0^t ((\E |X_r- X_r^n|^\gamma) \wedge1 )dr\to 0,
\end{align*}
as $n\to\infty$ by the dominated convergence theorem (we used that $X-X^n=\wt\phi-\phi^n$ and \eqref{limphistr}). Thus, we can pass to the limit in \eqref{nsde} as $n\to\infty$ to deduce that for any $t\in[0,1]$ we have a.s.
\begin{equation*}
X_t=x+\int_0^t b(X_r)\,dr +L_t.
\end{equation*}
Hence, $X$ is indeed a strong solution to \eqref{mainSDElevy}.
\end{proof}

\subsection{Numerics for SDEs with L\'evy noise}\label{s:levynum}

Now, after lots of  hard work, we can finally harvest the results and obtain the rate of convergence of the Euler scheme to the solution of an SDE driven by a L\'evy process.
We fix $\alpha\in(1,2)$, $x\in\R^d$, $b\in\C^\gamma(\R^d,\R^d)$, $\gamma\in(1-\frac\alpha2,1]$. By \cref{t:VKlevy}, under these assumptions, SDE \eqref{mainSDElevy} has a solution which we denote by $X$. For $n\in\N$, let $X^n$ be the corresponding Euler  approximation:
\begin{equation*}
X^n_t:=x+\int_0^t b(X^n_{\kappa_n(r)})\,dr +L_t.
\end{equation*}
Denote
\begin{equation*}
	\phi:=X-L,\quad \phi^n=X^n-L.
\end{equation*}
As  in \cref{s:numbd}, we aim to bound the approximation error 
\begin{equation}\label{whwbl}
	[\phi-\phi^n]_{\C^{1/2}L_2([S,T])}
\end{equation}
by half of itself plus terms of the form $C n^{-\rrr}$. It is sufficient to do this only for small intervals $[S,T]$, since the global bound then follows from \cref{l:ltg}. As a first attempt, we use the same decomposition for the error \eqref{whwbl} as in the Brownian case, see \cref{s:numbd}. We write  for $(s,t)\in\Delta_{[0,1]}$ 
\begin{align}
(\phi_t-\phi^n_t)-(\phi_s-\phi^n_s)&=\int_s^t (b(X_r)-b(X^n_r))\,dr+\int_s^t (b(X^n_{r})-b(X^n_{\kappa_n(r)}))\,dr
	\nn\\
	&=\int_s^t (b(L_r+\phi_r)-b(L_r+\phi^n_{r}))\,dr
	\nn\\
	&\phantom{=}+\int_s^t (b(L_r+\phi^n_{r})-b(L_{\kappa_n(r)}+\phi^n_{\kappa_n(r)}))\,dr\nn\\
	&=:\err_{n,1}(s,t)+\err_{n,2}(s,t).\label{eq:edlevy}	
\end{align}
Note however, that  in the Brownian case a bound on $\int (b(W_r)-b(W_{\kappa_n(r)}))\,dr$ implied almost immediately a bound on $\int (b(W_r+\phi^n_{r})-b(W_{\kappa_n(r)}+\phi^n_{\kappa_n(r)}))\,dr$ by Girsanov's theorem. This tool is clearly not available here. Therefore, we have to change our error decomposition and split the second error term, $\err_{n,2}$, into two. We have
\begin{align}
\err_{n,2}(s,t)&=\int_s^t (b(L_r+\phi^n_{r})-b(L_r+\phi^n_{\kappa_n(r)}))\,dr\nn\\
&\phantom{=}+\int_s^t (b(L_r+\phi^n_{\kappa_n(r)})-b(L_{\kappa_n(r)}+\phi^n_{\kappa_n(r)}))\,dr\nn\\
	&=:\err_{n,21}(s,t)+\err_{n,22}(s,t).\label{eq:edlevy2}	
\end{align}

Now we observe that our main integral bound for L\'evy processes, \cref{L:32}, can handle the first two error terms: $\err_{n,1}$ and $\err_{n,21}$. We only need to check that the drifts $\phi^n$ and $t\mapsto \phi^n_{\kappa_n(t)}$ have enough stochastic regularity, so that conditions of  \cref{L:32} are satisfied. This is not very difficult to check. We argue as in \cref{c:413}, but with one important modification: it is no longer true that $\E^s|L_{\kappa_n(r)}-L_s|$ can be bounded by $C|r-s|^{1/\alpha}$ whenever $s \le r$. Indeed, it may happen that $\kappa_n(r)<s \le r$, in which case $L_{\kappa_n(r)} - L_s$ is no longer independent of $\F_s$. Therefore, we have to treat this case separately.

\begin{lemma}\label{l:stochreg}
There exists a constant $C=C(\alpha,d)>0$ such that for any $n\in\N$, $(s,t)\in\Delta_{[0,1]}$
	\begin{align}\label{eq:uslevy}
		&\E^s|\phi^n_t-\E^s\phi^n_t|\le C \|b\|_{\C^\gamma}(1+\|b\|_{\C^0})(t-s)^{1+\frac\gamma\alpha};\\
		&\E^s|\phi^n_{\kappa_n(t)}-\E^s\phi^n_{\kappa_n(t)}|\le C \|b\|_{\C^\gamma}(1+\|b\|_{\C^0})(t-s)^{1+\frac\gamma\alpha}.		\label{eq:uslevy2}
	\end{align}
\end{lemma}
\begin{proof}
Let $(s,t)\in\Delta_{[0,1]}$. Denote by $s'$  the smallest grid point which is greater than or equal to s, that is, $s':=\frac1n\lceil ns\rceil$. If $s\le t\le s'$, then by construction $\phi_t^n$ is $\F_{s}$-measurable and thus the left-hand side of \eqref{eq:uslevy} is zero. Otherwise, if $s\le s'<t$, we rely again on \cref{lem:useful-lemma}. We note that the random variable $\phi^n_{s'}+(t-s') b(X^n_s)$ is $\F_s$ measurable and therefore 
	\begin{align*}
		\E^s|\phi^n_t-\E^s \phi^n_t|&\le 2\E^s |\phi^n_t-\phi^n_{s'}-(t-s')b(X^n_s)|=2 \E^s\Bigl|\int_{s'}^t (b(X^n_{\kappa_n(r)})-b(X^n_s))\,dr\Bigr|\\
		&\le 2 \|b\|_{\C^\gamma}\int_{s'}^t (\E^s |L_{\kappa_n(r)}-L_s|^\gamma+\E^s |\phi^n_{\kappa_n(r)}-\phi^n_s|^\gamma)\,dr\\
		&\le C 
		\|b\|_{\C^\gamma}(1+\|b\|_{\C^0})|t-s|^{1+\frac\gamma\alpha},
	\end{align*}
	for $C=C(\alpha,d)>0$ and we used that $|\phi^n_{\kappa_n(r)}-\phi^n_s|\le \|b\|_{\C^0}|\kappa_n(r)-s|\le \|b\|_{\C^0}|r-s|$ and $L_{\kappa_n(r)}-L_s$ is independent of $\F_s$. This proves \eqref{eq:uslevy}.
	
To show \eqref{eq:uslevy2}, we just note that if $\kappa_n(t)\le s$, then the left-hand side of this inequality is $0$ and the bound is trivial. Otherwise, we just apply \eqref{eq:uslevy} with $\kappa_n(t)$ in place of $t$. We get 
\begin{equation*}
\E^s|\phi^n_{\kappa_n(t)}-\E^s\phi^n_{\kappa_n(t)}|\le C \|b\|_{\C^\gamma}(1+\|b\|_{\C^0})(\kappa_n(t)-s)^{1+\frac\gamma\alpha}\le  C \|b\|_{\C^\gamma}(1+\|b\|_{\C^0})(t-s)^{1+\frac\gamma\alpha}.		\qedhere
\end{equation*}
\end{proof}

Now we have all the tools to  bound $\err_{n,1}$ and $\err_{n,21}$, defined in \eqref{eq:edlevy}	 and \eqref{eq:edlevy2} by applying \cref{L:32}. 

\begin{corollary}\label{l:2terms}
Let  $\eps>0$.
	There exists a constant $C={C(\|b\|_{\C^\gamma},\alpha,\gamma,\eps, d)>0}$ such that for any $n\in\N$, $(s,t)\in\Delta_{[0,1]}$
	\begin{align}\label{errterm1}
		&\|\err_{n,1}(s,t)\|_{L_2(\Omega)}\le C\|\phi-\phi^n\|_{\C^0L_2([s,t])}(t-s)^{1+\frac{\gamma-1}\alpha}
		+ C[\phi-\phi^n]_{\C^\frac12 L_2([s,t])} (t-s)^{\frac32+\frac{\gamma-1}\alpha}
		;\\
		&\|\err_{n,21}(s,t)\|_{L_2(\Omega)} \le C n^{-1}(t-s)^{1+\frac{\gamma-1}\alpha}
		+ C n^{-\frac12-(\frac\gamma\alpha\wedge\frac12)+\eps} (t-s)^{\frac32+\frac{\gamma-1}\alpha}.		\label{errterm2}
	\end{align}
\end{corollary}

\begin{proof}
We apply \cref{L:32} with $m=2$, $\tau=\frac12$, $\theta=1+\frac\gamma\alpha$.  We see that condition \eqref{l-gammatau} is satisfied by our standing assumption $\gamma>1-\frac\alpha2$.
	
First we take $\phi$ as before and $\psi:=\phi^n$.	Then by \cref{l:stochreg} the stochastic regularity condition \eqref{l-gamma0cond} is satisfied and we derive from \eqref{l-seckey} for any $(s,t)\in\Delta_{[0,1]}$
\begin{equation*}
\|\err_{n,1}(s,t)\|_{L_2(\Omega)}\le C\|\phi-\phi^n\|_{\C^0L_2([s,t])}(t-s)^{1+\frac{\gamma-1}\alpha}
+ C[\phi-\phi^n]_{\C^\frac12 L_2([s,t])} (t-s)^{\frac32+\frac{\gamma-1}\alpha}
\end{equation*}
for $C=C(\|b\|_{\C^\gamma},\alpha,\gamma,d)>0$, which is \eqref{errterm1}. Here we used again that $[\cdot]^\mathrm{cond}\le[\cdot]$, recall \cref{c:normbound}. 

Next, we take  in \cref{L:32} $\phi^n$ in place of $\phi$ and $\psi(t):=\phi^n_{\kappa_n(t)}$ and the same parameters $m, \tau,\theta$ as above. \cref{l:stochreg} guarantees that condition \eqref{l-gamma0cond} holds and, very similarly, we get from  \eqref{l-seckey} that for any $(s,t)\in\Delta_{[0,1]}$
\begin{equation}\label{err21}
	\|\err_{n,21}(s,t)\|_{L_2(\Omega)}\le C\|\phi^n-\phi^n_{\kappa_n(\cdot)}\|_{\C^0L_2([s,t])}(t-s)^{1+\frac{\gamma-1}\alpha}
	+ C\dnew{\phi^n-\phi^n_{\kappa_n(\cdot)}}{\frac12}{2}{[s,t]} (t-s)^{\frac32+\frac{\gamma-1}\alpha}
\end{equation}
for $C=C(\|b\|_{\C^\gamma},\alpha,\gamma,d)>0$. Now let us work with the norms and seminorms on the right-hand side of the above inequality. By definition, we have 
\begin{equation*}
|\phi^n_r-\phi^n_{\kappa_n(r)}|=(r-\kappa_n(r))|b(X^n_{\kappa_n(r)})|\le \frac1n\|b\|_{\C^0},\quad r\in[0,1].
\end{equation*}
Therefore 
\begin{equation}\label{norm1num}
\|\phi^n-\phi^n_{\kappa_n(\cdot)}\|_{\C^0L_2([s,t])}\le \frac1n\|b\|_{\C^0}.
\end{equation}

Let us now bound $\dnew{\phi^n-\phi^n_{\kappa_n(\cdot)}}{\frac12}{2}{[s,t]}$.  This is exactly the point where it becomes clear why we needed the sharper bound in \cref{L:32}, involving the conditional seminorm $\dnew{\cdot}{\tau}{2}{[s,t]}$, rather than a cruder bound involving the  H\"older seminorm $[\cdot]_{\C^\tau L_2([s,t])}$. This is the only place in the entire proof where the cruder bound would not suffice.

\begin{figure}
	\centering
	\includegraphics[width=0.99\textwidth, trim=0 35 0 35, clip]{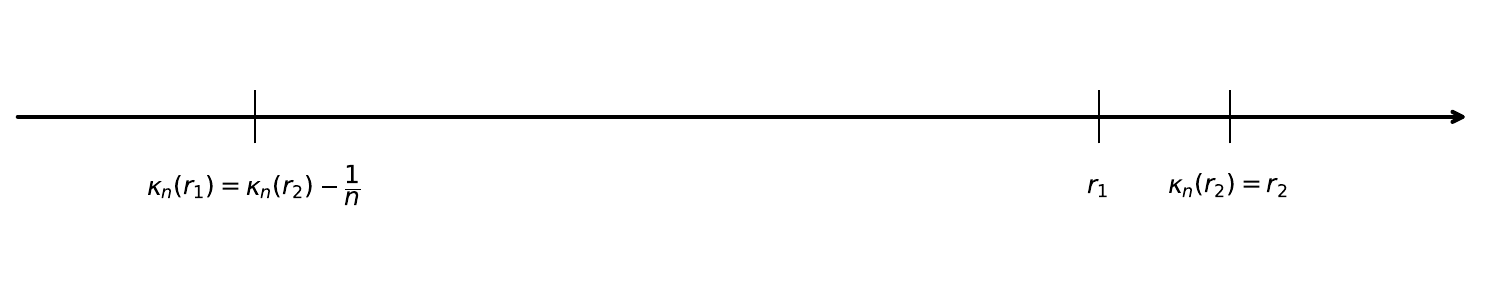}
	\caption{Pathological configuration:  as $r_1\nearrow r_2$, the increment of $\phi^n-\phi^n_{\kappa_n(\cdot)}$ from $r_1$ to $r_2$, divided by $(r_2-r_1)^\tau$, blows up for any $\tau>0$. Thus $[\phi^n-\phi^n_{\kappa_n(\cdot)}]_{\C^\frac12 L_2([s,t])}=\infty$, while 
$\dnew{\phi^n-\phi^n_{\kappa_n(\cdot)}}{\frac12}{2}{[s,t]}<\infty$.}
	\label{f:points}
\end{figure}

Indeed, if 
$(r_1,r_2)\in\Delta_{[s,t]}$ are such that $\kappa_n(r_1)=\kappa_n(r_2)-\frac1n<r_1<\kappa_n(r_2)=r_2$, see \cref{f:points}, then 
\begin{equation*}
|(\phi^n_{r_2}-\phi^n_{\kappa_n(r_2)})-(\phi^n_{r_1}-\phi^n_{\kappa_n(r_1)})|=(r_1-\kappa_n(r_1))|b(X^n_{\kappa_n(r_1)})|=(r_1-r_2+\frac1n)|b(X^n_{\kappa_n(r_1)})|
\end{equation*}
and therefore taking $r_1\nearrow r_2$, we see that
\begin{equation}\label{bad}
[\phi^n-\phi^n_{\kappa_n(\cdot)}]_{\C^\tau L_2([s,t])}=\infty
\end{equation}
for any $\tau>0$. On the other hand, we have for any $(r_1,r_2)\in\Delta_{[s,t]}$
\begin{equation}\label{phidif}
\|(\phi^n_{r_2}-\phi^n_{\kappa_n(r_2)})-\E^{r_1}(\phi^n_{r_2}-\phi^n_{\kappa_n(r_2)})\|_{L_2(\Omega)}=
(r_2-\kappa_n(r_2))\|b(X^n_{\kappa_n(r_2)})-\E^{r_1}b(X^n_{\kappa_n(r_2)})\|_{L_2(\Omega)}.
\end{equation} 
If $\kappa_n(r_2)\le r_1$ then the left-hand side of the above identity is zero. Otherwise, if 
$r_1<\kappa_n(r_2)\le r_2$, then by \cref{lem:useful-lemma},
\begin{align*}
\|b(X^n_{\kappa_n(r_2)})-\E^{r_1}b(X^n_{\kappa_n(r_2)})\|_{L_2(\Omega)}&\le 2
\|b(X^n_{\kappa_n(r_2)})-b(X^n_{r_1})\|_{L_2(\Omega)}\\
&\le C \| |L_{\kappa_n(r_2)}-L_{r_1}|^{\gamma}\wedge 1\|_{L_2(\Omega)}+C\| |\phi^n_{\kappa_n(r_2)}-\phi^n_{r_1}|^\gamma\|_{L_2(\Omega)}\\
&\le C \| |L_{\kappa_n(r_2)}-L_{r_1}|^{(\gamma\wedge \frac\alpha2)-\eps\alpha}\|_{L_2(\Omega)}+C (\kappa_n(r_2)-r_1)^\gamma\\
&\le C (\kappa_n(r_2)-r_1)^{(\frac\gamma\alpha\wedge\frac12)-\eps}
\le C (r_2-r_1)^{(\frac\gamma\alpha\wedge\frac12)-\eps},
\end{align*} 
for $C=C(\|b\|_{\C^\gamma},\alpha,\gamma,\eps,d)>0$.
Here we used again that the function $\phi^n$ is Lipschitz. Substituting this into \eqref{phidif} and using that in this regime we have $r_2-\kappa_n(r_2)\le (r_2-r_1)\wedge \frac1n$, we get that for any $(r_1,r_2)\in\Delta_{[s,t]}$
\begin{equation*}
	\|(\phi^n_{r_2}-\phi^n_{\kappa_n(r_2)})-\E^{r_1}(\phi^n_{r_2}-\phi^n_{\kappa_n(r_2)})\|_{L_2(\Omega)}
	\le C(r_2-r_1)^{\frac12} n^{-\frac12-(\frac\gamma\alpha\wedge\frac12)+\eps},
\end{equation*}
for $C=C(\|b\|_{\C^\gamma},\alpha,\gamma,\eps,d)>0$. This implies by the definition of the seminorm $\dnew{\cdot}{\frac12}{2}{[s,t]}$ in \eqref{newnorm}
\begin{equation*}
\dnew{\phi^n-\phi^n_{\kappa_n(\cdot)}}{\frac12}{2}{[s,t]}\le 	C n^{-\frac12-(\frac\gamma\alpha\wedge\frac12)+\eps}
\end{equation*} 
for $C=C(\|b\|_{\C^\gamma},\alpha,\gamma,\eps,d)>0$.
This bound is much better than \eqref{bad} and justifies the introduction of the conditional seminorm. Now we substitute this bound together with \eqref{norm1num} into \eqref{err21} and obtain the bound for the second error term \eqref{errterm2}.
\end{proof}	

It remains to bound the term $\err_{n,22}$ defined in \eqref{eq:edlevy2}. Recall that we already obtained a simpler bound in the case $\phi \equiv 0$ in \cref{t:levydif}. In fact, exactly the same bound remains valid in the general case. The idea is to combine the techniques used in the proof of \cref{t:levydif} with a shifted stochastic sewing lemma. We present here only the final result and leave its proof as \cref{e:err3}.

\begin{lemma}\label{l:err3}
Let $\eps\in(0,\frac12)$.
There exists a constant $C={C(\|b\|_{\C^\gamma},\alpha,\gamma,\eps,d)>0}$ such that for any $n\in\N$, $(s,t)\in\Delta_{[0,1]}$ we have 
\begin{equation}\label{errterm3}
\|\err_{n,22}(s,t)\|_{L_2(\Omega)}\le 
 C n^{-\frac12-(\frac\gamma\alpha\wedge\frac12)+2\eps} (t-s)^{\frac12+\eps}.
\end{equation}
\end{lemma}

Now we can finally prove \cref{t:levy}. We just need to combine the error bounds, substitute them into the error decompositions \eqref{eq:edlevy} and \eqref{eq:edlevy2}, and apply the tactics discussed before: first we bound $[\phi-\phi^n]_{\C^{\frac12} L_2([s,t])}$ by half of  itself and terms of the form $C n^{-\rrr}$ for small time intervals $t-s$, and then we use \cref{l:ltg} to pass to arbitrary time intervals.

\begin{proof}[Proof of \cref{t:levy}]
We prove the theorem only in the case $m=2$ and without the supremum over time inside the expectation. In the general case, one must consider more sophisticated norms and use the John–Nirenberg inequality, in the spirit of the proof of \cref{t:levydif}. To keep the proof relatively simple, we do not treat this extension here and refer the interested reader to \cite{BDGLevy} for the full proof.

Fix $\eps>0$ such that $\frac12+\frac{\gamma-1}\alpha>\eps$. For $(s,t)\in\Delta_{[0,1]}$, we combine \eqref{eq:edlevy}, \eqref{eq:edlevy2} and the error bounds \eqref{errterm1}, \eqref{errterm2}, \eqref{errterm3} to derive 
\begin{align*}
&\|(\phi_t-\phi^n_t)-(\phi_s-\phi^n_s)\|_{L_2(\Omega)}\\
&\quad\le  C_0
(t-s)^{\frac12+\eps}( n^{-\frac12-(\frac\gamma\alpha\wedge\frac12)+2\eps} +\|\phi-\phi^n\|_{\C^0L_2([s,t])}+[\phi-\phi^n]_{\C^\frac12 L_2([s,t])})
\end{align*}
for $C_0=C_0(\|b\|_{\C^\gamma},\alpha,\gamma,\eps, d)$. 
Let now $0\le S< T\le 1$. Then dividing the above inequality by $(t-s)^{\frac12}$ and taking the supremum over all $s,t\in\Delta_{[S,T]}$ we get 
\begin{align*}
[\phi-\phi^n]_{\C^{\frac12}L_2([S,T])}&\le  C_0(T-S)^{\eps}( n^{-\frac12-(\frac\gamma\alpha\wedge\frac12)+2\eps} +\|\phi-\phi^n\|_{\C^0L_2([S,T])}+[\phi-\phi^n]_{\C^\frac12 L_2([S,T])})\\
&\le C_0(T-S)^{\eps}( n^{-\frac12-(\frac\gamma\alpha\wedge\frac12)+2\eps} +\|\phi_S-\phi^n_S\|_{L_2(\Omega)}+2[\phi-\phi^n]_{\C^\frac12 L_2([S,T])}).
\end{align*}
where in the last inequality we also used \eqref{norm2normgen}. Take now $\ell\in(0,1]$ small enough so that 
\begin{equation*}
2C_0\ell^{\eps}\le \frac12.
\end{equation*}	
Then we derive for any $(S,T)\in\Delta_{[0,1]}$ such that $T-S\le \ell$
\begin{equation*}
[\phi-\phi^n]_{\C^{\frac12}L_2([S,T])}\le  n^{-\frac12-(\frac\gamma\alpha\wedge\frac12)+2\eps} +\|\phi_S-\phi^n_S\|_{L_2(\Omega)}.
\end{equation*}
By \cref{l:ltg}(ii), there exists a constant $C=C(\|b\|_{\C^\gamma},\alpha,\gamma,\eps, d)$ such that 
\begin{equation*}
[\phi-\phi^n]_{\C^{\frac12}L_2([0,1])}\le  C n^{-\frac12-(\frac\gamma\alpha\wedge\frac12)+2\eps}.
\end{equation*}
Recalling that $\phi_0=\phi^n_0$, we get
\begin{equation*}
\sup_{t\in[0,1]}\|X_t-X^n_t\|_{L_2(\Omega)}=\|\phi-\phi^n\|_{\C^{0}L_2([0,1])} \le [\phi-\phi^n]_{\C^{\frac12}L_2([0,1])}\le C n^{-\frac12-(\frac\gamma\alpha\wedge\frac12)+2\eps},
\end{equation*}
which is the desired bound.
\end{proof}

\subsection{Exercises}\label{s:levyex}

First, let us obtain a conditional version of the John-Nirenberg inequality.

\begin{exercise} Let $m\ge1$, $0\le S\le T$ and let $\A\colon\Omega\times[S,T]\to\R^d$ be a continuous process adapted to the filtration $(\F_t)_{t\in[0,T]}$. Assume that there exists a constant $\Gamma>0$ such that for any $(s,t)\in\Delta_{[S,T]}$ one has 
\begin{equation*}
\E^s|\A_t-\A_s|\le \Gamma,\quad \text{a.s.}
\end{equation*}
\begin{enumerate}[(i)]
\item Let $A\in\F_S$. Show that the pair $(V^*\I_A, 2\Gamma\I_A)$, where $V^*$ is defined as in \eqref{vstar}, satisfies the good $\lambda$ inequality.

\hn Argue as in Step 2 of the proof of \cref{t:john}.
\item Deduce that  there exists a constant $C>0$ independent of $d,m,S,T$ such that for any 
$s\in [S,T]$, $A\in\F_S$ one has 
\begin{equation*}
\|\I_A\sup_{r\in[s,T]}|\A_r-\A_s|\|_{L_m(\Omega)}\le C m\Gamma\P(A)^{\frac1m}.
\end{equation*}	
	
\item Conclude that  there exists a constant $C>0$ independent of $d,m,S,T$ such that for any 
$s\in [S,T]$
\begin{equation*}
\E^S\bigl[\sup_{r\in[s,T]}|\A_r-\A_s|^m\bigr]\le (C m\Gamma)^m\quad \text{$\P$-a.s.}
\end{equation*}	

\hn Take $A$ to be the bad event where this inequality does not hold, that is $A:=\{\E^S[\sup_{r\in[s,T]}|\A_r-\A_s|^m]> (C m\Gamma)^m\}$, where  $C$ is as in part (ii), and apply part~(ii).
\end{enumerate}
\end{exercise}

Next, let us improve the stability bound in \cref{c:stab}. Note that inequality \eqref{stabineqfin} involves a constant depending on the $\C^\gamma$ norm of one drift, while the difference of the two drifts is measured in a much weaker norm (the $\C^0$ norm). The norm in which we measure this difference can be weakened further.

\begin{exercise}
Let $\alpha\in(1,2)$, $\gamma\in(1-\alpha/2,1]$, $m\ge2$.  Let $b\in\mathcal C^\gamma$. Let $X$ be a solution to SDE \eqref{mainSDElevy} on the  filtered probability space $(\Omega,\F,(\F_t),\P)$, on which $L$ is defined.
\begin{enumerate}[(i)]
\item  Let $\theta\in\R$. Find a condition on $\theta$ which implies the following:  there exists a constant $C=C(\|b\|_{\C^\gamma},\alpha,\gamma,\theta,d,m)>0$ such that for any bounded continuous function $f\colon \R^d\to\R$ and $(s,t)\in\Delta_{[0,1]}$, we have 
\begin{equation*}
\left\|\int_s^t f(X_r)\,dr\right\|_{L_m(\Omega)}\le C\|f\|_{\C^{\theta}}(t-s)^{\frac12}.
\end{equation*}

\hn Introduce $\phi:=X-L$ and apply the shifted stochastic sewing lemma to the germ
\begin{equation*}
A_{s,t}:=\E^{s-(t-s)} \int_s^t f(L_r+\E^{s-(t-s)}\phi_r)\,dr.
\end{equation*}

\item Let $\overline{b}\in\mathcal C^\gamma$ and let $\overline X$ be a solution to SDE \eqref{mainSDElevy} with the drift $\overline{b}$ in place of $b$ on the same filtered probability space $(\Omega,\F,(\F_t),\P)$. Let $M:=\max(\|b\|_{\C^\gamma},\|\overline{b}\|_{\C^\gamma})$. Show that for some constant 
$C=C(\alpha,\gamma,d,m,M)$ one has
\begin{equation*}
\|X-\overline{X}\|_{\C^0L_m([0,1])}\le C\|b-\overline b\|_{\C^{\gamma-1}}.
\end{equation*}

\hn Split the term $\|(X_t-X_s)-(\overline X_t-\overline X_s)\|_{L_m(\Omega)}$ into two parts as in Step 1 of the proof of \cref{c:stab}. Treat the first part exactly as in that proof, and use part (i) of the exercise to treat the second part.
\end{enumerate}
\end{exercise}

\cref{t:appradf} and \cref{t:levydif} provide convergence rates of the Euler scheme for SDEs with Hölder-continuous drifts. While in general these rates are known to be optimal \cite{ellinger2025optimal}, finer properties of the drift, such as Besov regularity, can lead to better rates. The goal of the first group of exercises is to show that for an indicator drift in $d=1$, the correct rate is $n^{-3/4}$ rather than $n^{-1/2}$ as suggested by \cref{t:appradf}.
To establish this, we need a number of new tools, which are also of independent interest. 

The  idea is to establish an analogue of \cref{t:appradf} for Besov drifts in three steps;  step 1 and step 2 are a modification of  \cite{DGL23}. First, we prove an analogue of \eqref{nummain} when $S$ is large enough compared with $T-S$ and $m=2$. Next, we extend the result to general $S$ and $T$ using the taming singularities technique. Finally, we apply the John-Nirenberg inequality to pass to arbitrary $m\ge2$.

\begin{exercise}[Taming singularities lemma, \cite{BFG,le2021taming}]\label{e:mbb2}
	Let $(E,d)$ be a complete metric space, $T>0$. Suppose there exist constants $\tau,\eta \ge 0$ with $\tau>\eta$, and $\Gamma>0$, such that a function $Y\colon(0,T]\to E$ satisfies
	\begin{equation*}
	d(Y_s,Y_t) \le \Gamma s^{-\eta}(t-s)^{\tau}\quad\text{for any $0< s\le t\le T$ such that $t-s\le s$.}
\end{equation*}
Show that $Y$  can be continuously extended at $0$ and there exists a constant $C=C(\eta,\tau,T)>0$ such that
\begin{equation}\label{rezult}
	d(Y_s,Y_t)\le C \Gamma(t-s)^{\tau- \eta}\quad\text{for any $0\le s\le t\le T$.}
\end{equation}	
\hn First, prove \eqref{rezult} for $s=0$ by splitting the interval $[0,t]$ into subintervals $[t_{i+1},t_i]$, where $t_0=t$ and $t_i \searrow 0$ is a very carefully chosen sequence with $t_{n+1}\ge t_{n}-t_{n+1}$, and apply the triangle inequality.
\end{exercise}

We need to develop basic tools to deal with Besov functions of positive regularity. We start with the definition.

\begin{definition}[see, e.g., \cite{Bogachev} or \cite{semyeon}] Let $p\in[1,\infty)$, $\gamma\in(0,1]$. We say that a function $f\in L_p(\R^d,\R)$ belongs to the Besov (called also Nikolskii–Besov) space $\B^\gamma_p(\R^d)$ if
\begin{equation*}
\|f\|_{\B^\gamma_p(\R^d)}:=\|f\|_{L_p(\R^d)}+\sup_{\substack{h\in\R^d\\0<|h|\le1}}\frac{\|f(h+\cdot)-f\|_{L_p(\R^d)}}{|h|^\gamma}<\infty.
\end{equation*}	
\end{definition}	
Thus, Besov spaces extend the notion of Hölder spaces to the case $p<\infty$. Let us see in which Besov spaces the indicator function lies.
\begin{exercise}\label{e:526}
Show that the indicator function $f(x):=\I_{\{x\in[0,1]\}}$, $x\in\R$, belongs to $\B^{1/p}_p(\R)$ for any $p\in[1,\infty)$.
\end{exercise}

Next, we need to provide heat kernel estimates for Besov functions, extending \cref{l:gb2}.
\begin{exercise}\label{ex:527}
Let $f\in\B^\gamma_p(\R^d,\R)$, $\gamma\in(0,1]$, $p\in[1,\infty)$, $0< s\le t$. Show that
\begin{equation*}
\|P_tf-P_s  f\|_{L_p(\R^d)}\le  C\| f\|_{\B^\gamma_p} s^{\frac\gamma2-1}(t-s),
\end{equation*} 	
for $C=C(\gamma,d,p)>0$.

\hn Argue as in the proofs of \cref{l:gb,l:gb2}. Use that by the definition of  $\B^\gamma_p$
\begin{equation*}
\Bigl\|\int_{\R^d}\Delta p_t(y) (f(\cdot-y)-f(\cdot))\,dy\Bigr\|_{L_p(\R^d)}\le\|f\|_{\B^\gamma_p}\int_{\R^d} |\Delta p_t(y)| |y|^\gamma \,dy.
\end{equation*}
\end{exercise}

\cref{ex:527} allows us to improve an important estimate in the proof of \cref{t:appradf}. We used a very crude bound there involving $\|f\|_{\C^\gamma}$:
\begin{equation*}
\|P_{t_2} f (W_s) -P_{t_1} f(W_s)\|_{L_m(\Omega)} \le \|P_{t_2} f-P_{t_1} f\|_{\C^0}\le C\|f\|_{\C^\gamma} (t_2-t_1) t_1^{\frac\gamma2-1}.
\end{equation*} 
for $0 \le t_1 \le t_2$, $s \in [0,1]$, see \eqref{seethis}. As the next exercise shows, $\|f\|_{\C^\gamma}$ can be replaced by  $\|f\|_{\B^\gamma_p}$ at the cost of introducing a singularity (of which we should not be afraid!; recall the taming singularities lemma, \cref{e:mbb2}).
\begin{exercise}\label{ex:528}
Let $f\in\B^\gamma_p(\R^d,\R)$, $\gamma\in(0,1]$, $p\in[1,\infty)$, $0< t_1 \le t_2$, $s\in(0,1]$. Show that 
\begin{equation*}
\|P_{t_2} f (W_s) -P_{t_1} f(W_s)\|_{L_p(\Omega)} \le \|P_{t_2} f-P_{t_1} f\|_{L_p(\R^d)}\|p_s\|_{L_\infty}^{1/p}\le C\|f\|_{\B^\gamma_p} (t_2-t_1) t_1^{\frac\gamma2-1}s^{-\frac{d}{2p}}.
\end{equation*} 	
for $C=C(\gamma,d,p)>0$.
\end{exercise}

\begin{exercise}\label{e:529}
In this exercise we perform Step 1 of the proof strategy discussed above and establish an analogue of \eqref{nummain} for Besov $f$ though for restricted values of $S,T$. Thus, let $\gamma\in(0,1]$, $p\in[2,\infty)$. Let $f\colon\R^d\to\R$ be a bounded measurable function, $f\in \B_p^\gamma(\R^d,\R)$, $n\in\N$, $\eps>0$.  Suppose that $0<S<T\le1$, $T-S\le S$. We would like to apply the SSL to the process 
\begin{equation*}
\A_t:=\int_0^t (f(W_r)-f(W_{\kappa_n(r)}))\,dr,\quad t\in[S,T]
\end{equation*}	
and the germ 
\begin{equation*}
	A_{s,t}:=\E^{s}\int_s^t f(W_r)-f(W_{\kappa_n(r)})\,dr,\quad (s,t)\in\Delta_{[S,T]}.
\end{equation*}
Fix $\eps\in(0,\frac12)$.
\begin{enumerate}[(i)]
\item Show that for any $(s,u,t)\in\Delta^3_{[S,T]}$, we have $\E^s\delta A_{s,u,t}=0$.
\item First, assume that $S\in(0,\frac2n]$. Then $T-S\le S\le \frac2n$.
\begin{enumerate}[label=(ii\alph*)]
\item Suppose that $\gamma\le\frac{d}p$. Show that 
\begin{equation*}
|A_{s,t}|\le 2 \|f\|_{\C^0}|t-s|\le 8\|f\|_{\C^0}|t-s|^{\frac12+\eps}n^{-\frac12-\frac\gamma2+\eps}S^{-\frac{d}{2p}},\quad (s,t)\in\Delta_{[S,T]}.
\end{equation*}

\item  Suppose that $\gamma>\frac{d}p$. Argue exactly as in \eqref{easynum} to derive for $(s,t)\in\Delta_{[S,T]}$
\begin{equation*}
\|A_{s,t}\|_{L_p(\Omega)}\le  C \| f\|_{\mathcal{C}^{\gamma-\frac{d}p}} n^{-\frac12-\frac\gamma2+\frac{d}{2p}+\eps}(t-s)^{\frac12+\eps}\le  C \| f\|_{\B^{\gamma}_p} n^{-\frac12-\frac\gamma2+\eps}(t-s)^{\frac12+\eps}S^{-\frac{d}{2p}}.
\end{equation*}

\hn Use without the proof that by the Besov embedding theorem (see, e.g., \cite[Proposition~2.39]{bahouri}) in this regime $f$ has a continuous representative $g$ and we have $\| g\|_{\mathcal{C}^{\gamma-\frac{d}p}}\le C \|g\|_{\B_p^\gamma}$ for $C=C(\gamma,d,p)$.

\end{enumerate}
\item Now we consider the case $S\ge\frac2n$. Denote  $s': =\kappa_n(s) + \frac2n$, that is $s'$ is the second gridpoint to the right of $s$. For $(s,t)\in\Delta_{[S,T]}$ we split the term $A_{s,t}$
\begin{align*}
&\|A_{s,t}\|_{L_p(\Omega)}\nn\\
&\quad\le \int_s^{s'\wedge t}\|
	  f(W_r)-f(W_{\kappa_n(r)})\|_{L_p(\Omega)}\, dr+
	\Bigl\|\int_{s'\wedge t}^t
	\E^s \big( f(W_r)-f(W_{\kappa_n(r)})\big) dr\Bigr\|_{L_p(\Omega)}\nn\\\
	&\quad =:I_1+I_2 
\end{align*}
\begin{enumerate}[label=(iii\alph*)]
\item\label{iiia} Show that for any $r\in(0,1)$
\begin{equation*}
\|
f(W_r)-f(W_{\kappa_n(r)})\|_{L_p(\Omega)}\le C\|f\|_{\B^\gamma_p} ({\kappa_n(r)})^{-\frac{d}{2p}}(r-\kappa_n(r))^{\frac\gamma2}, 
\end{equation*}
for $C=C(\gamma,d,p)>0$.

\hn Use the definition of the space $\B^\gamma_p$ and independence of $W_{\kappa_n(r)}$ and $W_r-W_{\kappa_n(r)}$.

\item Deduce from \ref{iiia} that for $C=C(\gamma,d,p)>0$
\begin{equation*}
I_1\le  C \| f\|_{\B^\gamma_p}((s'\wedge t)-s) n^{-\frac\gamma2}S^{-\frac{d}{2p}} \le C \| f\|_{\B^\gamma_p} n^{-\frac12-\frac\gamma2+\eps}(t-s)^{\frac12+\eps}S^{-\frac{d}{2p}}.
\end{equation*}

\hn Use that $\kappa_n(r)\ge S-\frac1n\ge S/2$.

\item 
Show that either $s'\wedge t=t$  and then $I_2\equiv0$, or  $s'\wedge t=s'$ and then $t\ge s'\ge s+\frac1n$ and for $C=C(\gamma,\eps,d,p)>0$
\begin{equation*}
I_2\le 	C\|f\|_{\B^\gamma_p}|t-s|^{\frac12+\eps}n^{-\frac12-\frac\gamma2+2\eps} S^{-\frac{d}{2p}}.
\end{equation*}

\hn Argue as in \eqref{four34} and use \cref{ex:528} when needed.

\item Conclude that for $C=C(\gamma,\eps,d,p)>0$
\begin{equation*}
\|A_{s,t}\|_{L_p(\Omega)}\le 	C\|f\|_{\B^\gamma_p}|t-s|^{\frac12+\eps}n^{-\frac12-\frac\gamma2+2\eps} S^{-\frac{d}{2p}},\quad (s,t)\in\Delta_{[S,T]}.
\end{equation*}
\end{enumerate}

\item\label{part4ee} Deduce from parts (i)--(iii) that all the conditions of the SSL are satisfied and therefore for any $0<S<T\le1$, $T-S\le S$ one has
\begin{equation*}
\Bigl\|\int_S^T (f(W_r)-f(W_{\kappa_n(r)}))\,dr\Bigr\|_{L_p(\Omega)}\le  C  (\| f\|_{\C^0}+\| f\|_{\B_p^\gamma})(T-S)^{\frac12+\eps}n^{-\frac12-\frac\gamma2+2\eps}S^{-\frac{d}{2p}}.
\end{equation*}
for $C=C(\gamma,\eps,d,p)>0$ independent of $S$, $T$, $n$.
\end{enumerate}
\end{exercise}

\begin{exercise}\label{e:530}
We proceed to Step 2 of the proof. Let $\gamma\in(0,1]$, $p\in[2,\infty)$, $\eps>0$, $f\colon\R^d\to\R$ be a bounded measurable function. Assume additionally that $d/p\le 1$. Use the taming singularities lemma (\cref{e:mbb2}), to derive from \cref{e:529}\ref{part4ee} that there exists a constant $C=C(\gamma,\eps,d,p)>0$ such that for any $(S,T)\in\Delta_{[0,1]}$, $n\in\N$ we have
\begin{equation*}
	\Bigl\|\int_S^T (f(W_r)-f(W_{\kappa_n(r)}))\,dr\Bigr\|_{L_p(\Omega)}\le  C  (\| f\|_{\C^0}+\| f\|_{\B_p^\gamma})(T-S)^{\frac12-\frac{d}{2p}+\eps}n^{-\frac12-\frac\gamma2+2\eps}.
\end{equation*}
\end{exercise}

\begin{exercise}\label{e:531}
Finally, we can now get the main estimate: a bound on the convergence rate of integral functionals of Brownian motion for Besov functions. Let $\gamma\in(0,1]$, $p\in[2,\infty)$, $\eps>0$, $d/p\le 1$, $n\in\N$, $f\colon\R^d\to\R$ be a bounded measurable function.
\begin{enumerate}[(i)]
\item\label{p1jnend} Let $s\in[0,1]$ be a gridpoint, that is,  $\kappa_n(s)=s$. Deduce from \cref{e:530} that for any $t\in[s,1]$ one has 
\begin{equation}\label{eqjnbesov}
\E^s\Bigl|\int_s^t  (f(W_r)-f(W_{\kappa_n(r)}))\,dr\Bigr|\le C  (\| f\|_{\C^0}+\| f\|_{\B_p^\gamma})
n^{-\frac12-\frac\gamma2+2\eps}
\end{equation}
for $C=C(\gamma,\eps,d,p)>0$ independent of $n$.

\hn Use the independence of the increments of $W$ and that $\kappa_n(r)-s=\kappa_n(r-s)$ since $s$ is a gridpoint.
\item\label{threefinal} Use part \ref{p1jnend} to show that \eqref{eqjnbesov} holds for arbitrary $(s,t)\in\Delta_{[0,1]}$.

\hn If $s$ is not a gridpoint, split the integral $\int_s^t=\int_s^{s'\wedge t}+\int_{s'\wedge t}^t$, where $s'=\kappa_n(s)+\frac1n$ and argue as in \eqref{wheretoargue}.
\item Use the John-Nirenberg inequality (\cref{t:john}) to deduce from part \ref{threefinal} that for any $m\ge1$
 there exists a constant $C=C(\gamma,\eps,d,p,m)>0$ such that for any $n\in\N$ one has 
\begin{equation*}
\Bigl\|\sup_{t\in[0,1]}\Bigl|\int_0^t (f(W_r)-f(W_{\kappa_n(r)}))\,dr\Bigr|\,\Bigr\|_{L_m(\Omega)}\le   C  (\| f\|_{\C^0}+\| f\|_{\B_p^\gamma})n^{-\frac12-\frac\gamma2+2\eps}
\end{equation*}
\end{enumerate}
\end{exercise}

\begin{exercise} Finally, let us give an illustration of the just-obtained result.
Let $d=1$. Show that for any $m\ge1$, $\eps>0$, there exists a constant $C=C(\eps,m)>0$ such that for any $n\in\N$ one has
\begin{equation}\label{indicator}
\Bigl\|\sup_{t\in[0,1]}\Bigl|\int_0^t (\I(W_r\in[0,1])-\I(W_{\kappa_n(r)}\in[0,1]))\,dr\Bigr|\,\Bigr\|_{L_m(\Omega)}\le   C  n^{-\frac34+\eps}.
\end{equation}
\hn Use \cref{e:526}.
\end{exercise}

Thus, while \cref{t:appradf} suggested that the rate of convergence for the indicator functions is of order $C n^{-\frac12}$, we see that more sophisticated considerations allow us to improve the rate to $C n^{-3/4}$, which is known to be optimal \cite{AltGu}. Indeed, the numerical results provided in \cref{fig:occ}  are consistent with the $n^{-3/4}$ rate.

\begin{figure}[htbp]
  \centering
  \includegraphics[width=0.65\textwidth]{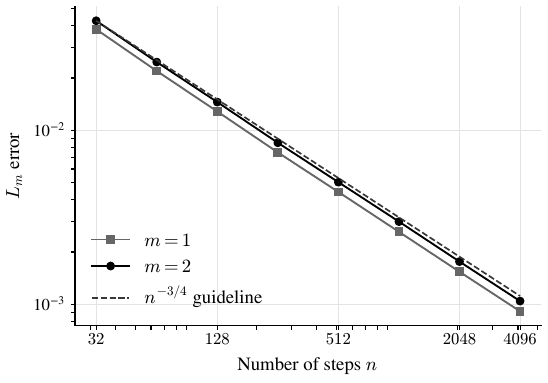}
  \caption{ Empirical estimates of the left-hand side of \eqref{indicator}
for $m=1,2$ and $n=2^5,2^6,\ldots,2^{12}$. The dashed line is an
$n^{-3/4}$ guideline.}
  \label{fig:occ}
\end{figure}

\begin{exercise}\label{levyboundex} In this exercise we prove \cref{l:fractional} in the regime $\alpha\in(1,2)$. Let $f\in \C^\gamma(\R^d,\R)$, $\gamma\in[0,1]$. 
\begin{enumerate}[(i)]
	\item Show that for any $x,y\in\R^d$ we have
	\begin{equation*}
	|\E (f(L_1+x)-f(L_1+y))|\le C \|f\|_{\C^0}|x-y|\le C \|f\|_{\C^\gamma}|x-y|.
	\end{equation*}
	You can use without proof that $|\nabla \pa_1|\in L_1(\R^d)$, see, e.g.,  \cite[Theorem~2.1(a,b)]{Knopova}.
\item Use part (i) and the scaling $\law(L_t)=\law(t^{\frac1\alpha}L_1)$, $t>0$, to derive \eqref{ineqlevy}. 
\item Use that the process $L_t$ has zero mean to show that for any $g\colon\R^d\to\R$, $x\in\R^d$, $t>0$ one has 
\begin{equation*}
P^\alpha_t g(x)-g(x)=\int_{\R^d}p^\alpha_t(y)\bigl(g(x+y)-g(x)-y^\t\nabla g(x)\bigr)\,dy.
\end{equation*}	
\item Let $\beta\in(1,\alpha)$. Deduce from part (iii) that there exists $C=C(\alpha,\beta,d)$ such that
for any $g\in\C^\beta(\R^d,\R)$ we have 
\begin{equation*}
\|P^\alpha_t g-g\|_{\C^0}\le C t^{\frac\beta\alpha}\|g\|_{\C^\beta}
\end{equation*}	
\item Let $0 \le s \le t$. Apply part (iv) to the function $g:=P^\alpha_s f$ and use \eqref{ineqlevy} to derive \eqref{ineqlevy2} up to the loss of arbitrary $\eps$ in the exponent of $(t-s)$. 
\end{enumerate}
\end{exercise}

\begin{exercise}\label{e:err3} 
In this exercise, we prove \cref{l:err3}. The main idea is to apply the shifted stochastic sewing lemma (\cref{lem:shiftedmodifiedSSL}) to the process
\begin{equation*}
\A_t:=\int_0^t (b(L_r+\phi^n_{\kappa_n(r)})-b(L_{\kappa_n(r)}+\phi^n_{\kappa_n(r)}))\,dr,\quad t\in[S,T].
\end{equation*}	
Motivated by the discussion in \cref{s:levywp}, we take as a germ 
\begin{equation*}
	A_{s,t}:=\E^{s-(t-s)}\int_s^t b(L_r+\E^{s-(t-s)}\phi^n_{\kappa_n(r)})-b(L_{\kappa_n(r)}+\E^{s-(t-s)}\phi^n_{\kappa_n(r)})\,dr,\quad (s,t)\in\ms{[S,T]}.
\end{equation*}
To simplify the notation, denote $\psi_r:=\phi^n_{\kappa_n(r)}$, $r\in[0,1]$.

\begin{enumerate}[(i)]
	\item Show that the germ satisfies
	\begin{equation*}
	\|A_{s,t}\|_{L_2(\Omega)}\le C 	 n^{-\frac12-(\frac\gamma\alpha\wedge\frac12)+2\eps} (t-s)^{\frac12+\eps},\quad (s,t)\in\ms{[S,T]}
	\end{equation*} 
	for $C=C(\|b\|_{\C^\gamma},\alpha,\gamma,\eps,d)>0$.
	
	\hn Argue very similarly to the proof of Step 1 of \cref{t:levydif} and use \eqref{ineqlevy2}.
	
	\item Let $(s,t)\in\ms{[S,T]}$, $u=(t+s)/2$. As in the proof of \cref{L:32} denote  $s_1:=s-(t-s)$, $s_2:=s-(u-s)$, $s_3:=s$, $s_4:=u$, $s_5:=t$, recall \cref{fig:arrow-ticks}. Show that 
	\begin{align*}
	\E^{s-(t-s)}\delta A_{s,u,t}&=\E^{s_1}\E^{s_2}\int_{s_3}^{s_4} b(L_r+\E^{s_1}\psi_r)-b(L_{\kappa_n(r)}+\E^{s_1}\psi_r)
	\nn\\
	&\hspace{3cm}-b(L_r+\E^{s_2}\psi_r)+b(L_{\kappa_n(r)}+\E^{s_2}\psi_r)\,dr
	\nn\\
	&\quad+\E^{s_1}\E^{s_3}\int_{s_4}^{s_5} b(L_r+\E^{s_1}\psi_r)-b(L_{\kappa_n(r)}+\E^{s_1}\psi_r)
	\nn\\
	&\hspace{3cm}-b(L_r+\E^{s_3}\psi_r)+b(L_{\kappa_n(r)}+\E^{s_3}\psi_r)\,dr\\
	&=:I_1+I_2	
	\end{align*}
	\item\label{part3ll} Show that if $t-s\ge \frac4n$, then $\kappa_n(r)-s_2>0$ and we have 
	\begin{align*}
		I_1&=\E^{s_1}\int_{s_3}^{s_4}\big(\Pa_{r-s_2}-\Pa_{\kappa_n(r)-s_2}\big)b(L_{s_2}+\E^{s_1}\psi_r\big)
	-\big(\Pa_{r-s_2}-\Pa_{\kappa_n(r)-s_2}\big)b(L_{s_2}+\E^{s_2}\psi_r\big)\,dr.
	\end{align*}
	\item Use the identity from part \ref{part3ll} (and a similar identity for $I_2$) to derive that in case $t-s\ge4/n$ we have 
	\begin{equation*}
		\|\E^{s-(t-s)}\delta A_{s,u,t}\|_{L_2(\Omega)}\le  C 	 n^{-\frac12-(\frac\gamma\alpha\wedge\frac12)+2\eps} |t-s|^{\frac32+\frac{\gamma-1}\alpha-\eps}
	\end{equation*}
	for $C=C(\|b\|_{\C^\gamma},\alpha,\gamma,\eps,d)>0$ 
	
	\hn Use \cref{l:fractional} and stochastic regularity result \cref{l:stochreg}.
	
	\item\label{partvll}  Show that for any $r\in[0,1]$, $\psi_r$ is $\F_{(\kappa_n(r)-\frac1n)\vee0}$ measurable and therefore if $r\in[\kappa_n(s_1),\kappa_n(s_1)+\frac2n)$, then the integrand in $I_1$ is zero.
	
	\item Use part \ref{partvll} to show that in the case $t-s\le4/n$ we have 
	\begin{equation*}
		\|\E^{s-(t-s)}\delta A_{s,u,t}\|_{L_2(\Omega)}\le C 	 n^{-\frac12-(\frac\gamma\alpha\wedge\frac12)+2\eps} |t-s|^{1+\eps}.
	\end{equation*}
	\item Conclude. Show that all conditions of the shifted stochastic sewing lemma are satisfied, and therefore the bound \eqref{errterm3} holds.
\end{enumerate}

\end{exercise}



\begingroup
\footnotesize
\bibliographystyle{alpha1}
\bibliography{biblio}
\endgroup
\end{document}